\documentclass[12pt]{article}

\usepackage{amssymb}
\usepackage{amsmath}
\usepackage{amsthm}
\usepackage{mathtools}
\usepackage{enumerate}
\usepackage{mathrsfs}
\usepackage[colorlinks=true,allcolors=black]{hyperref}
\usepackage{url}
\usepackage{float}

\usepackage{setspace}
\usepackage{tikz}
\tikzset{every picture/.style={line width=0.75pt}}
\usepackage[labelfont=bf]{caption}

\NewDocumentCommand{\sff}{}{\mathrm{I\!I}}

\usepackage{chngcntr}

\allowdisplaybreaks

\newtheorem{theorem}{Theorem}
\newtheorem{proposition}{Proposition}
\newtheorem{definition}{Definition}
\newtheorem{remark}{Remark}
\newtheorem{lemma}{Lemma}
\newtheorem{corollary}{Corollary}

\newcommand{\mres}{\mathbin{\vrule height 1.6ex depth 0pt width
0.13ex\vrule height 0.13ex depth 0pt width 1.3ex}}

\newcommand{\Addresses}{{
  \bigskip
  \footnotesize

KMS:
\textsc{Department of Mathematics, Johns Hopkins University, 404 Krieger Hall, 3400 N. Charles Street, Baltimore, MD 21218, USA}\par\nopagebreak
  \textit{Email address:} \texttt{kmarsh34@jh.edu}

\medskip

GN:
\textsc{Department of Mathematics, University of Rochester, 915 Hylan Building, Rochester, NY 14620, USA}\par\nopagebreak
  \textit{Email address:} \texttt{g.niu@rochester.edu}

\medskip

DP:
\textsc{Department of Mathematics, Huxley Building, South Kensington Campus, Imperial College London, London, SW7 2AZ, UK}\par\nopagebreak
  \textit{Email address:} \texttt{d.parise@imperial.ac.uk}
}}

\begin{document}

\title{\large \textbf{GENERIC REGULARITY OF ISOPERIMETRIC REGIONS IN DIMENSION EIGHT}}
\author{\small KOBE MARSHALL-STEVENS, GONGPING NIU, \& DAVIDE PARISE}
\date{\vspace{-5ex}}
\maketitle

\begin{abstract}
\noindent We establish generic regularity results for isoperimetric regions in closed Riemannian manifolds of dimension eight. In particular, we show that every isoperimetric region has a smooth nondegenerate boundary for a generic choice of smooth metric and enclosed volume, or for a fixed enclosed volume and a generic choice of smooth metric.
\end{abstract}

\tableofcontents

\section{Introduction}\label{sec: introduction}

Isoperimetric regions arise as minimisers of boundary area for a fixed enclosed volume, with sharp regularity theory, as established in \cite{GMT83} and \cite{M03}, guaranteeing that the boundary of such a region is a smooth hypersurface away from a closed singular set of codimension seven (see Subsection \ref{subsec: notation} for a precise definition). In closed Riemannian manifolds of dimension eight, this singular set consists of at most finitely many isolated points, with explicit singular examples having been constructed in \cite{N24B}. One may thus hope that, under some assumption on the choice of ambient metric and enclosed volume, all isoperimetric regions have smooth boundary in closed Riemannian manifolds of dimension eight; for example under a genericity assumption. We show that this is indeed the case:

\begin{theorem}\label{thm: generic regularity metrics and volumes}
    In a closed manifold of dimension eight, every isoperimetric region has smooth nondegenerate boundary for a generic choice of smooth Riemannian metric and enclosed volume.
\end{theorem}

By \textit{generic} in the above statement, and indeed throughout this work, we mean in the Baire category sense; namely, we show that the conclusions of the above statement hold on a countable intersection of open and dense sets in the space of metric volume pairs. Moreover, \textit{nondegenerate} refers to the triviality of the kernel of the linearised mean curvature operator (see Subsection \ref{subsec: twisted Jacobi fields}). We also obtain the following result for any fixed choice of enclosed volume:

\begin{theorem}\label{thm: generic regularity fixed volume}
    In a closed manifold of dimension eight and for a fixed enclosed volume, every isoperimetric region has smooth nondegenerate boundary for a generic choice of smooth Riemannian metric.
\end{theorem}

In the course of the proof of Theorems \ref{thm: generic regularity metrics and volumes} and \ref{thm: generic regularity fixed volume}, we will in fact obtain several ancillary results, the precise statements of which are contained in Theorems \ref{thm: implies theorem 1} and \ref{thm: implies theorem 2}. In particular, we mention here that one can phrase the statements of both results above for the class of $C^{k,\alpha}$ metrics (for $k \geq 4$ and $\alpha \in (0,1)$), from which the case for smooth metrics follows, and moreover restrict to the conformal class of a given metric and obtain analogous genericity results. We also note that it follows, by renormalisation, from the proof of both Theorem \ref{thm: generic regularity metrics and volumes} and \ref{thm: generic regularity fixed volume} that one can restrict to the class of unit volume Riemannian metrics in each of their statements.

\begin{remark}
    One cannot hope to fix a metric and vary enclosed volumes to obtain an analogous generic regularity result as above. To see this, we observe that, for sufficiently large $R>0$, the closed eight-dimensional Riemannian manifold, $(M(R),g_R)$, constructed in \cite[Theorem 3.3]{N24B} is such that every isoperimetric region with enclosed volume in an open interval (centred at half of the volume of this manifold) has a unique isoperimetric region whose boundary contains exactly two isolated singular points; this follows implicitly from the proof of \cite[Theorem 4.1]{N24B}. 
\end{remark}

As a direct consequence of Theorem \ref{thm: generic regularity metrics and volumes} we are able to extend a result on the generic Riemannian quantitative isoperimetric inequality, \cite[Theorem 1.2]{CES22}, previously known to hold in dimension seven, to dimension eight. Precisely, with the notation as introduced in Subsection \ref{subsec: notation}, we obtain: 

\begin{corollary} \label{corollary: stability isoperimetric inequality}
    Given a closed manifold of dimension eight, there exists a generic subset, $\mathcal{U} \subset \mathcal{G}^{k, \alpha} \times \mathbb{R}$, of Riemannian metrics and enclosed volumes with the following property. If $(g, t) \in \mathcal{U}$, then there is a constant $C > 0$, depending on $g$ and $t$, such that if $E \in \mathcal{C}(M)$ with $\mathrm{Vol}_g(E) = t$ then
    \begin{equation*}
        \mathrm{Per}_g(E) - \mathrm{I}_g(t) \geq C \alpha_g(E)^2,
    \end{equation*}
    where 
    \begin{equation*} \label{equation: isoperimetric profile}
        \mathrm{I}_g(t) = \inf_{F \in \mathcal{C}(M)}\{\mathrm{Per}_g(F) \, | \, \mathrm{Vol}_g(F) = t \} \quad \text{and} \quad  \alpha_g(E) = \inf\{ \mathrm{Vol}_g( E \Delta \Omega) \, | \, \Omega \in \mathcal{I}(g, t)\},
    \end{equation*}
    are the isoperimetric profile and the Fraenkel asymmetry respectively. 
\end{corollary}

This follows since \cite[(1.6)]{CES22} holds in all dimensions (for smooth metrics), from which the proof of  \cite[Theorem 1.2]{CES22} goes through verbatim by replacing the generic set of metric volume pairs there by the one provided by our Theorem \ref{thm: generic regularity metrics and volumes}. Moreover, as a consequence of Theorem \ref{thm: generic regularity fixed volume} we also extend \cite[Corollary 5.4]{CES22}, the fixed enclosed volume and generic metric analogue of \cite[Theorem 1.2]{CES22}, to closed manifolds of dimension eight.

\bigskip

It would be of interest to know to which other situations the results of Theorem \ref{thm: generic regularity metrics and volumes} and \ref{thm: generic regularity fixed volume} may be applied in order to extend effective applications of isoperimetric regions with smooth boundary to closed Riemannian manifolds of dimension eight. For instance, similarly to Corollary \ref{corollary: stability isoperimetric inequality}, we expect the main results of this work to be applicable to the generic quantitative stability problem for the Cheeger energy introduced in \cite{C70}; the (non-generic) quantitative stability for this problem has been addressed in \cite{d23} for closed Riemannian manifolds of dimension at most seven. Moreover, it would be interesting to investigate the case of non-compact manifolds with finite volume, as well as complete manifolds under various curvature assumptions (e.g.~nonnegative Ricci curvature, Euclidean volume growth, quadratic curvature decay, asymptotic flatness, etc.) in which isoperimetric regions exist; in this direction, we refer the reader to \cite{APS25} and references therein. 

\bigskip

As isoperimetric regions have constant mean curvature on smooth portions of their boundary, the present work is related to previous results on the generic regularity of constant mean curvature, and in particular minimal (i.e.~zero mean curvature), hypersurfaces, which we now summarise.

\bigskip

The generic existence of a smooth area-minimising minimal hypersurface in each non-zero homology class of a closed Riemannian manifold of dimension eight was established in \cite{S93}, utilising local metric perturbations based on the foliation result of \cite{HS85}, which in turn established the analogous generic existence result for smooth seven-dimensional Plateau minimisers. Similar local metric perturbations were then utilised in \cite{CLS22} to establish the generic existence of smooth minimal hypersurfaces in closed Riemannian manifolds with positive Ricci curvature in dimension eight. More recent works, \cite{CMS23}, \cite{CMS24}, and \cite{CMSW25}, develop further local metric perturbation techniques in order to show the generic existence of smooth area-minimising minimal hypersurfaces, in each non-zero homology classes and for Plateau solutions, up to ambient dimension eleven.

\bigskip

In \cite{LW20}, by exploiting the global metric perturbations for isolated singularities of minimal hypersurfaces  developed in \cite{W20}, it was shown that every eight-dimensional closed manifold equipped with a generic metric (with no ambient curvature assumption) admits a smooth minimal hypersurface. Building on these global metric perturbations and the results of \cite{E24}, which in particular show that seven-dimensional minimal hypersurfaces with bounded mass and index belong to a finite collection of diffeomorphism types, it was shown in \cite{LW25} that for a generic choice of metric, every embedded locally stable minimal hypersurface, with sufficiently small singular set, is smooth in each eight-dimensional Riemannian manifold. We also refer to recent work in \cite{CLW25}, which obtains generic regularity results allowing for strongly isolated singularities (those whose tangent cone is of multiplicity one with singular set consisting of one point) of minimal submanifolds (potentially of high codimension) to be perturbed away by index theoretic methods; it would be of interest to know if these methods could be adapted to the setting of constant mean curvature surfaces with strongly isolated singularities.

\bigskip

Local metric perturbations were employed in \cite{BM25}, exploiting the constant mean curvature analogue of the foliation result of \cite{HS85} developed in \cite{BL25}, in order to remove strongly isolated singularities of constant mean curvature hypersurfaces which locally minimise the naturally associated area-type functional. In particular, this was utilised to show, for each $\lambda \in \mathbb{R}$, the generic existence of a smooth closed embedded hypersurface of constant mean curvature $\lambda$ in closed Riemannian manifolds with positive Ricci curvature in dimension eight. As shown in \cite[Section 2.2.2]{M24}, applying the same method to the boundary of an isoperimetric region in a closed Riemannian manifold of dimension eight produces an entirely smooth hypersurface of constant mean curvature. However, since isoperimetry is a global property, it is not clear if the resulting smooth hypersurface bounds an isoperimetric region.

\subsection{Strategy}

Since local metric perturbation techniques may not preserve isoperimetry, the strategy of the present work instead establishes Theorems \ref{thm: generic regularity metrics and volumes} and \ref{thm: generic regularity fixed volume} via global metric perturbations and a decomposition of the space of isoperimetric regions in closed Riemannian manifolds of dimension eight, building upon the techniques introduced in \cite{LW25} and \cite{E24}. In the discussion that follows, we restrict our attention to closed Riemannian manifolds of dimension eight and outline the strategy taken in the present work.

\bigskip

We first develop a procedure that allows us to perturb away isolated singularities in the boundary of a given isoperimetric region subject to an appropriate assumption on its boundary hypersurface. To this end we develop (in Subsection \ref{subsec: twisted Jacobi fields}) the relevant theory for the linearised mean curvature operator, which we refer to as the \textit{twisted Jacobi operator}, on the boundary hypersurfaces of isoperimetric regions in the presence of isolated singularities. The kernel of this operator, which we refer to as the space of twisted Jacobi fields, can then be seen as a direct generalisation of the notion of twisted Jacobi fields introduced for smooth embedded constant mean curvature hypersurfaces in \cite{BdE88}. We then show (in Theorem \ref{thm: induced twisted jacobi fields}) that a convergent sequence of isoperimetric regions (with the same enclosed volume) induces a non-zero twisted Jacobi field on the limiting isoperimetric region. Under the assumption that every twisted Jacobi field has a sufficiently fast growth rate (as specified in Definition \ref{def: slow growth and semi-nondegenerate}) near at least one isolated singularity, these induced Jacobi fields allow us to perturb away such a point by a conformal metric change (this is carried out in Corollary \ref{prop: perturbation of singularities with fast growth}, building on the open and dense set of conformal perturbation functions defined in Proposition \ref{prop: perturbation functions open and dense}).

\bigskip

Isoperimetric regions whose twisted Jacobi fields posses this sufficiently fast growth rate near at least one isolated singularity are referred to as \textit{semi-nondegenerate}; a property which implies the usual notion of nondegeneracy if the boundary hypersurface is smooth. An analogous notion of semi-nondegeneracy was first introduced (and shown to be a generic property) for locally stable minimal hypersurfaces with isolated singularities in \cite{LW20} and \cite{LW25}, building on the analysis in \cite{W20}, which enabled such points to be perturbed away. We show (in Theorem \ref{thm: generic semi-nondegeneracy}) that  semi-nondegeneracy is a generic property for isoperimetric regions in dimension eight; generic nondegeneracy results for both smooth minimal and constant mean curvature hypersurfaces were established in the foundational work of \cite{W91} (see also \cite[Section 5]{CES22} for isoperimetric regions specifically). To achieve this, we establish a local Sard--Smale type theorem in appropriately defined \textit{pseudo-neighbourhoods} (see Subsection \ref{subsec: pseudo-neighbourhoods and three compactness lemmas}) of a given triple of data (consisting of a  metric, enclosed volume, and isoperimetric region associated to this metric volume pair); an analogous notion of pseudo-neighbourhoods were first introduced for metric and minimal hypersurface pairs in \cite{LW25}. Such a result shows that, for a given triple of data as above, we can produce (in Lemma \ref{lemma: generic pairs for neighbourhoods}) an open and dense set of metric volume pairs for which all associated isoperimetric regions in a small enough pseudo-neighbourhood of this triple are semi-nondegenerate.

\bigskip

In order to achieve a global result we introduce (in Theorem \ref{thm: cone decomposition}) a notion of cone decomposition for general almost minimisers of perimeter analogous to the cone decomposition for minimal hypersurfaces introduced in \cite{E24}. Since isoperimetric regions are themselves almost minimisers of perimeter, by using this cone decomposition result we are able to decompose the space of isoperimetric regions in a closed Riemannian manifold of dimension eight, showing that their boundary hypersurfaces belong to a countable collection of diffeomorphism types; a similar description of the space of almost minimisers of perimeter was established in \cite[Theorem 5.5]{ESV24}. This ultimately allows us (in Lemma \ref{lem: covering of triples by pseudo neighbourhoods}) to cover the space of triples by a countable collection of pseudo-neighbourhoods, to each of which we associate an open and dense set of metric volume pairs produced by the local Sard--Smale result mentioned above. Intersecting over the countable collection of these open and dense sets of metric volume pairs, we thus conclude that semi-nondegeneracy is a generic property for isoperimetric regions in dimension eight. 

\bigskip

Having a perturbation procedure for isolated singularities and the genericity of semi-nondegeneracy in hand, ideally one would inductively reduce the potential number of isolated singularities that arise along converging sequences of isoperimetric regions until the resulting isoperimetric regions are smooth. However, one issue in this approach is that, while we ensure that converging isoperimetric regions are smooth near a given singular point, this process does not necessarily strictly decrease the total number of singular points arising along the sequence (for instance, one cannot preclude the possibility that two singular points converge to one singular point, with a necessarily higher density, in the limit).

\bigskip

To overcome this we define a notion of \textit{singular capacity} (see Definition \ref{def: singular capacity for isoperimetric regions}) for isoperimetric regions, analogous to those introduced for minimal hypersurfaces in both \cite{LW20} and \cite{LW25}. This singular capacity accounts for the potential number of singularities that can arise along sequences of convergent isoperimetric regions, is upper semi-continuous with respect to this convergence, and is finite for a given metric volume pair (these latter two properties are shown in Proposition \ref{prop: upper semi-continuity of singular capacity}). These properties allow us to show (in Proposition \ref{prop: reducing singular capacity}) that, in the generic set of metric volume pairs for which every isoperimetric region is semi-nondegenerate, we can iteratively reduce the maximum value of the singular capacity associated to a metric volume pair. Repeated iterations of these perturbations produce (in Theorems \ref{thm: implies theorem 1} and \ref{thm: implies theorem 2}) a generic set of metric volume pairs for which the maximum value of the singular capacity is always zero; thus every isoperimetric region associated to such a metric volume pair has entirely smooth boundary. This line of reasoning directly establishes both Theorems \ref{thm: generic regularity metrics and volumes} and \ref{thm: generic regularity fixed volume}.

\bigskip

We expect that much of the theory for isoperimetric regions developed in the present work applies directly to the class of multiplicity one, locally stable (or indeed finite index), embedded constant mean curvature hypersurfaces with at most finitely many isolated singularities. The difficulty in concluding a generic regularity result for a more general class of constant mean curvatures hypersurfaces (which allow for touching spheres/cylinders for example) is that, while a robust regularity and compactness theory for the class of quasi-embedded locally stable constant mean curvature hypersurfaces was developed in \cite{BW18}, one would need to account in our arguments for both the presence of non-embedded points as well as higher multiplicity. Neither occurs for isoperimetric regions. Furthermore, in comparison to the class of minimal hypersurfaces considered in \cite{LW25} and \cite{E24}, we do not need to account for the presence of index in the boundary hypersurfaces in our arguments. Each of these above facts affords us several simplifications when compared to their approaches since, for the most part, we can rely solely on the theory of Caccioppoli sets. In particular, we emphasise that isoperimetric regions possess a relatively straightforward compactness theory (recorded below in Lemma \ref{lemma: compactness of isoperimetric regions}) when compared to the compactness results for minimal hypersurfaces with bounded index (which is detailed in \cite[Appendix G]{LW25} and the references therein).

\bigskip

However, several additional difficulties are present when studying the twisted Jacobi operator associated to an isoperimetric region when compared to the Jacobi operator for a minimal hypersurface. Since isoperimetric regions, and more generally constant mean curvature hypersurfaces, are only stationary with respect to volume preserving variations, their associated twisted Jacobi fields are necessarily of integral zero on the boundary. For Jacobi fields on minimal hypersurfaces no such integral constraint is required. This difference becomes particularly apparent whenever one wants to ensure integral control when taking limits of twisted Jacobi fields for isoperimetric regions whose boundaries contain isolated singularities, since one needs to rule out any integral concentration at these points in the limit; in the smooth case there is no such issue, since one has graphical convergence of the boundary hypersurfaces. The two main methods introduced here to overcome this integral constraint issue, which are exploited repeatedly throughout our arguments, involve constructions of appropriate integral zero test functions as well as analysing the local volume change around singular points, guaranteeing integral non-concentration along convergent sequences of isoperimetric regions that induce twisted Jacobi fields. 

\subsection{Structure}

We now proceed as follows. Section \ref{sec: twisted jacobi fields} records several preliminary results on isoperimetric regions that will be used frequently throughout this work, and introduces the notion of twisted Jacobi fields for isoperimetric regions with isolated singularities. Section \ref{sec: asymptotic rates and metric perturbations} studies the growth rates of isoperimetric regions near isolated singularities, shows that converging isoperimetric regions induce twisted Jacobi fields, and then uses this result to show that, under a semi-nondegenerate assumption, we can perturb away isolated singularities. Section \ref{sec: bumpy metric volume pairs} is devoted to showing that semi-nondegeneracy is a generic property for isoperimetric regions in dimension eight. Section \ref{sec: singular capacity for isoperimetric regions} introduces the singular capacity for isoperimetric regions, and shows that we can combine the results of Sections \ref{sec: asymptotic rates and metric perturbations} and \ref{sec: bumpy metric volume pairs} to reduce the singular capacity in the generic set of metric volume for which every isoperimetric region is semi-nondegenerate. Section \ref{sec: proof of theorems 1 & 2} is devoted to the proofs of Theorems \ref{thm: generic regularity metrics and volumes} and \ref{thm: generic regularity fixed volume}. Appendix \ref{sec: appendix A} records several technical lemmas on the growth rates for graphical hypersurfaces over cones and for the mean curvature operator for hypersurfaces that are graphical over one another. Appendix \ref{sec: appendix B} contains both the cone decomposition for almost minimisers of perimeter in dimension eight, as well as the relevant definitions that allow for the covering of the space of triples to be carried out in Section \ref{sec: bumpy metric volume pairs}.    

\section{Notation, isoperimetric regions, and twisted Jacobi fields}\label{sec: twisted jacobi fields}

In this section we will establish notation, record some preliminaries on isoperimetric regions, and introduce the notion of twisted Jacobi fields (homogeneous solutions to the linearised mean curvature operator on boundaries of isoperimetric regions) in the presence of isolated singularities.

\subsection{Notation}\label{subsec: notation}

We now collect some notation and definitions that will be used throughout this work:

\begin{itemize}
    \item We let $(M,g)$ be an $8$-dimensional closed (i.e.~compact with empty boundary) Riemannian manifold. Without loss of generality we will implicitly assume $M$ is connected with unit volume. For $k \geq 4$ and $\alpha \in (0,1)$ we denote by $\mathcal{G}^{k,\alpha}$ the set of all $C^{k,\alpha}$ Riemannian metrics on $M$ (note that this set is a Banach manifold since $k$ is finite). For $g \in \mathcal{G}^{k,\alpha}$ we denote the conformal class of $g$ (amongst $C^{k,\alpha}$ metrics) by $$[g] = \{(1+f)g \in \mathcal{G}^{k,\alpha} \, | \, f \in C^{k,\alpha}(M)\}.$$ We write $\mathrm{Vol}_g(E)$ for the volume of a measurable $E \subset M$, $\int_E u \, dV_g$ for the integral of some integrable function $u$, and $L^p_g(E)$ for the space of $p \in [1,\infty]$ integrable functions on $E$, all with respect to the metric $g$. We denote by $\mathrm{dist}_g$ the distance function on $M$, $B^g_r(p)$ the open geodesic ball in $M$ of radius $r > 0$ centred at $p$, and $A^g(p;s, r) = B^g_r(p)\setminus \overline{B^g_s(p)}$ for annuli, all with respect to the metric $g$. 
    
    \item A measurable set $E \subset M$ is a \textbf{Caccioppoli set} if the indicator function of $E$ is of bounded variation, or equivalently if
    \begin{equation*}
    \mathrm{Per}_g(E) = \sup \left\{\int_E \mathrm{div}_g X \: \bigg| \: X \in \Gamma(TM), \Vert X\Vert_\infty \leq 1 \right\} < \infty,
    \end{equation*}
    where $\mathrm{div}_g$ is the divergence with respect to the metric $g$, $\Gamma(TM)$ is the set of vector fields on $M$ and $\Vert \cdot \Vert_\infty$ denotes the supremum norm. We write $\mathcal{C}(S)$ for the set of Caccioppoli sets in a set $S$, and note that $\mathcal{C}(S)$ is independent of the choice of metric on $S$. We say that $\Omega \in \mathcal{C}(M)$ is an \textbf{isoperimetric region} of enclosed volume $t \in \mathbb{R}$ if
    \begin{equation*}
        \mathrm{Per}_g(\Omega ) = \inf_{E \in \mathcal{C}(M)}\{\mathrm{Per}_g(E) \: | \: \mathrm{Vol}_g(E) = t\},
    \end{equation*}
    where $\mathrm{Per}_g(E)$ is the perimeter of $E \in \mathcal{C}(M)$ with respect to the metric $g$ and the right-hand side is the \textbf{isoperimetric profile}, $\mathrm{I}_g(t)$, as introduced in the statement of Corollary \ref{corollary: stability isoperimetric inequality}. The existence of isoperimetric regions of a given enclosed volume is then guaranteed by the direct method of the calculus of variations (e.g.~see \cite[Section 12.5]{M12}). We denote by $\mathcal{I}(g,t)$ the set of all isoperimetric regions in $(M,g)$ of enclosed volume $t \in \mathbb{R}$. Notice that if $t < 0$ or $t > |M|_g$ then there are no Caccioppoli sets of enclosed volume $t$, and hence no isoperimetric regions, meaning that any statements concerning them will be vacuously true.

    \item We will often make use of the notion of a varifold; a reference for the notation and definitions may be found in \cite{S83A}. In particular, for a varifold $V$, a point $p$ and a radius $r > 0$, we will denote $\theta_{V, g}(p, r) = \omega_7^{-1}r^{-7} \Vert V \Vert_g (B^g_r(p))$ for the density ratios, where $\omega_7$ is the $7$-dimensional Lebesgue measure of the unit ball in $\mathbb{R}^7$, and $\theta_{V, g}(p) = \lim_{r \rightarrow 0^+}\theta_{V, g}(p,r)$ for the density. Furthermore, for a varifold $V$, in the definition of density, we denoted by $\Vert V \Vert_g$ the weight measure with respect to the metric $g$; more generally, we may omit metric dependence when working in Euclidean space or when it is clear from context. Finally, for $x \in \mathbb{R}^n$ and $r > 0$ we let $\eta_{x,r}(y) = \frac{y - x}{r}$ for each $y \in \mathbb{R}^n$ and denote by $(\eta_{x,r})_{\#}V$ the image of a varifold $V$ under $\eta_{x,r}$.

    \item Let $\mathcal{T}^{k,\alpha}$ denote the set of \textbf{triples}, $(g,t,\Omega)$, and $\mathcal{P}^{k,\alpha}(t)$ denote the set of \textbf{pairs}, $(g,\Omega)$, where $g \in \mathcal{G}^{k,\alpha}$, $t \in \mathbb{R}$, and $\Omega \in \mathcal{I}(g,t)$; it will often be convenient to view $\mathcal{P}^{k,\alpha}(t) \subset \mathcal{T}^{k,\alpha}$ for a fixed $t \in \mathbb{R}$. We endow $\mathcal{T}^{k,\alpha}$ with the topology induced by the $C^{k-1,\alpha}$ topology in the first factor, the standard topology on $\mathbb{R}$ in the second factor, and the $L^1_g$ topology in the last factor; this restricts to a topology on $\mathcal{P}^{k,\alpha}(t)$ for each $t \in \mathbb{R}$.

    \item Given an isoperimetric region $\Omega$ as above, we denote by $\Sigma \subset \partial \Omega$ (and $\Sigma_j \subset \partial \Omega_j$ for sequences, tildes, and primes etc.), with $\overline{\Sigma} = \partial \Omega$, the $C^2$, two-sided, embedded hypersurface of constant mean curvature, and the closed \textbf{singular set}, $\mathrm{Sing}(\Sigma) = \overline{\Sigma} \setminus \Sigma$, consisting of isolated points with multiplicity one tangent cones with isolated singularity at the origin; by \cite{GMT83} and \cite{M03}, this is always the case for boundaries of isoperimetric regions in ambient dimension $8$. We say that an isoperimetric region (or by abuse of terminology its boundary) is \textbf{regular} if its boundary has empty singular set. We will often also write $|\Sigma|_g$ for the multiplicity one varifold associated to $\Sigma$, $\int_\Sigma u \, dA_g$ to denote the integral of some integrable function on $\Sigma$, and denote by $\nu_{\Sigma,g}$ the outward pointing unit normal to $\Sigma$ (all with respect to the metric $g$). We will occasionally abuse notation and write $|\Sigma|_g = \mathrm{Per}_g(\Sigma)$ for notational convenience (in particular when considering the averaged integrals as introduced in Subsection \ref{subsec: twisted Jacobi fields}).
    
    \item We let $\mathrm{inj}(M,g)$ denote the injectivity radius of $M$ with respect to the metric $g$ and for each $\Omega \in \mathcal{I}(g,t)$ we denote 
    $$\mathrm{inj}(\Sigma,g) = \min\{\mathrm{inj}(M,g), \{\mathrm{dist_g}(p,\widetilde{p}) \, | \, p, \widetilde{p} \in \mathrm{Sing}(\Sigma) \text{ with } p \neq \widetilde{p}\}\}.$$
    For each $r > 0$ we also denote
    $$B^g_r(\mathrm{Sing}(\Sigma)) = \Sigma \cap \left(\cup_{p \in \mathrm{Sing}(\Sigma)} B^g_r(p)\right) =  \{x \in \Sigma \, | \, \mathrm{dist}_g(x,p) < r \text{ for some } p \in \mathrm{Sing}(\Sigma)\}.$$
    We write $\exp_{g,x}$ for the exponential map with respect to the metric $g$ based at $x \in M$ and for a measurable set $E \subset \Sigma$ and a function $u : E \rightarrow \mathbb{R}$ denote the graph of $u$ over $E$ with respect to the metric $g$ by 
    $$\mathrm{graph}_{E}^g(u) = \{ \exp_x^g(u(x)\nu_{\Sigma,g}(x)) \, | \, x \in E\}.$$

    \item For working in $\mathbb{R}^{8}$  we denote by $g_\mathrm{eucl}$ the standard Euclidean metric, $\mathbb{B}_r(p)$ the open ball of radius $r$ centred at $p$, $\mathbb{B}_r = \mathbb{B}_r(0)$, $\mathbb{S}_r(p)$ the sphere of radius $r$ centred at $p$, $\mathbb{S} = \mathbb{S}_1(0)$, $\mathbb{A}(p;s, r) = \mathbb{B}_r(p)\setminus \overline{\mathbb{B}_s(p)}$, and $\mathbb{A}(s,r) = \mathbb{A}(0;s,r)$. Furthermore, $d_\mathcal{H}(\cdot, \cdot)$ will denote the Hausdorff distance in $\mathbb{R}^8$ with respect to $g_\mathrm{eucl}$. We will occasionally omit metric dependence from notation that involves $g_\mathrm{eucl}$, in particular for the multiplicity one varifolds associated to hypercones.
    
    \item Suppose that $\mathbf{C} \subset \mathbb{R}^8$ is a \textbf{minimal hypercone}; i.e.~$\mathrm{Sing}(\mathbf{C}) \subset \{0\}$, we then denote the smooth minimal hypersurface $S = \mathbf{C} \cap \mathbb{S} \subset \mathbb{R}^8$ as its \textbf{link}, $\mathbf{B}_r = \mathbf{C} \cap \mathbb{B}_r$, and $\mathbf{A}(s,r)\ =\mathbf{B}_r\setminus \overline{\mathbf{B}_s}$. By parameterising $\mathbf{C}$ in radial coordinates, $(r,\omega) \in (0,\infty) \times S$, we decompose (as in \cite{S68}) the Jacobi operator, $L_{\mathbf{C}}$, of $\mathbf{C}$ as
    $$L_{\mathbf{C}} = \partial^2_r + \left(\frac{n-1}{r}\right)\partial_r + \frac{1}{r^2}(\Delta_S + |\mathrm{I\!I}_S|^2),$$
    where $\mathrm{I\!I}_S$ is the second fundamental form of $S$ in $\mathbb{S}^7$. We let $\mu_1 < \mu_2 \leq \mu_3 \leq \ldots \nearrow + \infty$ be the eigenvalues of $- (\Delta_S + \vert \mathrm{I\!I}_S \vert^2)$, and $\varphi_1, \varphi_2, \ldots$ be the corresponding $L^2(S)$-orthonormal eigenfunctions where $\varphi_1 > 0$. By \cite{S68}$, \mathbf{C}$ is \textbf{stable} if and only if $\mu_1\geq -\frac{(n-2)^2}{4}$ and, as in \cite{CHS84}, we say that $\mathbf{C}$ is \textbf{strictly stable} if $\mu_1 > -\frac{(n-2)^2}{4}$.
    By \cite{CHS84}, a general \textbf{Jacobi field}, $v \in C_\mathrm{loc}^\infty(\mathbf{C})$, on $\mathbf{C}$ (i.e.~a solution to $L_\mathbf{C}v = 0$) is given by a linear combination of homogeneous Jacobi fields
    \begin{equation*}
        v(r, \omega) = \sum_{j \geq 1} ( v_j^+(r) + v_j^-(r)) \varphi_j(\omega), 
    \end{equation*}
    where for each $j \geq 1$ we write
    \begin{equation*}
        v_j^+(r) = c_j^+ \cdot r^{\gamma_j^+} \qquad \text{and} \qquad v_j^- = \begin{cases}
    c_j^- \cdot r^{\gamma_j^-}, & \text{if $\mu_j > - \frac{(n - 2)^2}{4}$}, \\
    c_j^- \cdot r^{\gamma_j^-} \log(r) & \text{if $\mu_j = -  \frac{(n - 2)^2}{4}$}, 
        \end{cases}
    \end{equation*}
    for some constants $c_j^\pm \in \mathbb{R}$, and
    \begin{equation*}
        \gamma_j^\pm = \gamma_j^\pm(\mathbf{C}) = - \frac{n - 2}{2} \pm \sqrt{\mu_j + \frac{(n - 2)^2}{4}}.
    \end{equation*}
    The collection of the exponents $\gamma_j^\pm$ is called the \textbf{asymptotic spectrum} of $\mathbf{C}$, which we denote by $\Gamma(\mathbf{C})$. For every $\Lambda > 0$, we let
    \begin{equation*}
        \mathscr{C}_\Lambda = \{\text{stable minimal hypercones, $\mathbf{C} \subset \mathbb{R}^8$, with $\Vert \mathbf{C} \Vert(\mathbb{B}_1) \leq \Lambda$}\}, 
    \end{equation*}
    while we denote simply by $\mathscr{C}$ the collection of stable minimal hypercones without density bound. Given $\mathbf{C} \in \mathscr{C}$, we will denote by $\mathscr{C}(\mathbf{C})$ the collection of cones $\mathbf{C}^\prime \in \mathscr{C}$ satisfying $\theta_{|\mathbf{C}^\prime|}(0) = \theta_{|\mathbf{C}_i|}(0)$, and for which there exists a $C^2$-diffeomorphism $\phi : \partial B_1 \rightarrow \partial B_1$ with $\phi(\mathbf{C} \cap \partial B_1) = \mathbf{C}^\prime \cap \partial B_1$. By considering the varifold metric,
    \begin{equation*}
        \mathbf{F}(V, W) = \sup\{\Vert V \Vert(f) - \Vert W \Vert(f) \, | \, f \in C^1(M), \vert f \vert \leq 1, \vert Df \vert \leq 1 \}, 
    \end{equation*}
    defined for integer rectifiable $7$-varifolds $V$, $\mathscr{C}_\Lambda$ is compact under the $\mathbf{F}$-metric for every $\Lambda > 0$; hence 
    \begin{equation*}
        \gamma_{\mathrm{gap}}(\Lambda) = \inf\{\gamma_{2}^{+}(\mathbf{C}) - \gamma_{1}^{+}(\mathbf{C}) \, | \, \mathbf{C} \in \mathscr{C}_\Lambda \} > 0. 
    \end{equation*}
    The asymptotic spectrum is continuous under varifold convergence, i.e.~if $\mathbf{F}(\vert \mathbf{C}_i \vert, \vert \mathbf{C}_\infty \vert) \rightarrow 0$ for $\{\mathbf{C}_j\}_{j \geq 1} \subset \mathscr{C}$, then $\mu_j(\mathbf{C}_i) \rightarrow \mu_j(\mathbf{C}_\infty)$ for each $j \geq 1$ as $k \rightarrow \infty$. By \cite[Theorem 5.1]{E24}, given a sequence of stable minimal hypercones $\{\mathbf{C}_i\}_{i \geq 1} \subset \mathscr{C}_\Lambda$ for some $\Lambda > 0$ there is a subsequence (not relabelled) and a stable minimal hypercone $\mathbf{C}$ such that the links $\mathbf{C}_{i} \cap \partial B_1$ converge to $\mathbf{C} \cap \partial B_1$ smoothly with multiplicity one as $i \rightarrow \infty$, and so that $\theta_{|\mathbf{C}_{i}|}(0) = \theta_{|\mathbf{C}|}(0)$ for all $i \geq 1$ sufficiently large. Moreover, by \cite{S83B} (see also \cite[Theorem 5.1]{E24}) the densities of stable minimal hypercones in $(\mathbb{R}^8, g_\mathrm{eucl})$ are discrete:
    \begin{equation} \label{eqn: discrete densities stable cones}
        \{\theta_{|\mathbf{C}|}(0) \, | \, \mathbf{C} \subset \mathbb{R}^8 \text{ is a stable minimal hypercone}\} = \{1 = \theta_0 < \theta_1 < \theta_2 < \ldots \nearrow + \infty\}. 
    \end{equation}
    
    \item We define the regularity scale, $r_S(x)$, depending on $M,g,$ and $\Sigma$, at a point $x \in \Sigma$ to be the supremum among all $r \in (0, \mathrm{inj}(\Sigma,g)/2)$ such that both of the following properties hold:
        \begin{itemize}
            \item $r^{2} \Vert \mathrm{Rm}_g\Vert_{C^0(B_r^g(x))}+r^3\Vert \nabla \mathrm{Rm}_g\Vert_{C^0(B_r^g(x))} \le 1/10$, where $\mathrm{Rm}_g$ denotes the Riemann curvature tensor of $M$ with respect to the metric $g$.
            \item After pulling back by $\exp^g_x$ to $T_xM$ we have
                $$\frac1r (\exp^g_x)^{-1}(\Sigma) \cap \mathbb{B}_1 = \mathrm{graph}^{g_\mathrm{eucl}}_L (u) \cap \mathbb{B}_1,$$
          for some hyperplane $L \subset T_xM$ and $u \in C^3(L)$ with $\Vert u\Vert_{C^3} \le 1/10$.
        \end{itemize}

    \item For $\phi\in C_{\mathrm{loc}}^k(\Sigma)$ we define the following norm on a measurable subset $E \subset \Sigma$ to be
        $$\Vert \phi\Vert_{C^k_*(E)}= \sup_{x \in E} \sum_{j=0}^k r_S(x)^{j-1}|\nabla^j \phi(x)|;$$
    which is invariant under the scaling of $\phi$ to $\lambda \phi$ and $g$ to $\lambda^2g$ for $\lambda > 0$.
     For $f\in C^k(M)$ and $x \in \Sigma$, we define the pointwise norm
        $$[f]_{x,g,C^k_*} =  \sum_{j=0}^k r_S(x)^j\sup_{B^g_{r_S}(x)} |\nabla^j f|(x);$$
    which is invariant fixing $f$ and scaling $g$ to $\lambda^2 g$ for $\lambda > 0$.
    
    \item By \cite[Theorem 5]{S83B} and \cite[Theorem 6.3]{E24}, for any $p\in \mathrm{Sing}(\Sigma),$ we may express $\Sigma$ in \textbf{conical coordinates} near $p$ over its tangent cone $\mathbf{C}_p\Sigma$; precisely, for any $\epsilon >0$, there exists $\phi\in C^2(\mathbf{C}_p\Sigma)$ and $r_p(\epsilon)>0$ such that both $\Vert \phi\Vert_{C_*^2(\mathbb{B}_{r_p})}< r_p(\epsilon)$ and
    $$ \mathrm{graph}^{g_\mathrm{eucl}}_{\mathbf{C}_p\Sigma} (\phi)\cap {\mathbb{B}_{r_p(\epsilon)}} = (\exp_p^g)^{-1}\left(\Sigma\ \cap B^g_{r_p(\epsilon)}(p)\right).$$ 

\end{itemize}

\subsection{Preliminaries on isoperimetric regions}\label{subsec: preliminaries}

We record here several results and notions for isoperimetric regions that will be used frequently throughout this work:

\begin{lemma}[Compactness of isoperimetric regions]\label{lemma: compactness of isoperimetric regions}
For $t\in \mathbb{R}$, if $g_j \to g$ in $C^{k-1,\alpha}$ and $t_j\to t$, then for each sequence $\{\Omega_j\}_{j \geq 1} \subset \mathcal{I}(g_j, t_j)$ there exists $\Omega \in \mathcal{I}(g, t)$ such that, up to a subsequence (not relabelled), we have:
\begin{enumerate}
    \item We have $\Omega_j \to \Omega$ in $L_g^1(M)$ and $\mathrm{Per}_{g_j} (\Omega_j) \rightarrow \mathrm{Per}_g (\Omega)$.
    
    \item We have $|\Sigma_j|_{g_j} \rightarrow |\Sigma|_g$ as varifolds.
\end{enumerate} 
\end{lemma}
\begin{proof}
    Since $g_j \rightarrow g$ in $C^{k-1,\alpha}$  and $t_j \rightarrow t$ we have $\sup_{j\geq 1} \mathrm{Per}_g(\Omega_j)<\infty$ so that, by the compactness of sets of finite perimeters \cite[Corollary 12.27]{M12}, there exits a Caccioppoli set $\Omega \in \mathcal{C}(M)$ such that, up to a subsequence (not relabelled), we have $\Omega_j \rightarrow \Omega$ in $L_g^1(M)$; as $|\Omega_j|_{g_j} =t_j$ for all $j$, we also have $|\Omega|_g = t$. By the lower semicontinuity of perimeter,  \cite[Proposition 12.15]{M12}, we have 
    \[
            \liminf_{j\to \infty} \mathrm{Per}_{g_j} (\Omega_j) \geq \mathrm{Per}_g (\Omega). 
    \]
    Assuming $\limsup_{j\to \infty} \mathrm{Per}_{g_j} (\Omega_j) > \mathrm{Per}_g (\Omega)$, define $\delta = \limsup_{j\to \infty} \mathrm{Per}_{g_j} (\Omega_j) -\mathrm{Per}_g (\Omega)> 0$. By \cite[Lemma 17.21]{M12}, for every $j$, there exists $\widetilde \Omega_j \in \mathcal{C}(M)$ with $\vert \widetilde{\Omega}_j \vert = t_j$, such that $|\mathrm{Per}_{g_j} (\widetilde \Omega_j)-\mathrm{Per}_{g_j} (\Omega)|\leq C|\varepsilon_j|$, for a fixed constant $C > 0$ and with $\varepsilon_j = |\Omega|_{g_j} -t$. As $g_j\to g$ in $C^{k-1,\alpha}$ and $t_j \rightarrow t$, we see that $\varepsilon_j \to 0$, which implies that for sufficiently large $j \geq 1$ we have
    \[ 
    \mathrm{Per}_{g_j} (\Omega_j)= I(g_j,t_j) \leq \mathrm{Per}_{g_j} (\widetilde \Omega_j)\leq \mathrm{Per}_{g_j}(\Omega) + C|\varepsilon_j| \leq  \mathrm{Per}_{g} (\Omega) +C|\varepsilon_j| + \delta/4.
    \]
    As $\varepsilon_j \rightarrow 0$, for sufficiently large $j \geq 1$ we have
    \[
    \limsup_{j\to \infty} \mathrm{Per}_{g_j} (\Omega_j) \leq \mathrm{Per}_g(\Omega) + \delta/2,
    \] 
    contradicting the choice of $\delta$. In particular, we have $\lim_{j\to \infty}\mathrm{Per}_{g_j} (\Omega_j) = \mathrm{Per}_g (\Omega)$. 
    
    \bigskip
    
    Suppose now there exists $\widetilde \Omega \in \mathcal{C}(M)$ with $\vert \widetilde{\Omega} \vert_g = t$ but such that $\mathrm{Per}_g(\widetilde\Omega) < \mathrm{Per}_g(\Omega)$,  and set $\delta = \mathrm{Per}_g(\Omega)-\mathrm{Per}_g(\widetilde \Omega) =\lim_{j\to \infty}\mathrm{Per}_{g_j}(\Omega_j)-\mathrm{Per}_g(\widetilde \Omega) > 0$. By \cite[Lemma 17.21]{M12} again there exists $\widetilde \Omega_j \in \mathcal{C}(M)$ with $\vert \Omega_j \vert = t_j$, such that $|\mathrm{Per}_{g_j}(\widetilde \Omega)-\mathrm{Per}_{g_j}(\widetilde \Omega_j)|< C|\varepsilon_j|$, for a fixed constant $C > 0$ and with $\varepsilon_j = |\widetilde \Omega|_{g_j} -t$.
    In particular, for sufficiently large $j \geq 1$ we have
         \begin{align*}
             \mathrm{Per}_{g_j}(\widetilde \Omega_j) & \leq \mathrm{Per}_{g_j}(\widetilde \Omega) + C|\varepsilon_j| \leq \mathrm{Per}_{g}(\widetilde \Omega) +C|\varepsilon_j| + \delta/10 < \mathrm{Per}_{g}(\Omega) - \delta+ C|\varepsilon_j| + \delta/10 < \mathrm{Per}_{g_j}(\Omega_j).
         \end{align*} 
         For sufficiently large $j \geq 1$, this contradicts the assumption that $\Omega_j \in \mathcal{I}(g_j,t_j)$ and hence $$\mathrm{Per}_g(\Omega) = \lim_{j\to \infty} \inf_{E \in \mathcal{C}(M)}\{\mathrm{Per}_{g_j}(E) \: | \: \mathrm{Vol}_{g_j}(E) = t_j\} = \inf_{E \in \mathcal{C}(M)}\{\mathrm{Per}_g(E) \: | \: \mathrm{Vol}_g(E) = t\} ;$$ this shows that $\Omega \in \mathcal{I}(g,t)$ and thus concludes part 1 of the lemma. Part 2 follows directly from \cite[Proposition A.1]{DT13}.
    \end{proof} 

\begin{remark} \label{remark: apply Allard}
    In view of the varifold convergence guaranteed by Lemma \ref{lemma: compactness of isoperimetric regions} part 2, whenever we have a convergent sequence of isoperimetric regions, $\Omega_j \rightarrow \Omega$, throughout this work we will often simply say that we ``apply Allard's theorem'' (e.g.~as stated in \cite[Chapter 5]{S83A}) in a neighbourhood of a regular point of the limit, i.e.~$p \in \Sigma \setminus \mathrm{Sing}(\Sigma)$, to obtain the conclusion that the $\Omega_j$ are eventually regular in this neighbourhood for sufficiently large $j \geq 1$.
\end{remark}

\begin{lemma}[Fixed volume mean curvature bound] \label{Lemma: mean curvature bound} 
    For $\delta > 0$ and $g \in \mathcal{G}^{k,\alpha}$ there exists an open neighbourhood, $U$, of $g \in \mathcal{G}^{k,\alpha}$ and a constant $C > 0$, depending on $\delta, g$, and $U$, such that if $\Omega \in \mathcal{I}(\bar{g}, t)$ with $\bar{g} \in U$ and $t \in (\delta, \vert M \vert_{\bar{g}} - \delta)$, then $\Sigma$ has constant mean curvature at most $C$. 
\end{lemma}
\begin{proof}
    The mean curvature bound for a fixed metric was established in \cite[Lemma 4.3]{N24B} (extending the result in the smooth case in \cite[Lemma C.1]{CES22}). Assuming our desired statement fails, then there exist $g_j \rightarrow g$ in $C^{k-1,\alpha}$ and a sequence $\{\Omega_j \}_{j \geq 1} \subset \mathcal{I}(g_j, t_j)$ with $t_j \in (\delta, \vert M \vert_{g_j} - \delta)$, such that the mean curvatures, $H_j$, of the $\Sigma_j$ diverge to infinity. By Lemma \ref{lemma: compactness of isoperimetric regions} there exists $\Omega \in \mathcal{I}(g, t)$ such that, up to a subsequence (not relabelled), $\Omega_j \rightarrow \Omega \in \mathcal{I}(g,t)$, where $t \in (\delta/2, \vert M \vert_g - \delta/2)$; but then by \cite[Lemma 4.3]{N24B}, $\Sigma$ must have bounded mean curvature, say $H$. Allard's theorem applied to the regular part of $\Sigma$ then ensures that, since $|\Sigma_j|_{g_j} \rightarrow |\Sigma|_g$ as varifolds by Lemma \ref{lemma: compactness of isoperimetric regions}, we have $H_j \rightarrow H$; a contradiction.
\end{proof}

\begin{remark}\label{rem: Morgan-Johnson}
    From Lemma \ref{Lemma: mean curvature bound} we see that for each $\Omega \in \mathcal{I}(g,t)$, $\Sigma$ has at most finitely many connected components. Precisely, by \cite[Theorem 2.2]{MJ00}, there is a $\delta > 0$ such that if $t \in (0, \delta] \cup [\vert M \vert_g - \delta, \vert M \vert_g)$, then $\Sigma$ is a perturbation of a coordinate sphere (i.e.~a ``nearly round sphere'' in their notation), and hence connected. Then, by applying Lemma \ref{Lemma: mean curvature bound} for this choice of $\delta$ along with the monotonicity formula (which holds since the mean curvature of $\Sigma$ is then bounded by Lemma \ref{Lemma: mean curvature bound}), we see that $\Sigma$ has at most finitely many connected components for all $t$, see \cite[Lemma 2.4]{CES22}. Reasoning similarly to proof of Lemma \ref{Lemma: mean curvature bound}, up to potentially decreasing $\delta > 0$ above, the same result holds in an open neighbourhood of $g$ in $\mathcal{G}^{k,\alpha}$.
\end{remark}

It will be convenient at different stages of this work to view isoperimetric regions as almost minimisers of perimeter or as volume constrained minimisers. Specifically, the former is better suited to decompose the space of triples, $\mathcal{T}^{k,\alpha}$, while the latter is useful in order to define the notion of a singular capacity. We now briefly recall definition and some basic properties of each notion, referring to \cite[Chapter 21]{M12} for further details:   

\begin{definition}
    Given an open set $U \subset \mathbb{R}^8$, constants $\Lambda \geq 0$, and $r_0 > 0$, a set of locally finite perimeter $E$ in $\mathbb{R}^8$ is called a $(\Lambda, r_0)$-perimeter minimiser or \textbf{almost minimiser} in $U$ provided 
\begin{equation*}
   \mathrm{Per}(E; \mathbb{B}_r(x)) \leq \mathrm{Per}(F; \mathbb{B}_r(x)) + \Lambda \vert E \Delta F \vert, 
\end{equation*}
whenever $E \Delta F \subset \subset \mathbb{B}_r(x) \cap U$, and $r < r_0$. 
\end{definition}

We note that the compactness and regularity theory for almost-minimisers of perimeter in Euclidean space $\mathbb{R}^{n}$ and in a closed Riemannian manifold are comparable. More precisely, if $g$ is a $C^2$-Riemannian metric on $B_1(0) \subset \mathbb{R}^{n}$, satisfying $\Vert g - g_{\mathrm{eucl}} \Vert_{C^2} \leq \delta$, then, provided $\delta > 0$ is sufficiently small, we have 
\begin{equation*}
    (1 - \delta) \vert x - y \vert \leq \mathrm{dist}_g(x, y) \leq (1 + \delta) \vert x - y \vert \qquad \text{and} \qquad \mathbb{B}_{(1 - \delta)r}(x) \subset B_{r}^{g}(x) \subset \mathbb{B}_{(1 + \delta)r}(x). 
\end{equation*}
Furthermore, for every $\mathbb{B}_r(x) \subset \mathbb{B}_{1 - 10\delta}(0)$, there is a normal change of coordinates $\phi_x : B_{(1 + \delta)r}^{g}(x) \rightarrow B_{(1 + \delta)r}^{g}(x)$ in which the metric satisfies $\Vert (\phi_x^\ast g)(z)  - g_{\mathrm{eucl}} \Vert \leq C \delta \vert x - z \vert^2. $
Furthermore, if $E$ is $(\Lambda, r_0)$-perimeter minimising in $(B_1, g)$, then $E$ is $(\Lambda + C\delta, (1 - \delta)r_0)$-perimeter minimising in $(\mathbb{B}_{1 - 10\delta}, g_{\mathrm{eucl}})$, where the constant $C > 0$ depends only on $n, \Lambda,$ and $ \mathrm{Per}(E; B_1(0))$. Compactness results for almost minimisers can be found in \cite[Section 21.5]{M12}, while their regularity results, with identical conclusions on the singular set as for isoperimetric regions, is contained in \cite[Part 3]{M12}. 

\begin{definition}
Given an open set $U \subset \mathbb{R}^8$, a set of locally finite perimeter $E$ in $\mathbb{R}^8$ is a \textbf{volume-constrained minimiser} in $U$ if 
\[
\mathrm{Per} (E; U) \le \mathrm{Per} (F; U), 
\]
whenever $\mathrm{Vol}(E \cap U) = \mathrm{Vol}(F \cap U)$ and $E \Delta F \subset \subset U$. 
\end{definition}

Arguing as in \cite[Example 16.13]{M12}, we see that isoperimetric regions are volume constrained minimisers, while \cite[Example 21.3]{M12} shows that volume constrained minimisers are $(\Lambda, r_0)$-almost minimisers for $\Lambda \geq 0$ and $r_0 > 0$, depending only on $E$ and $U$. Alternatively, \cite[Example 21.2]{M12} guarantees that minimisers of the prescribed mean curvature problem are $(\Lambda, r_0)$-perimeter minimisers with $r_0$ arbitrary, and $\Lambda$ being the mean curvature value. See also \cite[Appendix B]{CES22}. Thus, by the above reasoning, the notion of volume constrained minimisers extends in a natural way to closed Riemannian manifolds; with compactness and regularity results, again with identical conclusions on the singular set as for isoperimetric regions, following by a volume-fixing argument similar to that of the proof of Lemma \ref{lemma: compactness of isoperimetric regions}. In particular, given $\delta > 0$ and a Riemannian metric, $g$, on $\mathbb{B}_\delta$ we will write
\begin{align*}
       \mathrm{VCM}(\delta,g) = \{& E \in \mathcal{C}(\mathbb{B}_\delta) \, | \, E \text{ is a volume constrained minimiser in } (\mathbb{B}_\delta,g) \};
\end{align*} 
this notation will used when defining the singular capacity for isoperimetric regions in Section \ref{sec: singular capacity for isoperimetric regions}.

\subsection{Twisted Jacobi fields and a spectral theorem}\label{subsec: twisted Jacobi fields}

In this subsection we introduce a suitable notion of Jacobi field for constant mean curvature hypersurfaces with isolated singularities which respects their stationarity with respect to only volume preserving deformations.

\bigskip

Given $\Omega \in \mathcal{I}(g,t)$, associated to the second variation of the area functional we have the following quadratic form on functions $\phi \in C^1_c(\Sigma)$ given by
\begin{equation*}
        Q_\Sigma (\phi,\phi) = \int_\Sigma \left(|\nabla \phi|_g^2 - (|\mathrm{I\!I}_\Sigma|_g^2 + \mathrm{Ric}_g(\nu_{\Sigma,g},\nu_{\Sigma,g}) ) \phi^2 \right) \, dA_g,
\end{equation*}
where $\mathrm{I\!I}_\Sigma$ is the second fundamental form of $\Sigma$ and $\mathrm{Ric}_g(\nu_{\Sigma,g},\nu_{\Sigma,g})$ is the Ricci curvature of $M$ with respect to the metric $g$ evaluated on the unit normal, $\nu_{\Sigma,g}$, to $\Sigma$. We then define the \textbf{Jacobi operator} associated to the second variation to be
\begin{equation*}
    L_{\Sigma,g} = \Delta_g + \left(|\mathrm{I\!I}_\Sigma|_g^2 + \mathrm{Ric}_g(\nu_{\Sigma,g}, \nu_{\Sigma,g}) \right),
\end{equation*}
where $\Delta_g$ is the Laplace operator with respect to the metric $g$. As in \cite{BdE88, BB00} we define the function spaces
$$\mathcal{D}_T(\Sigma)= \left\{\phi \in C^\infty_c (\Sigma) \, \bigg| \, \int_\Sigma \phi \, dA_g=0 \right\},$$ 
and for each $p \geq 1$
$$L^p_T(\Sigma)= \left\{\phi \in L_g^p (\Sigma)\, \bigg| \, \int_\Sigma \phi \, dA_g =0 \right\}.$$
We define the \textbf{twisted Jacobi operator} for $\phi \in C_c^\infty(\Sigma)$ by setting $$\widetilde{L}_{\Sigma,g}\phi =  L_{\Sigma,g}\phi - \frac{1}{|\Sigma|_g}\int_\Sigma L_{\Sigma,g}\phi \, dA_g,$$
which we note is $L^2$ self-adjoint when restricted to $\mathcal{D}_T(\Sigma)$. Recall that, as mentioned in Subsection \ref{subsec: notation}, in the above and hereafter when dealing with averaged integrals we are writing $|\Sigma|_g = \mathrm{Per}_g(\Sigma)$ for notational convenience. We are restricting the quadratic form, $Q_\Sigma$, (or more precisely its associated bilinear form) to a different function space in order to capture a suitable notion of stability for constant mean curvature hypersurfaces; often referred to as weak stability, e.g.~\cite{BdE88}, which in particular ensures that Euclidean spheres are (weakly) stable.

\bigskip

We now generalise the notion of twisted Jacobi fields:
\begin{definition}[Twisted Jacobi fields]
  Given $\Omega \in \mathcal{I}(g,t)$, we say that a function $u \in L^1_T(\Sigma) \cap C^2_\mathrm{loc}(\Sigma)$ is a \textbf{twisted Jacobi field} on $\Sigma$ if $\widetilde{L}_{\Sigma,g}u = 0$ pointwise on $\Sigma$.
\end{definition}
    
\begin{remark}
    If $\Sigma$ has no singularities then this definition agrees with the notion of twisted Jacobi field introduced in \cite{BdE88}.
\end{remark}

We now introduce function spaces adapted to the singular geometry. We first observe the following:
    \begin{lemma}\label{lem: C} For each $\Omega \in \mathcal{I}(g,t)$ there exists a constant $C > 0$, depending on $\Sigma$, $g$, and $t$, such that for all $\psi \in C^1_c(\Sigma)$ we have
        $$ Q_\Sigma (\psi,\psi) + C \Vert {\psi}^2\Vert_{L_g^2(\Sigma)} \ge 0.$$ 
    \end{lemma}
\begin{proof}
        This is an adaptation of \cite[Lemma D.1]{LW25} to the case where $\Sigma$ is locally stable with constant mean curvature. In particular, following the same computation we obtain the inequality $$Q_\Sigma(\psi, \psi) \geq \int_\Sigma \psi^2 \sum_{j \geq 1} \rho_j \Delta_\Sigma \rho_j, $$ where $\{\rho_j\}_{j \geq 1}$ is a partition of unity subordinate to a finite cover of $\overline{\Sigma}$ by open sets in each of which $\Sigma$ is stable. Noting that $\Delta_g f = \nabla^2 f(\nu_{\Sigma,g}, \nu_{\Sigma,g}) + H_\Sigma \nu_{\Sigma,g}(f) + \Delta_\Sigma f,$
        where $H_\Sigma$ denotes the (constant) mean curvature of $\Sigma$, we compute that
        \begin{equation*}
            \vert \Delta_\Sigma \rho_j \vert = \vert \Delta_g \rho_j - \nabla^2 \rho_j(\nu_{\Sigma,g}, \nu_{\Sigma,g}) - H_\Sigma \nu_{\Sigma,g}(\rho_j)\vert \leq n \vert \nabla^2 \rho_j(\nu_{\Sigma,g}, \nu_{\Sigma,g}) \vert + \vert H_\Sigma\vert \vert \nabla \rho_j \vert.
        \end{equation*}
        Thus we see that $|\sum_j \rho_j\cdot \Delta_g\rho_j | \ge - C $ for a constant $C > 0$, depending on $\Sigma$, $g$, and $t$ as desired.
\end{proof}

\begin{definition}
        Given $\Omega \in \mathcal{I}(g,t)$, for $\psi\in C^1_c (\Sigma)$ we define the norm $$\Vert \psi\Vert_{\mathscr{B}(\Sigma)} = Q_\Sigma (\psi,\psi) +(C+1) \Vert \psi\Vert_{L^2(\Sigma)},$$
        where $C > 0$ is as in in Lemma \ref{lem: C}. We then define the Hilbert spaces $\mathscr{B}(\Sigma) = \overline{C^1_c (\Sigma)}^{\mathscr{B}}$ (i.e.~the completion with respect to the $\mathscr{B}$ norm as defined above) and $\mathscr{B}_T(\Sigma) = \mathscr{B}(\Sigma) \cap L^2_T(\Sigma)$, both equipped with the $L^2$ inner product. 
\end{definition}

\begin{remark}
    The above Hilbert spaces serve as suitable replacements for the standard Sobolev spaces for inverting the twisted Jacobi operator in the presence of isolated singularities. As remarked in \cite[Section 5.3.1]{W23}, $W^{1,2}_0$ is only subset of $\mathscr{B}_0$ in general; however, when every singularity is strongly isolated with strictly stable tangent cone, we have $W^{1,2}_0 = \mathscr{B}_0$. We refer to \cite[Example 5.3.3]{W23} for an example where this equality fails (which is always the case if some isolated singularity has a tangent cone which is not strictly stable). 
\end{remark}

\begin{definition}[Weak twisted Jacobi fields] Given $h \in L^2_T(\Sigma)$ and $\Omega \in \mathcal{I}(g,t)$, we say that a function $u \in \mathscr{B}_T(\Sigma)$ is a \textbf{weak solution} to $\widetilde{L}_{\Sigma,g}u = h$ if for each $\phi \in C^\infty_c(\Sigma)$ we have $Q_\Sigma(u,\phi) = \langle u,h\rangle_{L^2_g(\Sigma)}$. In particular, we say that $u \in \mathscr{B}_T(\Sigma)$ is a \textbf{weak twisted Jacobi field} on $\Sigma$ if it is a weak solution to $\widetilde{L}_{\Sigma,g}u = 0$. We denote the collection of all weak twisted Jacobi fields on $\Sigma$ by $$\mathrm{Ker}\widetilde{L}_{\Sigma,g} = \{\omega \in \mathscr{B}_T(\Sigma) \, | \, \widetilde{L}_{\Sigma,g}\omega = 0\} \subset \mathscr{B}_T(\Sigma),$$
and say that $\Sigma$ is \textbf{non-degenerate} if $\mathrm{Ker}\widetilde{L}_{\Sigma,g} = \{0\}$ (and \textbf{degenerate} if not).
\end{definition}

\begin{remark}\label{rem: higher regularity for weak solutions}
    One can show, e.g.~as in \cite[Corollary 4.2.22]{N24A}, that weak solutions as defined above are in fact smooth. The method of proof is similar to establishing higher interior regularity for linear elliptic equations, the only difference being that we test with integral zero functions.
\end{remark}

We note that by Lemma \ref{lem: C}, $Q_\Sigma$ extends to a well defined quadratic form on $\mathscr{B}_T(\Sigma)$, which continuously embeds into $L^2$; in fact, this embedding is compact and we establish the following spectral theorem for twisted Jacobi fields:

\begin{theorem}[Spectral theorem for the twisted Jacobi operator]\label{thm: spectral theorem for twisted Jacobi operator}
    Given $\Omega \in \mathcal{I}(g,t)$, we have that:
\begin{enumerate}
    \item $\mathscr{B} (\Sigma)$ and $\mathscr{B}_{T}(\Sigma)$ are compactly embedded in $L^2(\Sigma)$ and $L^2_T(\Sigma)$ respectively.
            
    \item There exists a strictly increasing sequence, $\sigma_p(\Sigma)=\{\lambda_j\}_{j \geq 1}$, diverging to infinity and finite-dimensional pairwise $L^2$-orthogonal linear subspaces, $\{E_j\}$, of $\mathscr{B}_T(\Sigma) \cap C^\infty(\Sigma)$ such that
    $$-\widetilde{L}_{\Sigma,g} \psi  = \lambda_j \psi$$
    for all $\psi \in E_j$. Furthermore, 
    $$L^2_T (\Sigma) = span_{L^2} \{E_j\}_{j \geq 1} \quad \text{ and } \quad \mathscr{B}_T(\Sigma) = span_{\mathscr{B}} \{E_j\}_{j \geq 1},$$
    where the notation above denotes the closure of the span in the respective subscript norm.
            
    \item For each $f \in (\mathrm{Ker}\widetilde{L}_{\Sigma,g})^\perp = \{ g \in L^2(\Sigma) \, | \, \langle g,\omega \rangle_{L^2(\Sigma)} = 0 \text{ for all }  \omega \in \mathrm{Ker}\widetilde{L}_{\Sigma,g}\}$ there exists a unique $\psi \in \mathscr{B}_T(\Sigma) \cap (\mathrm{Ker}\widetilde{L}_{\Sigma,g})^\perp$ such that
        $$-\widetilde{L}_{\Sigma,g}\psi = f $$
        on $\Sigma$ and with the estimate
    \begin{align*}
        \|\psi\|_{\mathscr{B}(\Sigma)} \le C \|f\|_{L^2(\Sigma)},
    \end{align*}
    for a constant $C>0$, depending on $\Sigma$ and $g$. 
    \end{enumerate}
\end{theorem}

\begin{proof}
    For the first part, just as in \cite[Lemma D.3/Lemma E.1]{LW25}, the compact embeddings follow by establishing an $L^2$ non-concentration estimate near $\mathrm{Sing}(\Sigma)$ of the form in \cite[Lemma 3.9]{W20}. Precisely, given $\Omega \in \mathcal{I}(g,t)$, for each $\varepsilon > 0$ one can show that there is an open neighbourhood, $V_\varepsilon$, of $\mathrm{Sing}(\Sigma)$ in $M$ such that 
        $$ \int_{V_\varepsilon \cap \Sigma} \phi^2 \, dA_g \leq \varepsilon \cdot \Vert \phi\Vert^2_{\mathscr{B}(\Sigma)}$$
        for all $\phi \in C^1_c(\Sigma)$.
    As the space of test functions used is the same as in \cite[Lemma E.1]{LW25} and the argument hinges on the Michael--Simon--Sobolev inequality (see \cite{MS73}) through an embedding into Euclidean space, the same proof there carries through identically in our setting since the mean curvature of $\Sigma$ is bounded by Lemma \ref{Lemma: mean curvature bound}. Using the same reasoning as \cite[Proposition 3.5]{W20} we obtain the compact embeddings for part 1 as desired.

    \bigskip

    Given part 1, parts 2 and 3 then follow as discussed in the proof of \cite[Lemma D.3]{LW25}; namely part 2 follows by applying \cite[Theorem 8.37]{GT01} with the $\mathscr{B}$ norm in place of the $W^{1,2}_0$ norm, and part 3 follows from the spectral decomposition in part 2. More details on these arguments can be found in \cite[Theorem 4.2.17]{N24A}.
\end{proof}

\begin{remark}
    We note that an analogue of Theorem \ref{thm: spectral theorem for twisted Jacobi operator} in the case that $\mathrm{Sing}(\Sigma) = \emptyset$ for the twisted Jacobi operator was established in \cite[Proposition 2.2]{BB00}. 
\end{remark}

\section{Asymptotic rates and metric perturbations}\label{sec: asymptotic rates and metric perturbations}

In this section we study the behaviour of isoperimetric regions near isolated singularities and show that, provided a suitably defined growth rate for the hypersurface is sufficiently fast, we can remove these singularities by global metric perturbation.

\subsection{Definitions and perturbation functions}\label{subsec: definitions and perturbation functions}

We first introduce the notion of asymptotic rates and growth rates for functions defined on the boundaries of isoperimetric regions and regular cones: 

\begin{definition} \label{definition: asymptotic and growth rates} Given $\Omega \in \mathcal{I}(g,t)$, $p\in \mathrm{Sing}(\Sigma)$, $v\in L^2_{g,{\mathrm{loc}}}(\Sigma)$, $\mathbf{C} \in \mathscr{C}$, $w \in L^2_{\mathrm{loc}}(\mathbf{C})$, and $K > 1$, we define:
\begin{itemize}
    \item The \textbf{asymptotic rate} of $v$ at $p$ to be
    $$\mathcal{AR}_p(v) = \sup\left\{ \gamma  \, \bigg| \, \lim_{s\searrow 0} \int_{A(p,s,2s)} v^2 \cdot \rho^{-n-2\gamma} \, d|\Sigma|_g =0 \right\},$$
    and the \textbf{growth rate} of $v$ at $p$ to be
        $$ J_{K;\Sigma,g}^{\gamma}(v;r) = \left(\int_{A(p;K^{-1}r, r)} v^2 \cdot \rho^{- n - 2\gamma} \, d|\Sigma|_g\right)^{\frac{1}{2}},$$
    where $\rho(x) = \mathrm{dist}_g(x, p)$. 

    \item The \textbf{asymptotic rate} of $w$ at infinity to be
        $$\mathcal{AR}_\infty(w) = \inf\left\{ \gamma \, \bigg| \, \lim_{s\nearrow +\infty} \int_{\mathbf{A}(s,2s)} w^2(x) \cdot |x|^{-n-2\gamma} \, d|\mathbf{C}|(x) = 0 \right\},$$ 
    and the growth rate of $w$ on $\mathbf{C}$ to be
    \begin{equation*}
        J_{K;\mathbf{C}}^{\gamma}(w;r) = \left(\int_{\mathbf{A}(K^{-1}r, r)} w^2(x) \cdot \vert x \vert^{- n - 2\gamma} \, d|\mathbf{C}|(x)\right)^\frac{1}{2}.
    \end{equation*}

    \item For a stationary integral $7$-varifold, $V$, in $\mathbb{R}^8 \setminus \mathbb{B}_{R_0}$ for some $R_0 > 0$, we say that $V$ is \textbf{asymptotic to $\mathbf{C}$ at infinity} if $V \mres \mathbb{R}^8 \setminus \mathbb{B}_{R_0} = |\mathrm{graph}^{g_\mathrm{eucl}}_{\mathbf{C} \setminus \mathbb{B}_{R_0}}(w)|$ and $\Vert w\Vert_{C^2_*(\mathbf{A}(R,2R))} \rightarrow 0$ as $R \rightarrow \infty$ for some $w \in L^2_\mathrm{loc}(\mathbf{C})$; for such a varifold, $V$, asymptotic to $\mathbf{C}$ at infinity we write $\mathcal{AR}_\infty(V) = \mathcal{AR}_\infty(w)$. Moreover, if $\Sigma \subset \mathbb{R}^{8} \setminus \mathbb{B}_{R_0}$ is a hypersurface asymptotic to $\mathbf{C}$, we define $\mathcal{AR}_\infty(\Sigma) = \mathcal{AR}_\infty(\vert \Sigma \vert)$. 
\end{itemize}
For the above definitions we adopt the convention that $\inf \emptyset = \infty$ and $\inf\mathbb{R} = -\infty$. 
\end{definition}

\begin{remark}\label{rem: J on cone to J on cmc} By expressing $\Sigma$ in conical coordinates over its tangent cone $\mathbf{C}_p\Sigma$, there is some constant $C > 0$, depending on the radius in which these coordinates are defined, such that
\begin{align*}
       C^{-1}J^{\gamma}_{K;\mathbf{C}}(u;r)^2 \leq J^{\gamma}_{K;\Sigma,g}(u;r)^2
\leq CJ^{\gamma}_{K;\mathbf{C}}(u;r)^2.
   \end{align*}
\end{remark}

\begin{remark}\label{rem: property of J^gamma}    
    If $v$ grows at a rate $r^\gamma$ on approach to $p \in \mathrm{Sing}(\Sigma)$, then we should expect $\mathcal{AR}_p(v) = \gamma$. More precisely, if we suppose $\mathcal{AR}_p(v) \in \mathbb{R}$, then $\mathcal{AR}_p(v) < \gamma$ implies that $\mathrm{\limsup}_{t\to 0^+} J_{K;\Sigma,g}^{\gamma} (\phi; t) = \infty$, and similarly $\mathcal{AR}_p(v) > \gamma$ implies $\mathrm{\liminf}_{t\to 0^+} J_{K;\Sigma,g}^{\gamma} (\phi; t)  = 0$.
\end{remark}

\begin{definition}\label{def: slow growth and semi-nondegenerate}
    Given $\Omega \in \mathcal{I}(g,t)$, a function $v\in L^2_{g,\mathrm{loc}}(\Sigma)$ is of \textbf{slow growth} if for each $p\in \mathrm{Sing}(\Sigma)$ we have 
    $$ \mathcal{AR}_p (v) \geq \gamma^+_2(\mathbf{C}_p \Sigma),$$
    and we let $\mathrm{Ker}^+ \widetilde L_{\Sigma,g} \subset \mathrm{Ker}\widetilde{L}_{\Sigma,g}$ be the space of twisted Jacobi fields which are of slow growth. We say that $\Omega$, or equivalently $\Sigma$, is \textbf{semi-nondegenerate} if 
    $$ \mathrm{Ker}^+  \widetilde L_{\Sigma,g} = \{ 0 \},$$
    i.e.~the only twisted Jacobi field of slow growth on $\Sigma$ is trivial.
\end{definition}

We record here the following useful properties of the asymptotic rate:

\begin{lemma}\label{lemma: properties of slow growth}
    Suppose that $p \in \mathrm{Sing}(\Sigma)$ and $u \in W^{2,2}_{g,\mathrm{loc}}(\Sigma)$ is such that $L_{\Sigma,g}u \in L^{\infty}(\Sigma)$, then:
    \begin{enumerate}
    \item $\mathcal{AR}_p (u) \in \{-\infty\} \cup \Gamma(\mathbf{C}_p\Sigma) \cup [1,+\infty]$.
        \item If $u > 0$ and $L_{\Sigma,g} u$ vanishes near $p$, then $\mathcal{AR}_p(u) \in \{\gamma_1^- (\mathbf{C}_p \Sigma),\gamma_1^+(\mathbf{C}_p \Sigma)\}$.
        
        \item If $u \in W_{g}^{1,2}(\Sigma)$, then $\mathcal{AR}_p(u) \geq \gamma_1^{+}(\mathbf{C}_p\Sigma)$.
        
        \item If $\mathcal{AR}_p(u) > -(n-2)/2$, then
        \begin{align*}
            \int_\Sigma |\nabla_\Sigma u|^2 +\rho^{-2} u^2 \, dA_g <+\infty,
        \end{align*}
        where $\rho(x) = \mathrm{dist}_g(x, p)$.
\end{enumerate}
        In particular, if $u,v \in W^{2,2}_{g,\mathrm{loc}}(\Sigma) \cap L^1_T(\Sigma)$ such that $L_{\Sigma,g} u,L_{\Sigma,g} v \in L^{\infty}(\Sigma)$, and $\mathcal{AR}_p(u), \mathcal{AR}_p(v) > -(n-2)/2$, then for every $p \in \mathrm{Sing}(\Sigma)$ we have the following integration by parts:
        $$ \int_\Sigma u \cdot \widetilde{L}_{\Sigma,g}v \, d A_g =  \int_\Sigma v \cdot \widetilde{L}_{\Sigma,g}u \, d A_g .$$
\end{lemma}
\begin{proof}
Parts 1 through 4 are precisely the content of \cite[Lemma 2.4/Corollary D.5]{LW25}. In order to establish the integration by parts for slow growth functions we note that since $u,v \in W^{2,2}_{g,\mathrm{loc}}(\Sigma) \cap L^1_T(\Sigma)$, we have
\begin{align*}
    \int_\Sigma u \cdot \widetilde{L}_{\Sigma,g}v \, d A_g &= \int_\Sigma u \cdot \left(L_{\Sigma,g}v -\frac{1}{|\Sigma|_g}\int_\Sigma L_{\Sigma,g}v \, dA_g\right)\, d A_g\\
    &=\int_\Sigma u \cdot L_{\Sigma,g}v\, d A_g.
\end{align*}
Again by \cite[Lemma 2.4/Corollary D.5]{LW25} we can then apply the integration by parts formula there (since $L_{\Sigma,g}u,L_{\Sigma,g}v \in L^\infty(\Sigma)$) to see that
\begin{align*}
    \int_\Sigma u \cdot L_{\Sigma,g}v\, d A_g = \int_\Sigma L_{\Sigma,g}u \cdot v\, d A_g,
\end{align*}
then by reversing the calculation above with the roles of $u$ and $v$ swapped this yields the integration by parts for slow growth functions as desired.
\end{proof}

\begin{remark}\label{rem: slow growth remarks}
    By Lemma \ref{lemma: properties of slow growth} part 4, we see that every twisted Jacobi field of slow growth actually belongs to $W^{1,2}_g(\Sigma)$. As a consequence, if $\Sigma$ is nondegenerate, so that in the notation of Theorem \ref{thm: spectral theorem for twisted Jacobi operator} part 2 we have $0 \notin \sigma_p(\Sigma)$ (i.e.~$0$ is not an eigenvalue of $-\widetilde{L}_{\Sigma,g}$), then in fact $\Sigma$ is also semi-nondegenerate; hence if $\mathrm{Sing}(\Sigma) = \emptyset$, semi-nondegeneracy coincides with the usual notion of nondegeneracy. When $\mathrm{Sing}(\Sigma) \neq \emptyset$ however, semi-nondegeneracy does not in general guarantee either degeneracy or non-degeneracy of $\Sigma$.
\end{remark}

Given a semi-nondegenerate isoperimetric region, we now produce a set of functions that later will be utilised in order to perturb away isolated singularities by a conformal change in metric:

\begin{proposition}[Perturbation functions]\label{prop: perturbation functions open and dense}
    If $\Sigma$ is semi-nondegenerate then we have that the set $\mathcal{V}^{k,\alpha}$, defined to be $$\left\{f \in C^{k,\alpha}(M) \, \bigg| \, \widetilde{L}_{\Sigma,g}u = \nu_{\Sigma,g}(f) - \frac{1}{|\Sigma|_g}\int_\Sigma \nu_{\Sigma,g}(f) \, dA_g \text{ has no slow growth solution $u \in L^1_T(\Sigma)$}\right\},$$
    contains some open and dense subset of $C^{k,\alpha}(M)$.
\end{proposition}

\begin{proof}
    We consider cases on $\mathrm{Ker}\widetilde{L}_{\Sigma,g} = \{\omega \in \mathscr{B}_T(\Sigma) \, | \, \widetilde{L}_{\Sigma,g}\omega = 0\}$. First, if $\mathrm{Ker}\widetilde{L}_{\Sigma,g} \neq \{0\}$ then we guarantee the existence of some non-zero $w \in \mathrm{Ker}\widetilde{L}_{\Sigma,g} \subset \mathscr{B}_T(\Sigma)$. If $h = \nu_{\Sigma,g}(f) - \frac{1}{|\Sigma|_g}\int_\Sigma \nu_{\Sigma,g}(f)$ and we had some slow growth solution, $u \in L^1_T(\Sigma)$, to $\widetilde{L}_{\Sigma,g}u = h$ then, by integrating by parts from Lemma \ref{lemma: properties of slow growth}, we have for every such $\omega \in \mathrm{Ker}\widetilde{L}_{\Sigma,g}$ (noting that in particular $w \in L^1_T(\Sigma)$) that
    $$\int_{\Sigma} w \cdot h \, dA_g = \int_{\Sigma} w \cdot \widetilde{L}_{\Sigma,g}u\, dA_g  = \int_{\Sigma}\widetilde{L}_{\Sigma,g}w \cdot u \, dA_g= 0;$$
    thus $h \in(\mathrm{Ker}\widetilde{L}_{\Sigma,g})^\perp$. We therefore see from Theorem \ref{thm: spectral theorem for twisted Jacobi operator} part 3 that $$C^{k,\alpha}(M) \setminus \mathcal{V}^{k,\alpha} =\left\{f \in C^{k,\alpha}(M) \, \bigg| \, \nu_{\Sigma,g}(f) - \frac{1}{|\Sigma|_g}\int_\Sigma \nu_{\Sigma,g}(f) \, dA_g \in (\mathrm{Ker}\widetilde{L}_{\Sigma,g})^\perp\right\},$$
    from which we conclude that, since in this case $\mathrm{dim}(\mathrm{Ker}\widetilde{L}_{\Sigma,g}) < \infty$ by Theorem \ref{thm: spectral theorem for twisted Jacobi operator} part 2,  $C^{k,\alpha}(M) \setminus \mathcal{V}^{k,\alpha}$ forms a proper closed linear subspace of $C^{k,\alpha}(M)$; hence $\mathcal{V}^{k,\alpha}$ is open and dense in $C^{k,\alpha}(M)$.

    \bigskip

     If $\mathrm{Ker}\widetilde{L}_{\Sigma,g} = \{0\}$, first choose $r \in \left(0, \frac{\mathrm{inj}(\Sigma,g)}{2}\right)$ sufficiently small so that for each $p \in \mathrm{Sing}(\Sigma)$ the Jacobi operator $-L_\Sigma$ restricted to $B_{2r}^g(p) \cap \Sigma$ has positive first eigenvalue (e.g.~as shown in \cite[Claim E.1]{LW25}); note also that the sets $\{B_{2r}^g(p)\}_{p \in \mathrm{Sing}(\Sigma)}$ are pairwise disjoint. Exactly as in \cite[Proof of Lemma D.6]{LW25} for each $p \in \mathrm{Sing}(\Sigma)$ we define $\xi_p,\check{\xi}_p \in \mathscr{B}(B_r^g(p) \cap \Sigma) = \overline{C_c^1(B_r^g \cap \Sigma)}^\mathscr{B}$ as the unique positive solutions, guaranteed by the maximum principle and \cite[Lemma D.3]{LW25} (the non-twisted analogue of Theorem \ref{thm: spectral theorem for twisted Jacobi operator}), to
     $$\begin{cases}
         L_{\Sigma,g}\xi_p = 0 & \text{on } B_r^g(p)\\
         \xi_p = 1 & \text{on } \partial B_r^g(p)
     \end{cases} \text{ and} \begin{cases}
         (1 - L_{\Sigma,g}\check{\xi}_p) = 0 & \text{on } B_r^g(p)\\
         \xi_p = 1 & \text{on } \partial B_r^g(p)
     \end{cases},$$
    respectively. Then $\xi_p \geq \check{\xi}_p > 0$, and by Lemma \ref{lemma: properties of slow growth} parts 1 and 2 we ensure that both $\mathcal{AR}_p(\xi_p) = \mathcal{AR}_p(\check{\xi}_p) = \gamma^+_1(\mathbf{C}_p\Sigma)$ and $\xi_p(x),\check{\xi}_p(x) \rightarrow \infty$ as $x \rightarrow p$. In particular whenever $\limsup_{x \rightarrow p}\frac{|u(x)|}{\xi_p(x)} > 0$ we have $\mathcal{AR}_p(u) \leq \gamma_1^+(\mathbf{C}_p)$, which is saying that $u$ does not have slow growth at $p \in \mathrm{Sing}(\Sigma)$. We then consider the sets
    $$V_p = \left\{h \in L_g^\infty(\Sigma) \cap L^1_T(\Sigma) \, \bigg| \, \widetilde{L}_{\Sigma,g}u = h \text{ has unique solution } u \in \mathscr{B}_T(\Sigma) \text{ with } \limsup_{x \rightarrow p}\frac{|u(x)|}{\xi_p(x)} > 0 \right\};$$
    where the uniqueness of such solutions in the definition of $V_p$ follows from Theorem \ref{thm: spectral theorem for twisted Jacobi operator} part 3. We then define 
    \begin{equation}\label{eqn: perturbation sets}
        G = \bigcap_{p \in \mathrm{Sing}(\Sigma)}\left\{f \in C^{k,\alpha}(M) \, \bigg| \, \nu_{\Sigma,g}(f) - \frac{1}{|\Sigma|_g}\int_\Sigma \nu_{\Sigma,g}(f) \, dA_g\in V_p\right\} \subset \mathcal{V}^{k,\alpha};
    \end{equation}
    we will show that the set $G$ is open and dense in $C^{k,\alpha}(M)$.
    
    \bigskip

    We first note that if $V_p$ is open in $L_g^\infty(\Sigma) \cap L^1_T(\Sigma)$ (with respect to the supremum norm), then this implies that the set $$\left\{f \in C^{k,\alpha}(M) \, \bigg| \, \nu_{\Sigma,g}(f) - \frac{1}{|\Sigma|_g}\int_\Sigma \nu_{\Sigma,g}(f) \, dA_g\in V_p\right\}$$ is open in $C^{k,\alpha}(M)$; thus if we can show that $V_p$ is open in $L_g^\infty(\Sigma) \cap L^1_T(\Sigma)$ for each $p \in \mathrm{Sing}(\Sigma)$, we conclude that $G$ is also open in $C^{k,\alpha}(M)$.

    \bigskip

    To show this we first establish that if $h \in (L_g^\infty(\Sigma) \cap L_T^1(\Sigma)) \setminus\{0\}$ and $u \in \mathscr{B}_T(\Sigma)$ is such that $\widetilde{L}_{\Sigma,g}u = h$, then there is some constant $C > 0$, depending on $\Sigma,M,$ and $g$, such that
    \begin{equation}\label{eqn: bound for perturbation proof}
        \Vert L_{\Sigma,g}u\Vert_{L^\infty(\Sigma)} \leq C \Vert h\Vert_{L^\infty(\Sigma)}.
    \end{equation}
    Since in particular $h \in L^1_T(\Sigma)$, we thus have that $L_{\Sigma,g}u - h = \frac{1}{|\Sigma|_g} \int_\Sigma L_{\Sigma,g}u\, dA_g$ is constant. We now bound the constant $C_u =  \frac{1}{|\Sigma|_g} \int_\Sigma L_{\Sigma,g}u\, dA_g$ in terms of $\Vert h\Vert_{L^\infty(\Sigma)}$; noting that it suffices to bound it at an arbitrary point in $\Sigma$.

    \bigskip

    Fix disjoint open sets $U,V \subset \subset \Sigma$, non-negative $\phi_1 \in C^\infty_c(U)\setminus \{0\}$ and $\phi_2 \in C^\infty_c(V)\setminus \{0\}$ such that $\int_\Sigma\phi_1\, dA_g = -\int_\Sigma \phi_2\, dA_g$. As $\widetilde{L}_{\Sigma,g}u = h$ and $\phi_1 + \phi_2 \in \mathcal{D}_T(\Sigma) \setminus \{0\}$ we have, after integrating by parts (since $\phi_1,\phi_2$ are supported away from $\mathrm{Sing}(\Sigma)$), that
    $$Q_\Sigma(u,\phi_2) + \langle h, \phi_2 \rangle = \langle L_{\Sigma,g}u - h,\phi_1\rangle = C_u \cdot \Vert \phi_1\Vert_{L^1(\Sigma)}.$$
    We then bound
    $$|C_u| \cdot \Vert \phi_1\Vert_{L^1(\Sigma)}\leq |Q_\Sigma(u,\phi_2)| + \Vert h\Vert_{L^2(\Sigma)} \cdot \Vert \phi_2\Vert_{L^2(\Sigma)}.$$
    Using Theorem \ref{thm: spectral theorem for twisted Jacobi operator} part 3 we have that
    $$|Q_\Sigma(u,\phi_2)| \leq C \Vert u\Vert_{\mathscr{B}(\Sigma)} \cdot \Vert \phi_2\Vert_{\mathscr{B}(\Sigma)} \leq C \Vert h\Vert_{L^2(\Sigma)} \cdot \Vert \phi_2\Vert_{\mathscr{B}(\Sigma)}.$$
    Combining the above estimates we see that there exists some constant $C > 0$, depending on $\Sigma$, $M$, $g$, and the choice of $\phi_1,\phi_2$, such that
    $$|C_u| \leq C\Vert h\Vert_{L^\infty(\Sigma)},$$
    and so we see that (\ref{eqn: bound for perturbation proof}) holds as desired.
    
    \bigskip

    To show that $V_p$ is open for each $p \in \mathrm{Sing}(\Sigma)$ we will show that its complement is closed by deducing that $\limsup_{x \rightarrow p}\frac{|u(x)|}{\xi_p(x)}$ is a continuous function of $h$, where $\widetilde{L}_{\Sigma,g}u = h$; as $\widetilde{L}_{\Sigma,g}$ is linear it suffices to verify this at the zero function. If $\Vert h\Vert_{L^\infty_g(\Sigma)} \leq 1$, we set $\widehat{C} = \frac{\max\{1+C,\widetilde{C}\}}{\kappa}$ where $C > 0$ is as in (\ref{eqn: bound for perturbation proof}), $|u| \leq \widetilde{C}$ on $\partial B_r^g(p)$ (such a constant exists by Theorem \ref{thm: spectral theorem for twisted Jacobi operator} part 3), and $\kappa = \inf_{B_r^g(p)}\check{\xi}_p> 0$, then we compute that on $\partial B_r^g(p)$ we have
    $|u| \leq \widetilde{C} \leq \widehat{C}\kappa \leq \widehat{C}(2\xi_p - \check{\xi}_p)$ by definition and that by (\ref{eqn: bound for perturbation proof}) on $B_r^g(p)$ we have
    $$|L_{\Sigma,g}u| \leq \Vert h\Vert_{L^{\infty}_g(\Sigma)} + C_u \leq 1 + C \leq (1+C)\frac{\check{\xi}_p}{\kappa}\leq \widehat{C}L_{\Sigma,g}(2\xi_p - \check{\xi}_p).$$
    Then by applying the comparison/weak maximum principle we see that $|u| \leq \widehat{C}(2\xi_p - \check{\xi}_p) \leq 2\widehat{C}\xi_p$ on $B_r^g(p)$ and hence the operator norm of $\limsup_{x \rightarrow p}\frac{|u(x)|}{\xi_p(x)}$ is bounded by $2\widehat{C}$. Thus we deduce that $V_p$ is open for each $p  \in \mathrm{Sing}(\Sigma)$, and hence $G$ is open as argued above.
    
    \bigskip

    To conclude that $G$ is dense we note that given $\eta \in V_p$, $\zeta \in (L_g^\infty(\Sigma) \cap L^1_T(\Sigma)) \setminus V_p$, and $c \in \mathbb{R} \setminus \{0\}$ we have $c\eta + \zeta \in V_p$. Given this, we now claim that it is sufficient to establish the denseness of $G$ by showing that $C_c^\infty(\Sigma) \cap V_p \neq \emptyset$ for each $p \in \mathrm{Sing}(\Sigma)$. To see this, given any $h \in C^{k,\alpha}(M) \setminus G$, as  $\nu_{\Sigma,g}(h) - \frac{1}{|\Sigma|_g}\int_\Sigma \nu_{\Sigma,g}(h)\, dA_g \in (L_g^\infty(\Sigma) \cap L^1_T(\Sigma)) \setminus V_p$ for each $p \in \mathrm{Sing}(\Sigma)$, then if $\eta \in C_c^\infty(\Sigma) \cap V_p$ for some $p \in \mathrm{Sing}(\Sigma)$ then we have that $c\eta + \nu_{\Sigma,g}(h) - \frac{1}{|\Sigma|_g}\int_\Sigma \nu_{\Sigma,g}(h)\, dA_g \in V_p$. For any $f \in C^{k,\alpha}(M)$ with $\nu_{\Sigma,g}(f) = \eta$ on $\Sigma$ (which exists since $\eta \in C^\infty_c(\Sigma)$), we ensure that $cf + h \in G$ (since $\nu_{\Sigma,g}(cf  + h) = c\eta + \nu_{\Sigma,g}(h)$ and in particular $\eta \in L^1_T(\Sigma)$) and is arbitrarily close to $h$ in $C^{k,\alpha}(M)$ norm by taking $c$ small.
    
    \bigskip

    We now construct some $\eta \in C_c^\infty(\Sigma) \cap V_p$ which will conclude the denseness. Consider a smooth cutoff function $\chi \in C^\infty(\mathbb{R})$ identically equal to one if $|x| \leq \frac{1}{2}$ and zero if $|x| \geq \frac{3}{4}$, and consider $\widetilde{\xi}_p(x) = \chi(\frac{d_\Sigma(x,p)}{\tau_p}) \cdot \xi_p(x)$; we then see that both $\mathcal{AR}_p(\widetilde{\xi}_p) = \gamma_1^+(\mathbf{C}_p)$ and $\widetilde{\xi}_p \in L^1(\Sigma)$. We then choose $f \in C_c^\infty(\Sigma)$ such that $\int_\Sigma f \, dA_g= -\int_\Sigma \widetilde{\xi}_p\, dA_g$, so that $\widetilde{\xi}_p + f \in L^1_T(\Sigma)$. Next, we may choose a further function $h \in C_c^\infty(\Sigma) \cap L^1_T(\Sigma)$ such that $\frac{1}{|\Sigma|_g}\int_{\Sigma} L_{\Sigma,g} h \, dA_g= -\frac{1}{|\Sigma|_g}\int_\Sigma L_{\Sigma,g}(\widetilde{\xi}_p + f)\, dA_g$ (to do this one can consider scaling the sum of two integral zero test functions with disjoint support). Setting $\eta = \widetilde{L}_{\Sigma,g}(\widetilde{\xi}_p + f + h)$, we then have that $\eta \in C^\infty_c(\Sigma) \cap L^1_T(\Sigma)$ and since $\widetilde{\xi}_p + f + h = \xi_p$ near $p$ we ensure that $\eta \in C_c^\infty(\Sigma) \cap V_p$ also (as $L_{\Sigma,g}(\widetilde{\xi}_p + f + h) = 0$ near $p$).
\end{proof}

\subsection{Induced twisted Jacobi fields and metric perturbations}

The aim of this subsection is to prove the following theorem, which will ultimately allow us to perturb away singularities of isoperimetric regions under appropriate assumptions on the asymptotic growth rate:

\begin{theorem}[Induced twisted Jacobi fields]\label{thm: induced twisted jacobi fields}
    Suppose that for $t \in \mathbb{R}$ we have $(g_j,\Omega_j) \rightarrow (g,\Omega)$ in $\mathcal{P}^{k,\alpha}(t)$ with $\Sigma_j \neq \Sigma$ for all $j \geq 1$, and one of the following cases hold:
    \begin{enumerate}
        \item[(i)] $g_j = g$ for all $j \geq 1$.
        \item[(ii)] $g_j = (1 + c_jf_j)g$ where $f_j \rightarrow f$ in $C^4(M)$, $\nu_{\Sigma,g}(f)$ is not constant on $\Sigma$, and $c_j \rightarrow 0$.
    \end{enumerate}
    Then, there exists some non-zero $\phi \in C^2_{\mathrm{loc}}(\Sigma)$ with $\mathcal{AR}_p(\phi) \geq \gamma^-_1(\mathbf{C}_p)$ for each $p \in \mathrm{Sing}(\Sigma)$, induced by the sequence $(g_j,\Omega_j)$, such that:
    \begin{enumerate}
        \item In case (i) above, $\phi$ is a twisted Jacobi field.

        \item In case (ii), $\widetilde L_{\Sigma,g} \phi = c\left(\nu_{\Sigma,g}(f) - \frac{1}{|\Sigma|_g}\int_\Sigma \nu_{\Sigma,g}(f)\, dA_g\right)$ and $\int_\Sigma \phi \, dA_g = c\left(\frac{n}{2}\int_\Omega f\, dV_g\right)$ for some constant $c \geq 0$. 
    \end{enumerate}
\end{theorem}

\begin{proof} 
    We first consider case (i). For each $p \in \mathrm{Sing}(\Sigma)$ and arbitrary $\gamma_p \in (\gamma^-_2(\mathbf{C}_p),\gamma^-_1(\mathbf{C}_p))$ we choose parameters in order to apply Lemma \ref{lem: dichotomy}. Namely, let $\sigma = \inf_{p \in \mathrm{Sing}(\Sigma)}\mathrm{dist}_{\mathbb{R}}(\gamma_p, \Gamma(\mathbf{C}_p)\cup \{-\frac{n-2}{2}\}) > 0$, $\kappa = 1$, $\Lambda \geq 1$ be sufficiently large so that for each $p \in \mathrm{Sing}(\Sigma)$ we have $\mathbf{C}_p \in \mathcal{C}_\Lambda$, $K > 2$ determined by these choices of $\sigma$ and $\Lambda$ and sufficiently large so that $\Sigma$ can be written in conical coordinates in $B_{\frac{2}{K}}(\mathrm{Sing}(\Sigma))$. We then produce $\delta > 0$ as in the conclusion of Lemma \ref{lem: dichotomy} so that, up to rescaling, we may assume that around each $p \in \mathrm{Sing}(\Sigma)$ the hypothesis (i) and (ii) of Lemma \ref{lem: dichotomy} are satisfied. Now, if property (1) of Lemma \ref{lem: dichotomy} were to occur for some $p \in \mathrm{Sing}(\Sigma)$, then we have some stationary varifold, $V_\infty$, in $\mathbb{R}^{n+1}$ which is asymptotic to $\mathbf{C}_p$ but with $\mathcal{AR}_\infty(V_\infty) < \gamma_p$, which contradicts Lemma \ref{lemma: lower growth rate implies smooth}. Thus property (1) cannot occur and thus only property (2) of Lemma \ref{lem: dichotomy} can occur for each $p \in \mathrm{Sing}(\Sigma)$.

    \bigskip
    
    As $\Omega_j \rightarrow \Omega$ in $L^1$ we have that the varifolds $\Sigma_j \rightarrow \Sigma$ and hence, for sufficiently large $j \geq 1$, there exists a sequence of rescaled graphing radii, $s_j \rightarrow 0$, (i.e.~rescalings of $\tau_j$ in Lemma \ref{lem: dichotomy}) and graphing functions, $v_j \in C^2_{loc}(\Sigma)$ such that $\Vert v_j\Vert_{C^2_*(\Sigma \setminus B_{s_j}(\mathrm{Sing}(\Sigma)))} \leq \delta$ and
    $$\Sigma_j \setminus B_{s_j}(\mathrm{Sing}(\Sigma)) = \mathrm{graph}_{\Sigma}^g(v_j) \setminus B_{s_j}(\mathrm{Sing}(\Sigma))).$$
    Since only property (2) of Lemma \ref{lem: dichotomy} can hold in particular we must have for each $t \in (K^3s_j,1),$ we have
    \begin{equation}\label{eqn: monotonicity of $J$ for induced Jacobi fields}
        J^{\gamma_p}_{K;\Sigma,g}(v_j;K^{-1}t) \leq J^{\gamma_p}_{K;\Sigma,g}(v_j,t),        
    \end{equation}
    which in particular implies that given some $U \subset \subset \Sigma \setminus B_{s_j}(\mathrm{Sing}(\Sigma))$ if $j \geq 1$ is sufficiently large then
    $$\int_U v_j^2\, dA_g \leq C \cdot \int_{\Sigma\setminus B_{K^{-1}}(\mathrm{Sing}(\Sigma))}v_j^2\, dA_g,$$
    where the constant $C > 0$ depends on $U$ but is independent of $j \geq 1$. Set $a_j = \Vert v_j\Vert_{L^2(\Sigma \setminus B_{K^{-1}}(\mathrm{Sing}(\Sigma))}$, which is strictly positive since $\Sigma_j \neq \Sigma$ for all $j \geq 1$ with $a_j \rightarrow 0$ since $\Sigma_j \rightarrow \Sigma$ as $j \rightarrow \infty$, we ensure that the function $\frac{v_j}{a_j}$ is such that $\Vert \frac{v_j}{a_j}\Vert_{L^2(U)} \leq C$.
    
    \bigskip
    
    By Proposition \ref{prop: caccioppoli inequality} applied in the case that $f^\pm = 0$, $u = v_j$, and $v = 0$ we deduce that for $\widetilde{U} \subset \subset U$ there is some constant $C > 0$, depending on $\widetilde{U}, U, \Sigma, g,$ and $\delta$, such that
    \begin{equation}\label{eqn: Caccioppoli for cmc}
        \int_{\widetilde{U}} |\nabla v_j|^2 \leq C \int_{U} v_j^2.
    \end{equation}
    Using this, we can argue similarly to \cite[Theorem 4.2.21]{N24A} to obtain a $W^{2,2}$ bound on the $v_j$ in any subset of $\widetilde{U}$, then a bootstrapping argument using \cite[Proposition 9.11]{GT01} implies that, by (\ref{eqn: Caccioppoli for cmc}), the $W^{2,2}$ control, and the $L^2$ bounds on $\frac{v_j}{a_j}$ above, we ensure the existence of some non-zero function, $\phi \in C^2_{loc}(\Sigma)$, with $\frac{v_j}{a_j} \rightarrow \phi$ in $C^2_{loc}(\Sigma)$ (since $g \in \mathcal{G}^{k,\alpha}$ with $k \geq 4$); the fact that $\phi$ is non-zero follows since $\Vert \frac{v_j}{a_j}\Vert_{L^2(\Sigma \setminus B_{K^{-1}}(\mathrm{Sing}(\Sigma)))} = 1$ for sufficiently large $j \geq 1$. We will show that $\phi$ is a twisted Jacobi field and $\mathcal{AR}_p(\phi) \geq \gamma_1^{-}(\mathbf{C}_p)$ for each $p \in \mathrm{Sing}(\Sigma)$.

    \bigskip

    Denote by $\mathcal{M}^g$ the mean curvature operator with respect to the metric $g$ and $H_j$ and $H$ the mean curvature of $\Sigma_j$ and $\Sigma$ respectively, so that $\mathcal{M}^g(v_j) = H_j$ for each $j \geq 1$, we have for each $\psi \in L^1_T(\Sigma) \cap C^1_{c}(\Sigma)$ that
    \begin{equation}\label{eqn: case (i) difference of mean curvatures}
    \int_\Sigma (\mathcal{M}^g(v_j) - H) \cdot \psi \, dA_g=  (H_j - H) \cdot \int_\Sigma \psi  \, dA_g= 0.
    \end{equation}
    Dividing both sides of (\ref{eqn: case (i) difference of mean curvatures}) by $a_j$ and using the notation for the mean curvature operator as in Subsection \ref{subsec: mean curvature operator} we have that
    $$0 = \int_\Sigma \frac{(\mathcal{M}^g(v_j) - H)}{a_j} \cdot \psi \, dA_g= \int_\Sigma \left(-L_{\Sigma,g} \left(\frac{v_j}{a_j}\right) + \mathrm{div}_{\Sigma}\left(\frac{\mathcal{E}_1(v_j)}{a_j}\right) + r_S^{-1}\cdot\frac{\mathcal{E}_2(v_j)}{a_j}\right) \cdot \psi\, dA_g,$$
    which, after integrating by parts, using the estimates on the error terms, and sending $j \rightarrow \infty$ we have
    $$\int_\Sigma L_{\Sigma,g} \phi \cdot \psi \, dA_g= 0;$$
    thus $L_{\Sigma,g} \phi \, dA_g= \frac{1}{|\Sigma|_g}\int_\Sigma L_{\Sigma,g} \phi\, dA_g$, i.e.~$\widetilde{L}_{\Sigma,g}\phi = 0$ on $\Sigma$.
    
    \bigskip

    We now show that $\mathcal{AR}_p(\phi) \geq \gamma_1^{-}(\mathbf{C}_p)$ for each $p \in \mathrm{Sing}(\Sigma)$. By dividing both sides of (\ref{eqn: monotonicity of $J$ for induced Jacobi fields}) by $a_j^2$ and passing to the limit as $j \rightarrow \infty$ we conclude that (\ref{eqn: monotonicity of $J$ for induced Jacobi fields}) holds with $\phi$ in place of $v_j$. Now, by (\ref{eqn: monotonicity of $J$ for induced Jacobi fields}) for $\phi$ we see that $\limsup_{t \rightarrow 0^+}J^{\gamma_p}_{K;\Sigma,g}(\phi;t) < \infty$ and thus we have $\mathcal{AR}_p(\phi) \geq \gamma_p$ for each $p \in \mathrm{Sing}(\Sigma)$ by Remark \ref{rem: property of J^gamma}. As the choice of $\gamma_p \in (\gamma^{-}_2(\mathbf{C}_p),\gamma_1^-(\mathbf{C}_p))$ was arbitrary, we conclude that $\mathcal{AR}_p(\phi) \geq \gamma_1^{-}(\mathbf{C}_p)$ for each $p \in \mathrm{Sing}(\Sigma)$.

    \bigskip

    We now show that $\int_\Sigma \phi = 0$ which, combined with the above, implies that $\phi$ is a twisted Jacobi field. Note that since $\Omega_j,\Omega \in \mathcal{I}(g,t)$ we have $\mathrm{Vol}_g(\Omega) = \mathrm{Vol}_g(\Omega_j)$ for each $j \geq 1$, thus by denoting $\mathrm{Vol}^\pm_g(\Omega_j \Delta \Omega) = \mathrm{Vol}_g(\Omega \setminus \Omega_j) - \mathrm{Vol}_g(\Omega_j \setminus \Omega)$ we can write
    $$0 = \mathrm{Vol}^\pm_g(\Omega_j \Delta \Omega) = \mathrm{Vol}^\pm_g((\Omega \Delta \Omega_j) \cap B_{s_j}(\mathrm{Sing}(\Sigma))) + \mathrm{Vol}^\pm_g((\Omega \Delta \Omega_j) \cap (\Sigma \setminus B_{s_j}(\mathrm{Sing}(\Sigma)))).$$
    For some constant $C > 0$, depending on $g$, we have by inclusion that
    $$\mathrm{Vol}^\pm_g((\Omega \Delta \Omega_j) \cap B_{s_j}(\mathrm{Sing}(\Sigma))) \leq Cs_j^{8}.$$
    We note that the volume change in $B_{s_j}(\mathrm{Sing}(\Sigma))$ controls $\int_{\Sigma \setminus B_{s_j}(\mathrm{Sing}(\Sigma))} v_j$, since if $x \in \Sigma$ then $\exp_x(v_j(x)\nu_{\Sigma,g}(x)) \in \Sigma_j$ is the closest point on $\Sigma_j$ to $x$, and thus $\mathrm{Vol}^\pm_g((\Omega \Delta \Omega_j) \cap (\Sigma \setminus B_{s_j}(\mathrm{Sing}(\Sigma))))$ is given by
    $$\int_\Sigma\int_0^{\bar{v}_j(x)} \theta(x,t)\,dt\,d\mathcal{H}^n_g(x) = \int_\Sigma\int_0^{\bar{v}_j(x)} (1 - O(t^2))\,dt\,d\mathcal{H}^n_g(x) = \int_\Sigma \bar{v}_j(x)\,d\mathcal{H}^n_g(x)  - O(\Vert \bar{v}_j\Vert_{L^2_g(\Sigma)}^2),$$
    where we denote $\bar{v}_j = v_j \cdot \chi_{\Sigma \setminus B_{s_j}(\mathrm{Sing}(\Sigma))}$, $\theta(x,t)$ the Jacobian of the exponential map based at $x \in \Sigma$ at distance $t$, and we have applied Fubini's theorem. After rearranging and dividing by $a_j$ this gives
    \begin{equation}\label{eqn: integral in the ball}
        \bigg|\int_\Sigma \frac{\bar{v}_j}{a_j} \, dA_g\bigg| \leq C\frac{s_j^{8}}{a_j} + O(\Vert \bar{v}_j\Vert_{L^\infty_g(\Sigma)}).
    \end{equation}
    We will now show that $s_j^{8} \leq Ca_j^{1 + \beta}$ for some constants $C,\beta > 0$ independently of $j \geq 1$ which, by the fact that $\Vert \bar{v}_j\Vert_{L^2_g(\Sigma)} \rightarrow 0$ as $j \rightarrow \infty$ and the above, ensures that $\int_\Sigma \frac{\bar{v}_j}{a_j} \, dA_g\rightarrow 0$; moreover, by showing that $\frac{\bar{v}_j}{a_j} \rightarrow \phi$ in $L^1(\Sigma)$ we will conclude that $\int_\Sigma \phi = 0$ from the convergence in $L^1(\Sigma)$. 
    
    \bigskip
    
    We now deduce the estimate $s_j \leq C a_j^{\frac{1}{1-\widetilde{\gamma}}}$ for each $\widetilde{\gamma} < \gamma_p$ satisfying $\frac{8}{1-\widetilde{\gamma}} > 1$ for some constant $C > 0$, depending on $\widetilde{\gamma}$. By defining $b_j = \Vert v_j\Vert_{L^\infty(\Sigma \setminus B_{2K^{-1}}(\mathrm{Sing}(\Sigma))}$, we will in fact show that $s_j \leq b_j^{\frac{1}{1-\widetilde{\gamma}}}$ where $\widetilde{\gamma}$ is chosen as above; from which the above follows since by the reasoning following (\ref{eqn: Caccioppoli for cmc}) there exists some constant $C > 0$ independent of $j$, but depending on $\delta$, such that $b_j \leq Ca_j$. Assuming for a contradiction that $b_j < s_j^{1-\widetilde{\gamma}}$ for sufficiently large $j$, then by defining $\hat{v}_j(x) = \frac{v_j(s_j x)}{s_j}$ on $A(p;1,\frac{1}{Ks_j})$ we ensure in particular that
    $$\Vert \hat{v}_j\Vert_{L^\infty(A(p;\frac{1}{K^2s_j},\frac{1}{Ks_j}))} \leq \frac{b_j}{s_j} < s_j^{-\widetilde{\gamma}}.$$
    We then see that by Remark \ref{rem: J on cone to J on cmc} that for some constant $C > 0$ we have
    $$J^{\gamma_p}_{K;\Sigma,g}\left(\hat{v}_j;\frac{1}{Ks_j}\right)^2 \leq CJ^{\gamma_p}_{K;\mathbf{C}_p}\left(\hat{v}_j;\frac{1}{Ks_j}\right)^2 = C\int_{\mathbf{A}\left(\frac{1}{K^2s_j},\frac{1}{Ks_j}\right)} \hat{v}_j^2 |x|^{-n-2\gamma_p} \, d \Vert \mathbf{C}_p\Vert  < Cs_j^{2(\gamma_p - \widetilde{\gamma})},$$
    and thus $J^{\gamma_p}_{K;\Sigma,g}\left(\hat{v}_j;\frac{1}{Ks_j}\right) \rightarrow 0$ as $j \rightarrow \infty$. However, by writing (\ref{eqn: monotonicity of $J$ for induced Jacobi fields}) for $\hat{v}_j$ we deduce that $J^{\gamma_p}_{K;\Sigma,g}(\hat{v}_j;K^l)$ is increasing for integers $3 \leq l \leq \log_K(\frac{1}{s_j}) - 1$ and hence
    $$J^{\gamma_p}_{K;\Sigma,g}\left(\hat{v}_j;K^3\right) \leq J^{\gamma_p}_{K;\Sigma,g}\left(\hat{v}_j;\frac{1}{Ks_j}\right).$$
    Since $\hat{v}_j \rightarrow h \in C^2_{loc}(\mathbf{C}_p \setminus B_1)$ in $C^2_\mathrm{loc}$ where $h$ is the graph of a stationary varifold, $V$, which is asymptotic to but not equal to $\mathbf{C}_p$ (obtained as a limit of $s_j^{-1}\Sigma_j$),
    for all large $j$ we see that
    $$J^{\gamma_p}_{K;\Sigma,g}\left(\hat{v}_j;K^3\right) > \frac{1}{2}J^{\gamma_p}_{K;\mathbf{C}_p}\left(h;K^3\right)  > 0$$
    where the right hand term is strictly positive since $h \neq 0$ (since $V \neq \mathbf{C}_p$). This is a contradiction since we saw that $J^{\gamma_p}_{K;\Sigma,g}\left(\hat{v}_j;\frac{1}{Ks_j}\right) \rightarrow 0$ as $j \rightarrow \infty$. Thus, for each $\widetilde{\gamma} < \gamma$ such that $\frac{8}{1-\widetilde{\gamma}} > 1$, there is a constant $C > 0$ such that $s_j \leq C a_j^{\frac{1}{1-\widetilde{\gamma}}}$ for each $j$; we then immediately see that $\int_{\Sigma} \frac{\bar{v}_j}{a_j} \rightarrow 0$.

    \bigskip

    We now establish the existence of an integrable dominating function for $\frac{\bar{v}_j}{a_j}$ which, combined with the above and the dominated convergence theorem, will allow us to show that $\int_{\Sigma} \phi = 0$. Similarly to the above, we have that $J^{\gamma_p}_{K;\Sigma,g}(v_j;K^l)$ is increasing for integers $\log_K(s_j) - 1 \leq l \leq -1$; for this range of $l$ we thus have that
    \begin{equation}\label{eqn: monotonicity for dominating function}
        J^{\gamma_p}_{K;\Sigma,g}\left(\frac{v_j}{a_j};K^l\right) \leq J^{\gamma_p}_{K;\Sigma,g} \left(\frac{v_j}{a_j};\frac{1}{K}\right) \rightarrow J^{\gamma_p}_{K;\Sigma,g}\left(\phi;\frac{1}{K}\right) \text{ as } j \rightarrow \infty.
    \end{equation}
    Setting $v_j^{(r)}(x) = \frac{v_j(rx)}{r}$ on $A(p;\frac{s_j}{r},\frac{1}{Kr})$ for each $r \in (Ks_j,\frac{1}{K})$, which is the graph of the varifold $\frac{1}{r}\Sigma_j$ over $\Sigma$, again by the reasoning following (\ref{eqn: Caccioppoli for cmc}) we deduce that (since $A(p;\frac{1}{K},1) \subset A(p;\frac{s_j}{r},\frac{1}{Kr})$) we have some constant $C > 0$, depending on $\delta$, such that we have
    $$\Vert v_j^{(r)}\Vert_{L^\infty(A(p;\frac{1}{K},1))} \leq C\Vert v_j^{(r)}\Vert_{L^2(A(p;\frac{1}{2K},1))};$$
    relying in particular on the fact that $\Vert v_j\Vert_{C^2_*(\Sigma \setminus B_{s_j}(\mathrm{Sing}(\Sigma)))} \leq \delta$. After scaling back we see that this gives 
    $$\Vert v_j\Vert_{L^\infty(A(p;\frac{r}{K},r))} \leq Cr^{-\frac{n}{2}}\Vert v_j\Vert_{L^2(A(p;\frac{r}{2K},r))}.$$
    For sufficiently large $j \geq 1$ we consider $r = K^l$ for $\log_K(s_j) - 1 \leq l \leq -1$, which by (\ref{eqn: monotonicity for dominating function}) gives that for some constant $C > 0 $ independent of $j \geq 1$, but depending on $K$, we have
    \begin{align*}
        \bigg|\bigg|\frac{v_j}{a_j}\bigg|\bigg|_{L^2(A(p;K^{l-1},K^l))} \leq C \cdot J^{\gamma_p}_{K;\Sigma,g}\left(\frac{v_j}{a_j};K^l\right)  (K^{l})^{\gamma + \frac{n}{2}}  &\leq C\cdot J^{\gamma_p}_{K;\Sigma,g}\left(\phi;\frac{1}{K}\right)  (K^{l})^{\gamma + \frac{n}{2}}   = C(K^{l})^{\gamma + \frac{n}{2}}.
    \end{align*}
    Combining the above two estimates we deduce that 
    $$\bigg|\bigg|\frac{v_j}{a_j}\bigg|\bigg|_{L^\infty(A(p;K^{l-1},K^l))} \leq C \cdot (K^l)^\gamma,$$
    which shows that, since $x \in (K^{l-1},K^l)$ for some $l$, for each $x \in A(p;s_j,\frac{1}{K})$ we have
    $$\left|\frac{v_j}{a_j}\right|(x) \leq C |x|^\gamma,$$
    which is integrable (since $\gamma > -7$). Combining this integrable dominating function with the fact that $\int_\Sigma \frac{\bar{v}_j}{a_j}\, dA_g \rightarrow 0$ and $\frac{\bar{v}_j}{a_j} \rightarrow \phi$ pointwise, we see by the dominated convergence theorem that $\int_\Sigma \phi = 0$; this concludes case (i) as $\phi$ is then a twisted Jacobi field with $\mathcal{AR}_p(\phi) \geq \gamma_1^-(\mathbf{C}_p)$ for each $p \in \mathrm{Sing}(\Sigma)$.

    \bigskip
    
    The proof of case (ii) is similar but we alter our choice of $\kappa > 0$ when applying Lemma \ref{lem: dichotomy}, as we now explain. First fix a sufficiently large open set $U \subset \subset \Sigma$ such that $\nu_{\Sigma,g}(f)$ is not identically zero on $U$, which we can do by the assumption on $f$ and since $\Sigma \setminus B_{K^{-1}}(\mathrm{Sing}(\Sigma)) \subset U$, we now show that $\liminf_{j \rightarrow \infty} \frac{a_j}{c_j} > 0$ where we now instead set $a_j = \Vert v_j\Vert_{L^2(U)}$.

    \bigskip

    If the claim fails, then up to a subsequence (not relabelled) we have $\frac{a_j}{c_j} \rightarrow 0$; we will now get a contradiction from the assumption that $\nu_{\Sigma,g}(f)$ changes sign. Denoting $\mathcal{M}^{g_j}$ the mean curvature operator with respect to the metric $g_j$ and $H_\Sigma$ the constant mean curvature of $\Sigma$, we compute that, similarly to case (i), if $\Vert v_j\Vert_{C^2_*(U)} \leq \delta$ and $[c_jf_j]_{x,C^2_*} \leq \delta$ for each $x \in \Sigma$, where $\delta > 0$ is the dimensional constant from \cite[Appendix B]{LW25} (see also Subsection \ref{subsec: mean curvature operator}) then
    $$\mathcal{M}^{g_j}(v_j) - H = -L_{\Sigma,g}(v_j) + \frac{n}{2}\nu_{\Sigma,g}(c_jf_j) + \mathrm{div}_\Sigma(\mathcal{E}_1(v_j)) + r_S^{-1}\mathcal{E}_2(v_j).$$
    Now for each $j \geq 1$ and $\psi \in L^1_T(\Sigma) \cap C^1_c(\Sigma)$ we have, after dividing by $c_j$, that
    $$ 0 = \int_{\Sigma} \frac{(\mathcal{M}^{g_j}(v_j) - H)}{c_j} \cdot \psi \, dA_g= \int_\Sigma \left(-L_{\Sigma,g}\left(\frac{v_j}{c_j}\right) + \frac{n}{2}\nu_{\Sigma,g}(f_j) + \mathrm{div}_\Sigma\left(\frac{\mathcal{E}_1(v_j)}{c_j}\right) + r_S^{-1}\frac{\mathcal{E}_2(v_j)}{c_j}\right) \cdot \psi\, dA_g.$$
    By applying Proposition \ref{prop: caccioppoli inequality}, in the case that $f^+ = c_jf_j$, $f^-= 0$, $u = v_j$, and $v = 0$, along with the reasoning following (\ref{eqn: Caccioppoli for cmc}) above, we see that as $j \rightarrow \infty$ we have $\frac{v_j}{c_j} \rightarrow 0$ in $W^{1,2}_g(U)$ and so $\int_\Sigma \nu_{\Sigma,g}(f)\cdot \psi \, dA_g = 0$; thus $\nu_{\Sigma,g}(f)$ is constant, giving a contradiction. 
    
    \bigskip

    Since we have $\liminf_{j \rightarrow \infty} \frac{a_j}{c_j} > 0$ we know that there is some $\kappa > 0$ such that up to a subsequence (not relabelled) we have $\liminf_{j \rightarrow \infty}\frac{c_j\Vert f_j\Vert_{C^4(M)}}{a_j} \leq \frac{1}{\kappa}$ and thus there exists some constant $c \geq 0$ such that $\frac{c_j}{a_j} \rightarrow c$. We now use each of the same parameters as in case (i) above but with this choice of $\kappa > 0$ and apply Lemma \ref{lem: dichotomy}. By the same reasoning as for the claim just established, instead dividing now by $a_j$, we deduce that there exists some $\phi \in C^2_{loc}(\Sigma)$ with $\frac{v_j}{a_j} \rightarrow \phi$ in $C^2_{loc}(\Sigma)$ such that $\mathcal{AR}_p(\phi) \geq \gamma_1^-(\mathbf{C}_p)$ for each $p \in \mathrm{Sing}(\Sigma)$ with
    $$\widetilde{L}_{\Sigma,g}\phi = c\left(\nu_{\Sigma,g}(f) - \frac{1}{|\Sigma|_g}\int_\Sigma \nu_{\Sigma,g}(f)\, dA_g\right).$$
    Since we have $\mathrm{Vol}_g(\Omega) = \mathrm{Vol}_{g_j}(\Omega_j)$ we have
    $$\mathrm{Vol}_{g_j}(\Omega \cap \Omega_j) - \mathrm{Vol}_g(\Omega \cap \Omega_j) = \mathrm{Vol}_g(\Omega \setminus \Omega_j) - \mathrm{Vol}_{g_j}(\Omega_j \setminus \Omega);$$
    we now estimate each side of the above expression individually (notice in case 1 above that since $g = g_j$ the left hand side of the expression above is identically zero). We will show that
    \begin{equation}\label{eqn: limit of intersection}
        \frac{\mathrm{Vol}_{g_j}(\Omega \cap \Omega_j) - \mathrm{Vol}_g(\Omega \cap \Omega_j)}{a_j} \rightarrow c\left(\frac{n}{2}\int_{\Omega} f \, dV_g\right)
    \end{equation}
    and that
    \begin{equation}\label{eqn: limit of difference}
        \frac{\mathrm{Vol}_g(\Omega \setminus \Omega_j) - \mathrm{Vol}_{g_j}(\Omega_j \setminus \Omega)}{a_j} \rightarrow \int_{\Sigma}\phi \, dA_g.
    \end{equation}

    \bigskip

    First notice that we can write
    $$\mathrm{Vol}_{g_j}(\Omega \cap \Omega_j) - \mathrm{Vol}_g(\Omega \cap \Omega_j) = \int_{\Omega_j\cap\Omega} \left( (1+c_jf_j)^{n/2}-1\right) \, d\mathcal{H}^{n+1}_g(x) =  \int_{\Omega_j\cap\Omega} \left(\frac{n}{2} c_jf_j + O(c_j^2)\right) \, d\mathcal{H}^{n+1}_g(x)$$
    and thus by dividing by $a_j$ we see that (\ref{eqn: limit of intersection}) holds. Now, similarly to the derivation of (\ref{eqn: integral in the ball}), we see that if we write 
    $$\mathrm{Vol}^\pm((\Omega \Delta \Omega_j) \setminus B_{s_j}(\mathrm{Sing}(\Sigma)) = \mathrm{Vol}_g(\Omega \setminus (\Omega_j \cup B_{s_j}(\mathrm{Sing}(\Sigma))) - \mathrm{Vol}_{g_j}(\Omega_j \setminus (\Omega \cup B_{s_j}(\mathrm{Sing}(\Sigma)))$$ then we have that 
    $$|\mathrm{Vol}^\pm((\Omega \Delta \Omega_j) \setminus B_{s_j}(\mathrm{Sing}(\Sigma))| \leq Cs_j^8.$$
    Thus, we can compute similarly to case (i) that 
    $$\left| \int_\Sigma \frac{\bar{v}_j}{a_j} \, dA_g \right| \leq C\frac{s_j^8}{a_j} + \frac{O(c_j^2)}{a_j}+ O(\Vert \bar{v}_j\Vert_{L^2_g(\Sigma)}) .$$
    Then, using the fact that $\Vert \bar{v}_j\Vert_{L^2_g(\Sigma)} \rightarrow 0$ as $j \rightarrow \infty$, $\liminf_{j \rightarrow \infty} \frac{c_j}{a_j} < \infty$, and by arguing identically as in case (i) (i.e.~controlling $s_j$ by $Ca_j^{1+\beta}$ for some $\beta > 0$ and establishing the existence of a dominating function) we deduce that $\int_\Sigma \frac{\bar{v}_j}{a_j} \, dA_g \rightarrow \int_\Sigma \phi \, dA_g$ and so (\ref{eqn: limit of difference}) holds also; thus $\int_\Sigma \phi \, dA_g = c\left(\frac{n}{2}\int_\Omega f \, dV_g\right)$ as desired.
    \end{proof}

\begin{corollary}\label{corollary: perturbing singularities for not slow growth}
    With the same assumptions, if $\mathcal{AR}_p(\phi) < \gamma^+_2(\mathbf{C}_p)$ then there is an open neighbourhood, $U_p \subset M$, of $p$ such that for infinitely many $j \geq 1$ we have $\mathrm{Sing}(\Sigma_j) \cap U_p = \emptyset$.
\end{corollary}

\begin{proof}
    Since $\mathcal{AR}_p(\phi) \geq \gamma^-_1(\mathbf{C}_p)$, by Lemma \ref{lemma: properties of slow growth} part 1 we have $\mathcal{AR}_p(\phi) = \gamma^\pm_1(\mathbf{C}_p)$. Fixing some $\gamma \in (\gamma^+_1(\mathbf{C}_p),\gamma^+_2(\mathbf{C}_p))$ we know that $\mathcal{AR}_p(\phi) \leq\gamma^+_1(\mathbf{C}_p) < \gamma$ and hence by Remark \ref{rem: property of J^gamma} property (2) of Lemma \ref{lem: dichotomy} cannot occur for this choice of $\gamma$. Thus property (1) of Lemma \ref{lem: dichotomy} must occur and hence each stationary varifold, $V_\infty$, as in the conclusion of property (2) must satisfy $\mathcal{AR}_\infty(V_\infty) < \gamma$. Again by Lemma \ref{lemma: properties of slow growth} part 1,  we conclude that $\mathcal{AR}_\infty(V_\infty) \leq \gamma^+_1(\mathbf{C}_p)$ which by application of Lemma \ref{lemma: lower growth rate implies smooth} implies that $V_\infty$ is smooth; hence by Allard's theorem the $\Sigma_j$ are regular in a neighbourhood of $p$ for sufficiently large $j \geq 1$.
\end{proof}

\begin{proposition}[Perturbation of singularities with fast growth]\label{prop: perturbation of singularities with fast growth}
    Suppose that $\Sigma$ is semi-nondegenerate with $\mathrm{Sing}(\Sigma) \neq \emptyset$ and let $\mathcal{V}^{k,\alpha}$ be chosen as in Proposition \ref{prop: perturbation functions open and dense}, $(g_j,\Omega_j)\rightarrow (g,\Omega)$ in $\mathcal{P}^{k,\alpha}(t)$ for some $t \in \mathbb{R}$ with $\Sigma_j \neq \Sigma$ for all $j \geq 1$, and one of the following cases hold:
    \begin{enumerate}
        \item[(i)] $g_j = g$ for all $j \geq 1$.
        \item[(ii)] $g_j = (1 + c_jf_j)g$ where $f_j \rightarrow f$ in $C^4(M)$ for $f \in \mathcal{V}^{k,\alpha}$ with $\int_{\Omega}f \, dV_g= 0$, $\nu_{\Sigma,g}(f)$ is not constant on $\Sigma$, and $c_j \rightarrow 0$.
    \end{enumerate}
    
    Then there exists some $p \in \mathrm{Sing}(\Sigma)$, and an open neighbourhood, $U_p \subset M$, of $p$, both depending on the sequence, such that for infinitely many $j \geq 1$ we have $\mathrm{Sing}(\Sigma_j) \cap U_p = \emptyset$.
\end{proposition}

\begin{proof}
    In case (i), since $\Sigma$ is semi-nondegenerate we must have that any induced twisted Jacobi field, $\phi$, as in Theorem \ref{thm: induced twisted jacobi fields} cannot be of slow growth. In case (ii), as $f \in \mathcal{V}^{k,\alpha}$ and $\int_{\Omega}f = 0$, any induced Jacobi field, $\phi$, as in Theorem \ref{thm: induced twisted jacobi fields} is not of slow growth by Proposition \ref{prop: perturbation functions open and dense}. In either case, since $\phi$ is not of slow growth, there exists some $p \in \mathrm{Sing}(\Sigma)$ such that $\mathcal{AR}_p(\phi) < \gamma_2^+(\mathbf{C}_p)$; hence by Corollary \ref{corollary: perturbing singularities for not slow growth} there is a neighbourhood of $p$ in which $\Sigma_j$ is regular for sufficiently large $j \geq 1$.
\end{proof}

\section{Bumpy metric volume pairs}\label{sec: bumpy metric volume pairs}

In this section we will show that semi-nondegeneracy is a generic property for isoperimetric regions in dimension eight.

\subsection{Pseudo-neighbourhoods and three compactness lemmas}\label{subsec: pseudo-neighbourhoods and three compactness lemmas}

We first introduce a suitable notion to appropriately decompose the space of triples, and define the following:

\begin{definition}[Pseudo-neighbourhoods]\label{def: pseudo-neighbourhoods}
    Given $(g,t, \Omega) \in \mathcal{T}^{k,\alpha}$, $\Lambda \geq 1$, and $\delta > 0$ we define the \textbf{pseudo-neighbourhood}, denoted $\mathcal{L}^{k,\alpha}(g,t,\Omega;\Lambda,\delta)$, to be the set of all triples $(\bar{g},\bar{t},\bar{\Omega}) \in \mathcal{T}^{k,\alpha}$ such that:
    \begin{itemize}
        \item $\Vert \bar{g}\Vert_{C^{k,\alpha}} \leq \Lambda$.
        
        \item $\Vert g - \bar{g}\Vert_{C^{k-1,\alpha}} \leq \delta$, $|t - \bar{t}| \leq \delta$, and $|\Omega \Delta \bar{\Omega}|_g \leq \delta $.

        \item For each $p \in \mathrm{Sing}(\Sigma)$ there exists $\bar{p} \in \mathrm{Sing}(\bar{\Sigma}) \cap \mathrm{inj}(\Sigma,g)$ such that $\theta_{|\Sigma|_g}(p) = \theta_{|\bar{\Sigma}|_{\bar{g}}}(\bar{p})$.
    \end{itemize}

    We endow these spaces with the topology induced by the $C^{k - 1, \alpha}$ topology in the first factor, the standard topology on $\mathbb{R}$ in the second factor, and in the $L^1$ topology in the last factor. We denote by $\Pi : \mathcal{L}^{k,\alpha}(g,t,\Omega;\Lambda,\delta) \rightarrow \mathcal{G}^{k,\alpha} \times \mathbb{R}$ the smooth projection map taking $(\bar{g},\bar{t},\bar{\Omega}) \in \mathcal{L}^{k,\alpha}(g,t,\Omega;\Lambda,\delta)$ to $(\bar{g},\bar{t}) \in \mathcal{G}^{k,\alpha} \times \mathbb{R}$.
\end{definition}

Note that, in general, $\mathcal{L}^{k, \alpha}(g, t, \Omega; \Lambda, \delta)$ is not actual neighbourhood of $(g, t, \Omega)$ in $\mathcal{T}^{k,\alpha}$.  We now establish three compactness results for pseudo-neighbourhoods that will be utilised numerous times throughout Subsections \ref{subsec: sard-smale} and \ref{subsec: generic semi-nondegeneracy} in order to ultimately establish that semi-nondegeneracy is a generic property for metric volume pairs:

\begin{lemma}[Compactness of pseudo-neighbourhoods] \label{lemma: compactness of pseudo neighbourhoods}
    For each $(g, t, \Omega) \in \mathcal{T}^{k, \alpha}$ and $\Lambda \geq 1$, there is $\delta_0 \in (0, 1)$, depending on $g, t, \Omega, \Lambda, k$, and $\alpha$, such that for every $\delta \in (0, \delta_0)$ the space  $\mathcal{L}^{k, \alpha}(g, t, \Omega; \Lambda, \delta)$ is compact.
\end{lemma}

\begin{proof}
    Assume for a contradiction that there exists a sequence $\delta_j \rightarrow 0$ so that for each $j \geq 1$ we have a sequence, $\{(g_{j}^{i}, t_{j}^{i}, \Omega_{j}^{i})\}_{i \geq 1} \subset \mathcal{L}^{k, \alpha}(g, t, \Omega; \Lambda, \delta_j)$, with no convergent subsequence. By Arzela--Ascoli and Lemma \ref{lemma: compactness of isoperimetric regions}, for each $j \geq 1$ there exists $(g_j,t_j,\Omega_j) \in \mathcal{T}^{k,\alpha} \setminus \mathcal{L}^{k,\alpha}(g,t,\Omega;\Lambda,\delta_j)$ such that, up to a subsequence (not relabelled), we have $g_j^i \rightarrow g_j$ in $C^{k-1,\alpha}(M)$, $t_j^i \rightarrow t_j$, and $|\Sigma_j^i|_{g^i_j} \rightarrow |\Sigma_j|_{g_j}$ as varifolds as $i \rightarrow \infty$. By the definition of $\mathcal{L}^{k,\alpha}(g,t,\Omega;\Lambda,\delta_j)$, Allard's theorem, and the upper semi-continuity of denisty, we ensure that for each $i,j \geq 1$ there are $p^i_j \in \mathrm{Sing}(\Sigma_j^i)$ and $p_j \in \mathrm{Sing}(\Sigma_j)$ such that both $\theta_{|\Sigma_j^i|_{g_j^i}}(p^i_j) < \theta_{|\Sigma_j|_{g_j}}(p_j)$ and $\theta_{|\Sigma_j^i|_{g_j^i}}(p^i_j) = \theta_{|\Sigma|_{g_j}}(p)$ for some $p \in \mathrm{Sing}(\Sigma)$. Since $\delta_j \rightarrow 0$ we ensure that $p_j \rightarrow p$ as $j \rightarrow \infty$ which implies that as the densities of stable minimal hypercones are discrete (as mentioned in Subsection \ref{subsec: notation}) we have that $\theta_{|\Sigma|_g}(p) < \limsup_{j \rightarrow \infty} \theta_{|\Sigma_j|_{g_j}}(p_j)$, contradicting the upper semi-continuity of density.
\end{proof}

\begin{lemma}[Compactness of twisted Jacobi fields]\label{lemma: compactness of twisted JF}
Let $(g,t,\Omega) \in \mathcal{T}^{k,\alpha}$ and $\delta_0 > 0$ be as in Lemma \ref{lemma: compactness of pseudo neighbourhoods}. Suppose that there is a sequence $\{(g_j, t_j, \Omega_j)\}_{j \geq 1} \subset \mathcal{L}^{k, \alpha}(g, t, \Omega; \Lambda, \delta_0)$ such that:
\begin{enumerate}
    \item[(i)] $(g_j,t_j, \Omega_j) \rightarrow (g_\infty, t_\infty,\Omega_\infty)$ in $\mathcal{L}^{k, \alpha}(g, \Omega, t; \Lambda, \delta)$. 
    \item[(ii)] For each $j \geq 1$ there exist non-zero twisted Jacobi field of slow growth, $u_j \in \mathrm{Ker}^{+}\widetilde{L}_{\Sigma_j, g_j}$, on $\Sigma_j$, such that $\Vert u_j \Vert_{L^2_{g_j}(\Sigma_j)} = 1$. 
\end{enumerate}
Then, up to a subsequence (not relabelled), the $u_j$ converge in $C_{\mathrm{loc}}^{2}(\Sigma_\infty)$ to a twisted Jacobi field of slow growth, $u_\infty \in \mathrm{Ker}^{+}\widetilde{L}_{\Sigma_{\infty}, g_\infty}$, such that $\Vert u_\infty \Vert_{L^2_{g_\infty}(\Sigma_\infty)} = 1$. In particular, there exists $\kappa_1 = \kappa_1(g, t, \Omega, \Lambda) \in (0, \delta_0)$ such that for every $(g^\prime,t', \Omega^\prime) \in \mathcal{L}^{k, \alpha}(g, t, \Omega; \Lambda, \kappa_1)$, we have 
\begin{equation}\label{eq: top kernel}
    \dim(\mathrm{ker}^{+} \widetilde{L}_{\Sigma^\prime, g^\prime}) \leq \dim(\mathrm{ker}^{+}\widetilde{L}_{\Sigma, g}). 
\end{equation}
\end{lemma}

\begin{proof}
    For fixed $r > 0$ and $j \geq 1$ sufficiently large we have that there is some constant $C > 0$, depending on $\Sigma_\infty, g_\infty,$ and $r$, so that in particular $\Vert u_j\Vert_{W^{1,2}_g(\Sigma_\infty \setminus B^{g_\infty}_r(\mathrm{Sing}(\Sigma_\infty))} \leq C$ and thus up to a subsequence (not relabelled) there is some $u_\infty \in W^{1,2}_{g,\mathrm{loc}}(\Sigma_\infty)$ weak, and hence by by Remark \ref{rem: higher regularity for weak solutions} strong, solution  solving $\widetilde{L}_{\Sigma,g}u_\infty = 0$ with $u_j \rightarrow u_\infty$ in $C^2_\mathrm{loc}(\Sigma_\infty)$; so in particular $u_\infty \in \mathrm{Ker}\widetilde{L}_{\Sigma_\infty,g_\infty}$.
    
    \bigskip
    
    In order to show that $u_\infty$ is of slow growth with $\Vert u_\infty\Vert_{L^2_{g_\infty}(\Sigma_\infty)} = 1$, we note that for fixed $\sigma \in \left(0,\frac{\gamma_\mathrm{gap}(\Lambda)}{2}\right)$ and $K > 2$ chosen as in Lemma \ref{lemma: J monotone for divergence form}, by the same argument as in the proof of \cite[Lemma 8.2]{LW25}, utilising conical coordinates on the $\Sigma_j$ and $\Sigma_\infty$ and replacing the use of \cite[Corollary 6.2]{LW25} with Lemma \ref{lemma: J monotone for divergence form}, we deduce that there is some constant $C > 0$, depending on $\Sigma_\infty$ and $g_\infty$, such that whenever $p_j \in \mathrm{Sing}(\Sigma_j)$ and $\tau > 0$ is sufficiently small we have
    \begin{equation} \label{eqn: L2 nonconcentration for twisted jacobi}
        \Vert u_j \Vert_{L^2(B^{g_j}(p_j, \tau))} \leq C \tau^{n/2 + \gamma_{2}^{+}(\mathbf{C}_{p_j} \Sigma_j) - \sigma}.
    \end{equation}
    Then whenever $p_j \rightarrow p_\infty \in \mathrm{Sing}(
    \Sigma_\infty)$, by applying Allard's theorem and the varifold convergence implied by Lemma \ref{lemma: compactness of isoperimetric regions} part 2, we see that up to a subsequence (not relabelled) $\gamma_{2}^{+}(\mathbf{C}_{p_j}\Sigma_j) \rightarrow \gamma_{2}^{+}(\mathbf{C}_{p}\Sigma_\infty)$. Since Remark \ref{rem: property of J^gamma} implies
    $\gamma_{2}^{+}(\mathbf{C}_{p_j}\partial^{\ast}\Omega_j) - \sigma > \gamma_{1}^{+}(\mathbf{C}_{p_j}\partial^{\ast}\Omega_j) \geq - \frac{n - 2}{2}$, by (\ref{eqn: L2 nonconcentration for twisted jacobi})
    we ensure that $u_\infty$ is of slow growth with $\Vert u_\infty\Vert_{L^2_{g_\infty}(\Sigma_\infty)} = 1$.
    
    \bigskip
    
    Since $\int_{\Sigma_j} u_j = 0$ and $\Vert u_j\Vert_{L^2_{g_j}(\Sigma_j)} = 1$ for each $j \geq 1$, we have, for each $s > 0$, that
    $$\left\vert \int_{\Sigma_j \setminus B_{s}^{g_j}(\mathrm{Sing}(\Sigma_j))} u_j \, dA_{g_j} \right\vert = \left\vert \int_{B_s^{g_j}(\mathrm{Sing}(\Sigma_j))} u_j \, dA_{g_j}\right\vert \leq \Vert u_j \Vert_{L^2_{g_j}(\Sigma_j)} \vert B_s(p) \cap \Sigma_j \vert^{1/2}_{g_j} \leq C s^{7/2},$$
    where $C > 0$, independent of $j \geq 1$, arises from the monotonicity formula. Since $u_j \rightarrow u_\infty$ in $C^2_{\mathrm{loc}}(\Sigma)$ we may apply the dominated convergence theorem (e.g.~with dominating function $|u_\infty| + 1$ for large $j \geq 1$) to see that for each $s >  0$ we have 
    $$\left\vert \int_{\Sigma \setminus B_{s}(\mathrm{Sing}(\Sigma_j))} u_\infty \, dA_{g}\right\vert  \leq Cs^{\frac{7}{2}};$$
    sending $s \rightarrow 0$ we conclude that $\int_{\Sigma_\infty} u_\infty\, dA_{g} = 0$ and hence $u_\infty \in \mathrm{Ker}^+\widetilde{L}_{\Sigma_\infty,g}$ as desired.

    \bigskip

    For the final statement we argue by contradiction and assume that there is a sequence, $\{(g_j,t_j,\Omega_j)\} \subset \mathcal{L}^{k,\alpha}(g,t,\Omega;\Lambda,\delta_0)$ such that $(g_j,t_j,\Omega_j) \rightarrow (g,t,\Omega)$ in $\mathcal{L}^{k,\alpha}(g,t,\Omega;\Lambda,\delta_0)$ but with $\mathrm{dim}(\mathrm{Ker}^+\widetilde{L}_{\Sigma_j,g_j}) >\mathrm{dim}(\mathrm{Ker}^+\widetilde{L}_{\Sigma,g})$ for all $j \geq 1$. Letting $I = \mathrm{dim}(\mathrm{Ker}^+\widetilde{L}_{\Sigma,g})$ and for each $j \geq 1$ choosing $L^2$ orthonormal $u_j^1,\dots,u^{I+1} \in \mathrm{Ker}^+\widetilde{L}_{\Sigma_j,g_j}$, we can apply the first part of the lemma established above to see that up to a subsequence (not relabelled) we have $u_j^i \rightarrow u^i$ for some non-zero $u^i \in \mathrm{Ker}^+\widetilde{L}_{\Sigma,g}$ for each $i = 1, \dots, I+1$. Moreover, by (\ref{eqn: L2 nonconcentration for twisted jacobi}) we ensure that that the $u^1,\dots,u^{I+1}$ are also $L^2$ orthonormal, contradicting the assumption that $I = \mathrm{dim}(\mathrm{Ker}^+\widetilde{L}_{\Sigma,g})$.
\end{proof}

In order to state the final compactness lemma of this subsection we introduce notation (similar to that of \cite[Section 7]{LW25}) in order to define global graphs of isoperimetric regions over one another. Given $g,g^1,g^2 \in \mathcal{G}^{k,\alpha}$, $t \in \mathbb{R}$, $\Omega^1 \in \mathcal{I}(g^1,t)$, and $\Omega^2 \in \mathcal{I}(g^2,t)$, then we define the \textbf{graphing function}, $G^{\Sigma^2}_{\Sigma^1,g^1} \in L^\infty(\Sigma_1)$, of $\Sigma_2$ over $\Sigma_1$ with respect to the metric $g$ by setting for each $x \in \Sigma_1$
$$G^{\Sigma^2}_{\Sigma^1,g^1}(x) = \begin{cases}
    \sup\{ t \geq 0 \, | \, \exp^g_x(t\nu_{\Sigma^1,g^1}(x)) \in \Omega^1 \text{ for all } s \in [0,t]\} & \text{for } x \in \Omega^1\\
    \inf\{ t \leq 0 \, | \, \exp^g_x(t\nu_{\Sigma^1,g^1}(x)) \in \Omega^1 \text{ for all } s \in [t,0]\} & \text{for } x \in \Sigma \setminus \Omega^1
\end{cases},$$
where $\nu_{\Sigma^1,g^1}$ is the outward pointing unit normal to $\Sigma^1$ with respect to the metric $g^1$. Also, given $g \in \mathcal{G}^{k,\alpha}$, $t \in \mathbb{R}$, $(g,t,\Omega) \in \mathcal{T}^{k,\alpha}$, $\Lambda > 1$, $\delta > 0$ sufficiently small, and $(g^1,\bar{t},\Omega^1),(g^2,\bar{t},\Omega^2) \in \mathcal{L}^{k,\alpha}(g,t,\Omega;\Lambda,\delta)$ we define the following semi-metric
    $$\mathbf{D}[(g^1,\bar{t},\Omega^1),(g^2,\bar{t},\Omega^2)] = \Vert g^1-g^2\Vert_{L^\infty(M)} + \Vert G^{\Sigma^1}_{\Sigma^2,g^2}\Vert_{L^2_{g^2}(\Sigma^2)}+\Vert G^{\Sigma^2}_{\Sigma^1, g^1}\Vert_{L^1_{g^1}(\Sigma_1)}.$$
With this notation, we establish the following analogue of Theorem \ref{thm: induced twisted jacobi fields} for sequences of pairs in pseudo-neighbourhoods, which plays a key role in establishing the results of the next subsection:

\begin{lemma}[Compactness for pairs]\label{lemma: compactness for pairs}
    Let $(g,t,\Omega) \in \mathcal{T}^{k,\alpha}$, $\Lambda \geq 1$, $t_j \rightarrow t$, $r \in (0,\mathrm{inj}(\Sigma,g))$ and $\delta_0 > 0$ be as in Lemma \ref{lemma: compactness of pseudo neighbourhoods}. Suppose that:

    \begin{itemize}
        \item[(i)] For $i = 1,2$, there are distinct sequences $\{(g_j^i,t_j,\Omega_j^i)\}_{j \geq 1} \subset \mathcal{L}^{k,\alpha}(g,t,\Omega;\Lambda,\delta_0)$ and a further (not necessarily distinct) sequence $ \{(\bar{g}_j,t_j,\bar{\Omega}_j)\}_{j \geq 1} \subset \mathcal{L}^{k,\alpha}(g,t,\Omega;\Lambda,\delta_0)$ each of which converge to $(g,t,\Omega)$ in the pseudo-neighbourhood topology.

        \item[(ii)] For $i = 1,2$, we have $g_j^i = (1 + c_jf_j^i)\bar{g}_j$ where $f_j^i \in C^{k,\alpha}(M)$ are uniformly bounded in norm with $\mathrm{spt}
        (f_j^i) \subset M \setminus B_r^g(\mathrm{Sing}(\Sigma))$, $(f_j^1 - f_j^2) \rightarrow f_\infty$ in $C^4(M)$ with $\nu_{\Sigma,g}(f_\infty)$ is not constant on $\Sigma$ unless $f^1_j = f^2_j$ for all $j \geq 1$ sufficiently large in which case we set $f_\infty = 0$, and $c_j \rightarrow 0$.

    \end{itemize}
    If for $i = 1,2$ we denote $u_j^i = G_{\bar{\Sigma}_j,\bar{g}_j}^{\Sigma_j^i} \in L^2_{\bar{g}_j}({\bar{\Sigma}_j})$ and $d_j = \mathbf{D}((g_j^1,t_j,\Omega_j^1),(g_j^2,t_j,\Omega_j^2)) > 0$ then, up to a subsequence (not relabelled), we have:

    \begin{enumerate}
        \item $c_j\frac{f^1_j - f^2_j}{d_j} \rightarrow \hat{f}_\infty$ in $C^{k,\alpha}(M)$ and $\hat{f}_\infty  = cf_\infty$ for some constant $c \geq 0$.
        
        \item $\frac{u^1_j - u^2_j}{d_j} \rightarrow \hat{u}_\infty$ in $L^2_{g,\mathrm{loc}}(\Sigma)$ and $\hat{u}_\infty \in C^2_\mathrm{loc}(\Sigma)$ is a non-zero function of slow growth.
    \end{enumerate}
    Moreover, $\widetilde{L}_{\Sigma,g} \hat{u}_\infty = \frac{n}{2}\left( \nu_{\Sigma,g}(\hat{f}_{\infty}) - \frac{1}{|\Sigma|_g}\int_{\Sigma}\nu_{\Sigma,g}(\hat{f}_\infty)\, dA_g\right)$ and $\int_\Sigma \hat{u}_\infty \, dA_g= \frac{n}{2}\int_{\Omega}\hat{f}_\infty \, dV_g$.
\end{lemma}

\begin{remark}
    The proof of Lemma \ref{lemma: compactness for pairs} parallels \cite[Section 7]{LW25} and is also similar in both its proof and conclusions to that of Theorem \ref{thm: induced twisted jacobi fields}. We note however that the conclusions of Theorem \ref{thm: induced twisted jacobi fields} part 2 differ from those of Lemma \ref{lemma: compactness for pairs} part 2 since the function produced is of slow growth; this arises from the fact that the sequences of triples in the statement all lie in the same pseudo-neighbourhood.
\end{remark}

\begin{proof}
    First, there is some constant $C > 0$, depending on $g$, such that (whenever $j \geq 1$ is sufficiently large if $f_j^1 \neq f_j^2$ and regardless if $f_j^1 = f_j^2$) we have
    $$d_j \geq \Vert g_j^1 - g_j^2\Vert_{L^\infty(M)} \geq C \cdot c_j \cdot \Vert f_\infty\Vert_{L^\infty(M)} \geq 0,$$
    thus up to a subsequence (not relabelled) $\frac{c_j}{d_j} \rightarrow c \geq 0$ and hence
    \begin{equation}\label{eqn: conformal functions converge}
        c_j\frac{f_j^1 - f_j^2}{d_j} \rightarrow c f_\infty = \hat{f}_\infty,
    \end{equation}
    in $C^{k,\alpha}(M)$; this establishes part 1.

    \bigskip

    Since $\Sigma_j^i \rightarrow \Sigma$ for $i = 1,2$ we ensure that for sufficiently large $j \geq 1$ there exist radii, $r_j \rightarrow 0$, such that $\Vert u^i_j\Vert_{C^2_*(\Sigma \setminus B^{g_j}_{r_j}(\mathrm{Sing}(\Sigma^i_j))} \rightarrow 0$ and with
        $$\Sigma^i_j  \setminus B^{g_j}_{r_j}(\mathrm{Sing}(\Sigma^i_j)) = \mathrm{graph}_{\Sigma_j^i}^{g_j^i}(u^i_j) \setminus B^{g_j}_{r_j}(\mathrm{Sing}(\Sigma^i_j)
        ).$$
    Moreover, the difference of the graphs, $w_j = u^1_j - u^2_j$, satisfies an equation of the form
    \begin{equation}\label{eqn: pde for difference}
    \mathcal{M}^{\bar{g}_j}(w_j) - H_j = -L_{\Sigma,g}(w_j) + \frac{c_jn}{2}\nu_{\Sigma,g}(f^1_j - f^2_j) + \mathrm{div}_{\Sigma}(\mathcal{E}_1(w_j)) + r_S^{-1}\mathcal{E}_2(w_j),
    \end{equation}
    in the notation of Appendix \ref{subsec: mean curvature operator}, so that in particular $\mathcal{M}^{\bar{g}_j}$ is the mean curvature operator with respect to the metric $\bar{g}_j$, and $H_j$ is the constant mean curvature of $\bar{\Sigma}_j$.

    \bigskip

    For fixed $\sigma \in (0,\frac{\gamma_{\mathrm{gap}}(\Lambda)}{2})$ and $K > 2$ chosen as in Lemma \ref{lemma: J monotone for divergence form} there is an $r_0 > 0$ and a constant $C > 0$, both depending on $\Sigma, g, \sigma$, and $\Lambda$, such that for the difference of the graphs, $w_j = u^1_j - u^2_j$, whenever $p \in \mathrm{Sing}(\bar{\Sigma}_j)$, $\tau \in [2r_j,2r_0]$, and $p_j^1 \in \mathrm{Sing}(\Sigma^1_j) \cap B_{r_j}^{g^1}(p)$ we deduce that, similarly to the estimate (\ref{eqn: L2 nonconcentration for twisted jacobi}) in the proof of Lemma \ref{lemma: compactness of twisted JF}, we have \begin{equation}\label{eqn: L2 nonconcentration for difference}
    \Vert w_j\Vert_{L^2_{\bar{g}_j}(A^{\bar{g}_j}(p;\frac{\tau}{2},\tau))} \leq C \cdot d_j \cdot \tau^{\frac{n}{2} + \gamma_2^+(\mathbf{C}_{p^1_j}\Sigma^1_j) - \sigma}.
    \end{equation}
    In order to deduce (\ref{eqn: L2 nonconcentration for difference}) above we observe that we can establish identical estimates to those obtained in \cite[Section 7.1]{LW25} in our setting. To see this we note that by utilising Lemma \ref{lemma: compactness of isoperimetric regions} in place of \cite[Theorem G.1]{LW25} we obtain a direct analogue of \cite[Lemma 7.1]{LW25}, with near identical computations leading to analogues of \cite[Corollaries 7.2 and 7.3]{LW25}; in particular, (\ref{eqn: L2 nonconcentration for difference}) follows by a direct application of an analogue of \cite[Corollary 7.3]{LW25}.
    
    \bigskip

    By writing $d_j^\prime = \Vert w_j\Vert_{L^2_{\bar{g}_j}(\bar{\Sigma}_j \setminus B_{r_0}^{\bar{g}_j}(\mathrm{Sing}(\bar{\Sigma}_j))}$ we now show that (compare with the proof of Theorem \ref{thm: induced twisted jacobi fields} part 2) that $\liminf_{j \rightarrow \infty} \frac{d_j}{d_j^\prime} \in (0,\infty)$. To see this, if $\liminf_{j \rightarrow \infty} \frac{d_j}{d_j^\prime} = 0$ then dividing (\ref{eqn: conformal functions converge}) by $d_j^\prime$ we see that up to a subsequence (not relabelled) we have
    $c_j\frac{f^1_j - f^2_j}{d_j^\prime} \rightarrow 0$
    in $C^{k,\alpha}(M)$, and
    $\frac{w_j}{d_j^\prime} \rightarrow w_\infty^\prime$
    in $L^2_{g,\mathrm{loc}}(\Sigma)$, with $\Vert w_\infty^\prime\Vert_{L^2_g(\Sigma \setminus B_{r_0}^g(\mathrm{Sing}(\Sigma))} = 1$.
    However, by dividing (\ref{eqn: pde for difference}) by $d_j^\prime$ we see that additionally $w_\infty^\prime \in W^{1,2}_{g,\mathrm{loc}}(\Sigma)$ weakly, and hence by Remark \ref{rem: higher regularity for weak solutions} strongly, solves $\widetilde{L}_{\Sigma,g}w_\infty^\prime = 0$ with $w_\infty^\prime = 0$ on $B_{r_0}(\mathrm{Sing}(\Sigma))$ by (\ref{eqn: L2 nonconcentration for difference}); this contradicts $\Vert w_\infty^\prime\Vert_{L^2_g(\Sigma \setminus B_{r_0}^g(\mathrm{Sing}(\Sigma))} = 1$ by unique continuation.
    
    \bigskip
    
    We now rule out the possibility that $\liminf_{j \rightarrow \infty} \frac{d_j}{d_j^\prime} = \infty$, first noting that up to a subsequence (not relabelled) we have that $\frac{w_j}{d_j} \rightarrow w_\infty$ in $L^2_{g,\mathrm{loc}}(\Sigma)$. Also, similarly to the above, by dividing (\ref{eqn: pde for difference}) by $d_j$ and using (\ref{eqn: L2 nonconcentration for difference}) we see that, by similar reasoning in the paragraph before (\ref{eqn: Caccioppoli for cmc}) in the proof of Theorem \ref{thm: induced twisted jacobi fields}, $w_\infty \in W^{1,2}_{g,\mathrm{loc}}(\Sigma)$ weakly, and hence by Remark \ref{rem: higher regularity for weak solutions} strongly, solves $\widetilde{L}_{\Sigma,g}w_\infty = \nu_{\Sigma,g}(\hat{f}_\infty) - \frac{1}{|\Sigma|_g}\int_{\Sigma} \nu_{\Sigma,g}(\hat{f}_\infty) \, dA_g$ with $w_\infty = 0$ on $\Sigma \setminus B_{r_0}(\mathrm{Sing}(\Sigma))$. Now since for $i = 1,2$ we have $\mathrm{spt}(f_j^i) \cap B_{r}^g(\mathrm{Sing}(\Sigma)) = \emptyset$ for all $r \in (0,\mathrm{inj}(\Sigma,g))$ we see that in particular $\mathrm{spt}(\hat{f}_\infty) \cap B^g_{2r_0}(\mathrm{Sing}(\Sigma)) = \emptyset$ also; thus by unique continuation again, we must have that $\omega_\infty = 0$. As $\nu_{\Sigma,g}(f_\infty)$ is not constant on $\Sigma$ (if it is non-zero) we therefore see that we must have $\hat{f}_\infty = 0$ also; hence by construction we see that
    \begin{equation}\label{eqn: metrics over distance for contradiction}
        \frac{\Vert g^1_j - g^2_j\Vert_{L^{\infty}(M)}}{d_j} \rightarrow 0;
    \end{equation}
    we will exploit the definition of $d_j$ to see that this yields a contradiction. To this end we define $v^{(1,2)}_j = G^{\Sigma^2_j}_{\Sigma_j^1, g_j^1} \in L_{g_j^1}^2(\Sigma_j^1)$, $v^{(2,1)}_j = G^{\Sigma^1_j}_{\Sigma_j^2, g_j^2} \in L_{g^2_j}^2(\Sigma_j^2)$, and set $$\tilde{d}_j = \Vert v^{(1,2)}_j\Vert_{L^2_{g^1_j}(\Sigma^1_j \setminus B_{\frac{r_0}{K}}^{g_j^1}(\mathrm{Sing}(\Sigma^1_j))} + \Vert v^{(2,1)}_j\Vert_{L^2_{g^2_j}(\Sigma^2_j \setminus B_{\frac{r_0}{K}}^{g_j^2}(\mathrm{Sing}(\Sigma^2_j))} \in (0, d_j),$$
    for $K > 2$ chosen as in Lemma \ref{lemma: J monotone for divergence form}. We can then apply \cite[Lemma 7.4]{LW25} (which requires no assumption on the mean curvature) which combined with the fact that $w_\infty = 0$ to see that up to a subsequence (not relabelled) both $\frac{v^{(1,2)}_j}{d_j}, \frac{v^{(2,1)}_j}{d_j} \rightarrow 0$ in $L^2_{g,\mathrm{loc}}(\Sigma)$ and thus 
    \begin{equation}\label{eqn: ratios of distances}
        \liminf_{j \rightarrow \infty} \frac{\tilde{d}_j}{d_j} = 0.
    \end{equation}
    By the reasoning leading to (\ref{eqn: L2 nonconcentration for difference}) above, we can establish an analogue of \cite[Corollary 7.2(ii)]{LW25} which ensures that, arguing identically to the derivation of \cite[(59)/(60)]{LW25}, we obtain the bounds
    $$\Vert v^{(1,2)}_j\Vert_{L_{g_j^1}^2(\Sigma_j^1)}, \Vert v^{(2,1)}_j\Vert_{L_{g_j^2}^2(\Sigma_j^2)} \leq O(\tilde{d}_j);$$
    combined with (\ref{eqn: metrics over distance for contradiction}), (\ref{eqn: ratios of distances}), and the definition of $d_j$ this is yields a contradiction.
    
    \bigskip
    
    With $\liminf_{j \rightarrow \infty} \frac{d_j}{d_j^\prime} \in (0,\infty)$ established, by sending $j \rightarrow \infty$ we deduce from (\ref{eqn: L2 nonconcentration for difference}) that up to a subsequence (not relabelled) we have the estimate
    $$\Vert w_j\Vert_{C^2(\bar{\Sigma}_j \setminus B^{\bar{g}_j}_{2r_0}(\mathrm{Sing}(\bar{\Sigma}_j)))} \leq C(d_j^\prime + d_j + \Vert f^1_j - f^2_j\Vert_{C^2(M)}),$$
    for a constant $C > 0$, depending on $g, \Sigma,$ and $r_0$. Then, by dividing (\ref{eqn: pde for difference}) by $d_j$ we conclude, by similar reasoning in the paragraph before (\ref{eqn: Caccioppoli for cmc}) in the proof of Theorem \ref{thm: induced twisted jacobi fields}, that $\frac{w_j}{d_j} \rightarrow \hat{u}_\infty$ in $C^3_\mathrm{loc}(\Sigma)$ with $\hat{u}_\infty \in C^2_\mathrm{loc}(\Sigma)$ which weakly, and hence by Remark \ref{rem: higher regularity for weak solutions} strongly, solves  
    $$\widetilde{L}_{\Sigma,g}\hat{u}_\infty = \frac{n}{2}\left( \nu_{\Sigma,g}(\hat{f}_{\infty}) - \frac{1}{|\Sigma|_g}\int_{\Sigma}\nu_{\Sigma,g}(\hat{f}_\infty)\, dA_g\right);$$
    moreover, by the definition of $d_j$ and the claim, we ensure that $\hat{u}_\infty$ is non-zero.
    
    \bigskip
    
    By dividing (\ref{eqn: L2 nonconcentration for difference}) by $d_j$ we see that, by the reasoning in the paragraph following (\ref{eqn: L2 nonconcentration for twisted jacobi}) in the proof of Lemma \ref{lemma: compactness of twisted JF}, for each $p \in \mathrm{Sing}(\Sigma)$ and $\tau \in (0,r_0)$ we have
    $$\Vert \hat{u}_\infty\Vert_{L^2_{g}(A^g(p;\frac{\tau}{2},\tau))} \leq C \tau^{\frac{n}{2} + \gamma_2^+(\mathbf{C}_{p}\Sigma) - \sigma},$$
    which by definition of the asymptotic rate at $p$ implies that $\mathcal{AR}_p(\Sigma) \geq \gamma_2^+(\mathbf{C}_p\Sigma)$; i.e.~that $\hat{u}_\infty$ of slow growth and so by Lemma \ref{lemma: properties of slow growth} part 4 we ensure that $\hat{u}_\infty \in W^{1,2}_g(\Sigma)$.
  
    \bigskip

    It remains to show that $\int_{\Sigma}\hat{u}_\infty \, dA_g = \frac{n}{2}\int_\Omega \hat{f}_\infty \, dV_g$, which follows by similar reasoning for the derivation of (\ref{eqn: limit of intersection}) and (\ref{eqn: limit of difference}) in proving Theorem \ref{thm: induced twisted jacobi fields} part 2 and the proof of Lemma \ref{lemma: compactness of twisted JF}. Namely, for each $s > 0$ and $j \geq 1$ sufficiently large we can compute similarly to as in the proof of Theorem \ref{thm: induced twisted jacobi fields} to deduce that for some constant $C > 0$ independent of $j \geq 1$ we have
    $$\left|\int_{\bar{\Sigma}_j \setminus B_s^{\bar{g}_j}(\mathrm{Sing}(\bar{\Sigma}_j))} w_j \, dA_{\bar{g}_j} - c_j\frac{n}{2}\int_{\Omega^1_j \cap \Omega^2_j} (f^1_j - f^2_j) \, dV_{\bar{g}_j}\right| \leq C d_j s^8 + O(d_j^2).$$
    Thus, after dividing by $d_j$ and sending $j \rightarrow \infty$ we see that
    $$\left| \int_{\Sigma \setminus B_s^g(\mathrm{Sing}(\Sigma))} \hat{u}_\infty \, dA_g - \frac{n}{2}\int_{\Omega } \hat{f}_\infty \, dV_g \right| \leq Cs^8;$$
    sending $s \rightarrow 0$ we conclude that $\int_{\Sigma}\hat{u}_\infty \, dA_g = \frac{n}{2}\int_\Omega \hat{f}_\infty \, dV_g$ as desired.
\end{proof}

\subsection{A Sard--Smale theorem for metric volume pairs}\label{subsec: sard-smale}

We define, for each $(g,t,\Omega) \in \mathcal{T}^{k,\alpha}$, the \textbf{top part} of the pseudo-neighbourhood by setting
\begin{equation*}
    \mathcal{L}_{\mathrm{top}}^{k, \alpha}(g,t,\Omega; \Lambda, \delta) = \{(g^\prime, t^\prime, \Omega^\prime) \in \mathcal{L}^{k, \alpha}(g, t, \Omega; \Lambda, \delta) \, | \, \dim(\mathrm{ker}^{+}\widetilde{L}_{\Sigma^\prime, g^\prime}) = \dim(\mathrm{ker}^{+}\widetilde{L}_{\Sigma, g})\}. 
\end{equation*}
Note that this set is closed in $\mathcal{L}^{k, \alpha}(g,t, \Omega; \Lambda, \delta)$ for $\delta \leq \kappa_1$, where $\kappa_1$ is as in Lemma \ref{lemma: compactness of twisted JF}. We now prove that we can parametrise slices of the top part of each pseudo-neighbourhood by compact subsets of the kernel of the twisted Jacobi operator:

\begin{lemma}\label{lemma: paramatrisation and projection}
    Fix $(g,t,\Omega) \in \mathcal{T}^{k,\alpha}$ and let $I = \dim(\mathrm{ker}^{+} \widetilde{L}_{\Sigma, g}) < \infty$. There exists constants $\kappa_2 \in (0, \kappa_1)$ (for $\kappa_1$ as in Lemma \ref{lemma: compactness of twisted JF}) and $r_0 > 0$, and an $I$-dimensional linear subspace $\mathcal{F} \subset C_{c}^{k, \alpha}(M \setminus B_{r_0}^{g}(\mathrm{Sing}(\Sigma)))$ with $\int_\Omega f \, dV_g = 0$ for each $f \in \mathcal{F}$, all depending on $g, t, \Omega,$ and $\Lambda$, for which the following holds: 
    \begin{enumerate}
        \item Given $(g^\prime,t',\Omega^\prime) \in \mathcal{L}^{k, \alpha}_{\mathrm{top}}(g,t, \Omega; \Lambda, \kappa_2)$ the map 
        \begin{equation*}
            A_{\mathcal{F}} : \mathcal{F} \rightarrow \mathrm{Ker}^+ \widetilde{L}_{\Sigma^\prime, g^\prime}, \quad f \mapsto \pi_{\Sigma^\prime, g^\prime}\left(\nu_{\Sigma^\prime, g^\prime}(f) - \frac{1}{|\Sigma^\prime|_{g^\prime}}\int_{\Sigma^\prime}\nu_{\Sigma^\prime,g^\prime}(f)\, dA_{g^\prime}\right)
        \end{equation*}
        is a linear isomorphism, where $\pi_{\Sigma^\prime, g^\prime} : L_T^2(\Sigma^\prime) \rightarrow \mathrm{Ker}^{+} \widetilde{L}_{\Sigma^\prime, g^\prime}$ denotes the orthogonal projection. Moreover, there is no solution, $u \in L^1_T(\Sigma^\prime) \cap C^2_\mathrm{loc}(\Sigma^\prime)$, of slow growth to $\widetilde{L}_{\Sigma^\prime,g^\prime}u = \nu_{\Sigma^\prime, g^\prime}(f) - \frac{1}{|\Sigma^\prime|_{g^\prime}}\int_{\Sigma^\prime}\nu_{\Sigma^\prime,g^\prime}(f)\, dA_{g^\prime}$ for $f \in \mathcal{F} \setminus \{0\}$.
        
        \item Let $\mathcal{F} \cdot \bar{g} = \{(1 + f)\bar{g}; f \in \mathcal{F}\}$. For every $(\bar{g},\bar{t}, \bar{\Omega}) \in \mathcal{L}_{\mathrm{top}}^{k, \alpha}(g, t,\Omega; \Lambda, \kappa_2)$, the map
        \begin{align*}
            P_{\bar{g},\bar{t},\bar{\Omega}} : \mathcal{L}_{\mathrm{top}}^{k, \alpha}(g, t,\Omega; \Lambda, \kappa_2) \cap \Pi^{-1}(\mathcal{F} \cdot \bar{g} \times \{\bar{t}\}) & \rightarrow \mathrm{Ker}^+ \widetilde{L}_{\bar{\Sigma}, \bar{g}}, \\
            (g^\prime,t^\prime,\Omega^\prime) & \mapsto \pi_{ \bar{\Sigma},\bar{g}}\left(G_{\bar{\Sigma}, \bar{g}}^{\Sigma^\prime} \zeta_{\bar{\Sigma}, \bar{g}, r_0} - \frac{1}{|\bar{\Sigma}|_{\bar{g}}}\int_{\bar{\Sigma}}G_{\bar{\Sigma}, \bar{g}}^{\Sigma^\prime} \zeta_{\bar{\Sigma}, \bar{g}, r_0} \, dA_{\bar{g}} \right), 
        \end{align*}
        where $\zeta_{\bar{\Sigma}, \bar{g}, r_0}(x) = \zeta(\mathrm{dist}_g(x, \mathrm{Sing}(\bar{\Sigma}))/r_0)$, for a fixed cutoff function $\zeta \in C^{\infty}(\mathbb{R}, [0, 1])$ such that $\zeta = 0$ on $(-\infty, 1]$, and $\zeta = 1$ on $[2, \infty)$, is uniformly bi-Lipschitz onto its image with a bi-Lipschitz constant $C > 0$, depending on $g, t, \Omega$, and $\Lambda$. 
    \end{enumerate}
\end{lemma}

\begin{proof}

    We first fix a cutoff function $\zeta$ and define $\zeta_{\Sigma,g,r}$ for sufficiently small $r > 0$ as in the statement of part 2 above. Choosing some $r_0 \in (0,\mathrm{inj}(\Sigma,g))$ sufficiently small, we ensure that the map $A_{\zeta_{\Sigma,g,r_0}} : \mathrm{Ker}^+\widetilde{L}_{\Sigma,g} \rightarrow \mathrm{Ker}^+\widetilde{L}_{\Sigma,g}$ defined by setting
    $$A_{\zeta_{\Sigma,g,r_0}}(v) = \pi_{\Sigma,g}\left(\zeta_{\Sigma,g,r_0} v - \frac{1}{|\Sigma|_g}\int_{\Sigma} \zeta_{\Sigma,g,r_0}v \, dA_g\right)$$
    is a linear isomorphism; to see this we note that if $A_{\zeta_{\Sigma,g,r_0}}(v) = 0$ then $\zeta_{\Sigma,g,r}v -  \frac{1}{|\Sigma|_g}\int_{\Sigma} \zeta_{\Sigma,g,r}v \, dA_g\in (\mathrm{Ker}^+\widetilde{L}_{\Sigma,g})^\perp$ and so by testing with $v$ itself we conclude that
    $0 = \int_{\Sigma}\zeta_{\Sigma,g,r}v^2 \, dA_g$ from which it follows that $v = 0$ whenever $r_0 > 0$ is chosen sufficiently small since $v \in \mathrm{Ker}^+\widetilde{L}_{\Sigma,g}$. We then choose an orthonormal (with respect to the $L^2$ inner product) basis, $u_1, \dots, u_I$, of $\mathrm{Ker}^+\widetilde{L}_{\Sigma,g}$, and for each $i = 1,\dots,I$ choose some $f_i \in C^{k,\alpha}(M \setminus B_{r_0}^g(\mathrm{Sing}(\Sigma))$ which satisfy both $\nu_{\Sigma,g}(f_i) = \zeta_{\Sigma,g,r_0} u_i$ (such functions exist since $\zeta_{\Sigma,g,r_0} \in C^\infty_c(\Sigma)$) and $\int_{\Omega} f_i  \, dV_g = 0$; we then define $\mathcal{F}$ to be the span of the $\{f_i\}_{i = 1}^I$. 
    
    \bigskip
    
    As $A_{\zeta_{\Sigma,g,r_0}}$ was shown to be a linear isomorphism we see that the linear map $\pi_{\Sigma,g} \circ \nu
    _{\Sigma,g}$ is invertible, and hence has bounded inverse, defined on $\mathrm{Ker}^+\widetilde{L}_{\Sigma,g}$, when restricted to $\mathcal{F}$; let us denote by $C > 0$ the operator norm of the inverse of this restriction. Since by construction $\mathcal{F}$ has the same dimension as $\mathrm{dim}(\mathrm{Ker}^+\widetilde{L}_{\Sigma,g})$, it suffices by the rank-nullity theorem to show that $A_\mathcal{F}$ is injective to conclude that it is a linear isomorphism; we will now show that in fact $\Vert f\Vert_{C^{k,\alpha}(M)} \leq (1 + C)\Vert A_\mathcal{F}(f)\Vert_{L^2_{g'}(\Sigma')}$ which implies the injectivity of $A_\mathcal{F}$. If not, then we can find a sequences $\{(g_j,t_j,\Omega_j)\}_{j \geq 1} \subset \mathcal{L}^{k,\alpha}_\mathrm{top}(g,t,\Omega;\Lambda,\kappa_1)$ such that $(g_j,t_j,\Omega_j) \rightarrow (g,t,\Omega)$ in $\mathcal{L}^{k,\alpha}(g,t,\Omega;\Lambda,\kappa_1)$ and $\{f_j\}_{j \geq 1} \subset \mathcal{F}$ with $\Vert f_j\Vert_{C^{k,\alpha}(M)} = 1$ but such that $\Vert A_\mathcal{F}(f_j)\Vert_{L^2_{g'}(\Sigma')} < \frac{1}{1+C}$ for all $j \geq 1$. Up to a subsequence (not relabelled) we have that $f_j \rightarrow f$ in $C^{k-1,\alpha}(M)$ for some $f \in \mathcal{F}$ with $\Vert f\Vert_{C^{k,\alpha}(M)} = 1$, thus by the boundedness of $A_\mathcal{F}$ we see that $\Vert A_\mathcal{F}(f)\Vert_{L^2_{g'}(\Sigma')} \leq \frac{1}{1+C}$. However, by Lemma \ref{lemma: compactness of twisted JF} we see that $\pi_{\Sigma,g}(\nu_{\Sigma,g}(f)) = A_\mathcal{F}(f)$ and hence 
    $$\Vert f\Vert_{C^{k,\alpha}(M)} = \Vert (\pi_{\Sigma,g}\circ \nu_{\Sigma,g})^{-1} (A_\mathcal{F}(f))\Vert_{C^{k,\alpha}(M)} \leq C \Vert A_{\mathcal{F}}(f)\Vert_{C^{k,\alpha}(M)} \leq \frac{C}{1+C} < 1,$$
    contradicting the fact that $\Vert f\Vert_{C^{k,\alpha}(M)} = 1$; hence $A_\mathcal{F}$ is a linear isomorphism.

    \bigskip

    The final assertion in part 1 follows since if we have a slow growth function, $u \in L^1_T(\Sigma^\prime) \cap C^2_\mathrm{loc}(\Sigma^\prime)$, satisfying $\widetilde{L}_{\Sigma^\prime,g^\prime}u = \nu_{\Sigma',g'}(f) - \frac{1}{|\Sigma^\prime|_{g^\prime}}\int_{\Sigma^\prime} \nu_{\Sigma',g'}(f) \, dA_{g^\prime}$ then, by applying integration by parts as in Lemma \ref{lemma: properties of slow growth}, it follows that $A_{\mathcal{F}}(f) = 0$ which implies that $f = 0$ since $A_\mathcal{F}$ is an isomorphism.

    \bigskip

    In order to establish part 2 we utilise the compactness result of Lemma \ref{lemma: compactness for pairs} above along with a contradiction argument. If there is no such $\kappa_2 \in (0,\kappa_1)$ such that part 2 holds, then there exists some sequence $(\bar{g}_j,\bar{t}_j,\bar{\Omega}_j)_j \subset \mathcal{L}^{k,\alpha}_{\mathrm{top}}(g,t,\Omega;\Lambda,\kappa_1)$ converging to $(g,t,\Omega)$ in $\mathcal{L}^{k,\alpha}(g,t,\Omega;\Lambda,\kappa_1)$ such that the maps $P_{\bar{g}_j,\bar{t}_j, \bar{\Omega}_j}$ are not uniformly bi-Lipschitz. Thus, for $i = 1,2$ there are distinct sequences $\{(\bar{g}_j^i,\bar{t}_j,\bar{\Omega}_j^i)\}_{j \geq 1} \subset \mathcal{L}_{\mathrm{top}}^{k,\alpha}(g,t,\Omega;\Lambda,\kappa_1)$ converging to $(g,t,\Omega)$ in $\mathcal{L}^{k,\alpha}(g,t,\Omega;\Lambda,\kappa_1)$ with $\bar{g}_j^i = (1 + c_jf_j^i)\bar{g}_j$ for $f_j^i \in \mathcal{F}$ and $c_j \rightarrow 0$. Since $\bar{g}_i^j \in \mathcal{F} \cdot \bar{g}_j$, by denoting $u_j^i = G_{\bar{\Sigma}_j,\bar{g}_j}^{\Sigma_j^i} \in L_{\bar{g}_j}^2({\bar{\Sigma}_j})$ and $d_j = \mathbf{D}((\bar{g}_j^1,\bar{t}_j,\bar{\Omega}_j^1),\bar{g}_j^2,\bar{t}_j,\bar{\Omega}_j^2)) > 0$ in the notation of Lemma \ref{lemma: compactness for pairs}, then exactly one of the following cases occurs for infinitely many $j \geq 1$:
    \begin{equation}\label{eqn: contradiction blowup}
        \,\left\vert\left\vert\pi_{\bar{\Sigma}_j,\bar{g}_j}\left((u^1_j - u^2_j)\zeta_{\bar{\Sigma}_j,\bar{g}_j,r_0} - \frac{1}{|\bar{\Sigma}_j|_{\bar{g}_j}}\int_{\bar{\Sigma}_j} (u^1_j - u^2_j)\zeta_{\bar{\Sigma}_j,\bar{g}_j,r_0} \, dA_{\bar{g}_j}\right)\right\vert\right\vert_{L^2(\bar{\Sigma}_j)} \geq jd_j
    \end{equation}
    \begin{equation}\label{eqn: contradiction null}
        \left\vert\left\vert\pi_{\bar{\Sigma}_j,\bar{g}_j}\left((u^1_j - u^2_j)\zeta_{\bar{\Sigma}_j,\bar{g}_j,r_0} - \frac{1}{|\bar{\Sigma}_j|_{\bar{g}_j}}\int_{\bar{\Sigma}_j} (u^1_j - u^2_j)\zeta_{\bar{\Sigma}_j,\bar{g}_j,r_0} \, dA_{\bar{g}_j}\right)\right\vert\right\vert_{L^2(\bar{\Sigma}_j)} \leq \frac{d_j}{j}
    \end{equation}
    As $f_j^1 - f_j^2 \in \mathcal{F}$ for each $j \geq 1$ (so that $f^1_j - f_j^2 \rightarrow f_\infty \in \mathcal{F}$ (so that either $f_\infty = 0$ or $\nu_{\Sigma,g}(f_\infty)$ is not constant), by applying Lemma \ref{lemma: compactness for pairs} we see that up to a subsequence (not relabelled) we have:
    \begin{enumerate}
        \item $c_j\frac{f^1_j - f^2_j}{d_j} \rightarrow \hat{f}_\infty$ in $C^{k,\alpha}(M)$ and $\hat{f}_\infty  = cf_\infty$ for some $c \geq 0$.
        
        \item $\frac{u^1_j - u^2_j}{d_j} \rightarrow \hat{u}_\infty$ in $L^2_{g,\mathrm{loc}}(\Sigma)$ and $\hat{u}_\infty \in C^2_\mathrm{loc}(\Sigma)$ is a non-zero function of slow growth.
    \end{enumerate}
    Moreover, $\widetilde{L}_{\Sigma,g} \hat{u}_\infty = \frac{n}{2}\left( \nu_{\Sigma,g}(\hat{f}_{\infty}) - \frac{1}{|\Sigma|_g}\int_{\Sigma}\nu_{\Sigma,g}(\hat{f}_\infty)\, dA_g\right)$ and $\int_\Sigma \hat{u}_\infty \, dA_g= 0 $ (since $\hat{f}_\infty \in \mathcal{F}$). Thus, by applying the part 1 established above (for $(g,t,\Omega) \in \mathcal{L}^{k,\alpha}_\mathrm{top}(g,t,\Omega;\Lambda,\kappa_2)$) we see that $\hat{f}_\infty = 0$, and so $\hat{u}_\infty \in \mathrm{Ker}^+\widetilde{L}_{\Sigma,g}$. 

    \bigskip

    First, by the convergence of $\frac{u^1_j - u^2_j}{d_j} \rightarrow \hat{u}_\infty$ and $\zeta_{\bar{\Sigma}_j,\bar{g}_j,r_0} \rightarrow \zeta_{\bar{\Sigma},\bar{g},r_0}$ in $L^2_{g,\mathrm{loc}}$ we see that up to  subsequence (not relabelled) we have
    $$\frac{u^1_j - u^2_j}{d_j}\zeta_{\bar{\Sigma}_j,\bar{g}_j,r_0} \rightarrow \hat{u}_\infty\zeta_{\bar{\Sigma},\bar{g},r_0}$$
    in $L^2_{g,\mathrm{loc}}$ also. Thus by applying Lemma \ref{lemma: compactness of twisted JF} we see that the left hand side of (\ref{eqn: contradiction blowup}) (divided by $d_j$) is finite for all $j \geq 1$, thus (\ref{eqn: contradiction blowup}) cannot hold for infinitely many $j \geq 1$.
    
    \bigskip
    
    Similarly, if (\ref{eqn: contradiction null}) above were to occur for infinitely many $j \geq 1$, then we would have
    $$A_{\zeta_{\Sigma,g,r_0}}(\hat{u}_\infty)  = \pi_{\Sigma,g} \left(\hat{u}_\infty \zeta_{\Sigma,g,r_0} - \frac{1}{|\Sigma|_g}\int_\Sigma \hat{u}_\infty \zeta_{\Sigma,g,r_0} \, dA_g\right) = 0.$$
    Now by the proof of part 1 above, we know that $A_{\zeta_{\Sigma,g,r_0}}$ is a linear isomorphism, hence $A_{\zeta_{\Sigma,g,r_0}}(\hat{u}_\infty) = 0$ implies that $\hat{u}_\infty = 0$, a contradiction. Thus, for some potentially smaller $\kappa_2$ and $C > 0$ depending on $g,t,\Omega,\Lambda$, we conclude that the map $P_{\bar{g},\bar{t},\bar{\Omega}}$ is uniformly bi-Lipschitz onto its image.
\end{proof}

By utilising the above parametrisation result for the top part of each pseudo-neighbourhoods we are able to establish the following ``local'' Sard-Smale theorem which, when combined with a decomposition of the space of triples in the next subsection based on Appendix \ref{sec: appendix B}, will allow us to conclude that semi-nondegeneracy is a generic property for isoperimetric regions in dimension eight:

\begin{lemma}\label{lemma: generic pairs for neighbourhoods}
    Given $(g, t, \Omega) \in \mathcal{T}^{k,\alpha}$, $\Lambda \geq 1$, and $\delta > 0$, let $\mathcal{G}^{k,\alpha}(g,t,\Omega;\Lambda,\delta)$ denote the set
    $$\{(\bar{g},\bar{t}) \in \mathcal{G}^{k,\alpha} \times \mathbb{R} \, | \, \text{every } \bar{\Omega} \in \mathcal{I}(\bar{g},\bar{t}) \text{ with } (\bar{g},\bar{t},\bar{\Omega}) \in \mathcal{L}^{k,\alpha}(g,t,\Omega;\Lambda,\delta) \text{ is semi-nondegenerate}\}.$$ Then there exists some $\kappa_0 > 0$, depending on $g,t,\Omega$, and $\Lambda$, such that $\mathcal{G}^{k,\alpha}(g,t,\Omega;\Lambda,\kappa_0)$ is open and dense in $\mathcal{G}^{k,\alpha} \times \mathbb{R}$. Moreover, if $\bar{g} \in \mathcal{G}^{k,\alpha}$ then $\mathcal{G}^{k,\alpha}(g,t,\Omega;\Lambda,\kappa_0) \cap ([\bar{g}] \times \mathbb{R})$ is open and dense in $[\bar{g}] \times \mathbb{R}$.
\end{lemma}

\begin{proof}
    For openness, if $\{(g_j,t_j)\}_{j \geq 1} \in (\mathcal{G}^{k,\alpha}\times \mathbb{R} )\setminus \mathcal{G}^{k,\alpha}(g,t,\Omega;\Lambda,\delta) $ with $g_j \rightarrow g_\infty$ in $C^{k-1,\alpha}$ and $t_j \rightarrow t_\infty$, then for each $j \geq 1$ there exist isoperimetric regions, $\Omega_j \in \mathcal{I}(g_j,t_j)$, and non-zero twisted Jacobi fields, $u_j \in \mathrm{Ker}^+\widetilde{L}_{\Sigma_j,g_j}$ (which without loss satisfy $\Vert u_j\Vert_{L^2(\Sigma_j)} = 1$), such that up to a subsequence (not relabelled) there is some $u_\infty \in \mathrm{Ker}^+\widetilde{L}_{\Sigma_\infty,g_\infty}$ by combining Lemmas \ref{lemma: compactness of pseudo neighbourhoods} and \ref{lemma: compactness of twisted JF}; hence the complement is closed.

    \bigskip

    For denseness, first notice that if $\dim(\mathrm{ker}^{+} \widetilde{L}_{\Sigma, g}) = 0$ then $\mathcal{G}^{k,\alpha}(g,t,\Omega;\Lambda,\kappa_0) = \mathcal{G}^{k,\alpha} \times \mathbb{R}$ for every $\kappa_0 \in (0,\kappa_2)$ by (\ref{eq: top kernel}) in Lemma \ref{lemma: compactness of twisted JF} part 1. We establish the denseness for the positive dimensions by induction and successive approximation; namely, we first show that $\mathcal{G}^{k,\alpha}(g,t,\Omega;\Lambda,\kappa_0)$ is dense in $(\mathcal{G}^{k,\alpha} \times \mathbb{R}) \setminus \Pi(\mathcal{L}^{k,\alpha}_{\mathrm{top}}(g,t,\Omega;\Lambda,\kappa_0))$, and then show that for each $(\bar{g},\bar{t}) \in \Pi(\mathcal{L}^{k,\alpha}_\mathrm{top}(g,t,\Omega;\Lambda,\kappa_0))$ there is a sequence of approximating pairs in $(\mathcal{G}^{k,\alpha} \times \mathbb{R}) \setminus \Pi(\mathcal{L}^{k,\alpha}_{\mathrm{top}}(g,t,\Omega;\Lambda,\kappa_0))$ by exploiting Lemma \ref{lemma: paramatrisation and projection}. Thus we assume, for $I \geq 1$, that the desired density holds whenever $\dim(\mathrm{ker}^{+} \widetilde{L}_{\Sigma, g}) \leq I - 1$.

    \bigskip

    Firstly, if $(g^\prime,t^\prime) \in (\mathcal{G}^{k,\alpha} \times \mathbb{R}) \setminus \Pi(\mathcal{L}^{k,\alpha}_{\mathrm{top}}(g,t,\Omega;\Lambda,\kappa_0))$ but $(g^\prime,t^\prime) \in (\mathcal{G}^{k,\alpha}\times \mathbb{R} )\setminus \mathcal{G}^{k,\alpha}(g,t,\Omega;\Lambda,\kappa_0)$ then $\Pi^{-1}(g^\prime,t^\prime) \cap \mathcal{L}^{k,\alpha}(g,t,\Omega;\Lambda,\kappa_0)$ is nonempty by definition and compact since $\Pi^{-1}(g^\prime,t^\prime)$ is closed and $\mathcal{L}^{k,\alpha}(g,t,\Omega;\Lambda,\kappa_0)$ is compact by Lemma \ref{lemma: compactness of pseudo neighbourhoods}. Thus, there are $\{(g^\prime,t^\prime,\Omega_i)\}_{i = 1}^n \subset \Pi^{-1}(g^\prime,t^\prime) \cap \mathcal{L}^{k,\alpha}(g,t,\Omega;\Lambda,\kappa_0)$ such that 
    \begin{equation}\label{eqn: covering argument for denseness}
    \Pi^{-1}(g^\prime,t^\prime) \cap \mathcal{L}^{k,\alpha}(g,t,\Omega;\Lambda,\kappa_0)  \subset \bigcup_{i = 1}^n \mathcal{L}^{k,\alpha}(g^\prime,t^\prime,\Omega_i;\Lambda,\kappa_0^i);
    \end{equation}
    where here we choose positive numbers, $\{\kappa_0^i\}_{i = 1}^n$, such that $\mathcal{G}^{k,\alpha}(g^\prime,t^\prime,\Omega_i;\Lambda,\kappa_0^i)$ is open and dense in $\mathcal{G}^{k,\alpha} \times \mathbb{R}$ by the inductive assumption since for each $i = 1,\dots, n$ we have $\mathrm{dim}(\mathrm{Ker}^+\widetilde{L}_{\Sigma_i,g^\prime}) \leq I-1$. 
    
    \bigskip
    
    We note then that there must be a sequence $\{(g^\prime_j,t^\prime_j)\}_{j \geq 1} \subset \bigcap_{i = 1}^n\mathcal{G}^{k,\alpha}(g_i,t_i,\Omega_i;\Lambda,\kappa_0^i)$ (which is still open and dense as the finite intersection of such sets) such that $(g^\prime_j,t^\prime_j) \rightarrow (g^\prime,t^\prime)$ as $j \rightarrow \infty$; we will now show that $(g^\prime_j,t^\prime_j) \in \mathcal{G}^{k,\alpha}(g,t,\Omega;\Lambda,\kappa_0)$ for all $j \geq 1$ sufficiently large. If not, then there exist $\Omega^\prime_j \in \mathcal{I}(g^\prime_j,t^\prime_j)$ with $\mathrm{Ker}^+\widetilde{L}_{\Sigma^\prime_j,g^\prime_j} \neq \{0\}$ such that $(g^\prime_j,t^\prime_j,\Omega^\prime_j) \rightarrow (g^\prime,t^\prime,\Omega^\prime)$ for some $\Omega^\prime \in \mathcal{I}(g^\prime,t^\prime)$ with $(g^\prime,t^\prime,\Omega^\prime) \in \Pi^{-1}(g^\prime,t^\prime) \cap\mathcal{L}^{k,\alpha}(g,t,\Omega;\Lambda,\kappa_0)$ by Lemma \ref{lemma: compactness of pseudo neighbourhoods}; however, by the covering above, we must have that $(g^\prime,t^\prime,\Omega^\prime) \in \mathcal{L}^{k,\alpha}(g^\prime,t^\prime,\Omega_i;\Lambda,\kappa_0^i)$ for some $i = 1,\dots,n$ and in particular since $(g^\prime_j,t^\prime_j) \in \mathcal{G}^{k,\alpha}(g^\prime,t^\prime,\Omega_i;\Lambda,\kappa_0^i)$ we also have $\mathrm{Ker}^+\widetilde{L}_{\Sigma_j^\prime,g_j^\prime} = \{0\}$, a contradiction. Thus we see that $\mathcal{G}^{k,\alpha}(g,t,\Omega;\Lambda,\kappa_0)$ is dense in $(\mathcal{G}^{k,\alpha} \times \mathbb{R}) \setminus \Pi(\mathcal{L}^{k,\alpha}_{\mathrm{top}}(g,t,\Omega;\Lambda,\kappa_0))$.
    
    \bigskip

    To conclude we will show that there is a sequence in $(\mathcal{G}^{k,\alpha} \times \mathbb{R}) \setminus \Pi(\mathcal{L}^{k,\alpha}_{\mathrm{top}}(g,t,\Omega;\Lambda,\kappa_0))$ that approximates each choice of $(\bar{g},\bar{t}) \in \Pi(\mathcal{L}^{k,\alpha}_\mathrm{top}(g,t,\Omega;\Lambda,\kappa_0))$; for each such pair we choose some $\bar{\Omega} \in \mathcal{I}(\bar{g},\bar{t})$. We fix $\mathcal{F}$ as in Lemma \ref{lemma: paramatrisation and projection} for the triple $(\bar{g},\bar{t},\bar{\Omega}) \in \mathcal{T}^{k,\alpha}$ so that the map
    $$\widetilde{\Pi} : \mathcal{Z} = P_{\bar{g},\bar{t},\bar{\Omega}}\left(\mathcal{L}^{k,\alpha}_{\mathrm{top}}(\bar{g},\bar{t},\bar{\Omega};\Lambda,\kappa_0) \cap \Pi^{-1}(\mathcal{F} \cdot \bar{g} \times \{\bar{t}\})\right) \subset \mathrm{Ker}^+\widetilde{L}_{\bar{\Sigma},\bar{g}} \rightarrow \mathcal{F},$$
    defined by $\widetilde{\Pi}(P_{\bar{g},\bar{t},\bar{\Omega}}((1+f)\bar{g},\bar{t},\cdot)) = f$ for each $f \in \mathcal{F}$, 
    is a Lipschitz map between two vector spaces of the same finite dimension with compact domain. More precisely, Lemma \ref{lemma: paramatrisation and projection} part 2 shows $\mathcal{L}^{k,\alpha}_{\mathrm{top}}(g,t,\Omega;\Lambda,\kappa_0) \cap \Pi^{-1}(\mathcal{F} \cdot \bar{g} \times \{\bar{t}\})$ is bi-Lipschitz to a compact subset of $\mathrm{Ker}^+\widetilde{L}_{\bar{\Sigma},\bar{g}}$ and Lemma \ref{lemma: paramatrisation and projection} part 1 shows that this vector space is bi-Lipschitz to $\mathcal{F}$; we then set $\widetilde{\Pi} = I_{\mathcal{F}} \circ \Pi \circ P_{\bar{g},\bar{t},\bar{\Omega}}^{-1}$ where $I_\mathcal{F}((1 +f)\bar{g},\bar{t}) = f$ for $I_\mathcal{F} : \mathcal{F} \cdot \bar{g} \times \mathbb{R} \rightarrow \mathcal{F}$ which is a Lipschitz map between vector spaces. 

    \bigskip

    We now show that in some neighbourhood of $0 \in \mathrm{Ker}^+\widetilde{L}_{\bar{\Sigma},\bar{g}}$ in the domain of $\widetilde{\Pi}$ consists of critical points, and hence by the Sard--Smale theorem for Lipschitz maps (see \cite[Lemma 8.5]{LW25}) this neighbourhood in the domain has image under $\widetilde{\Pi}$ of zero measure in $\mathcal{F}$; this ensures that there exists a sequence $\{(\bar{g}_j,\bar{t})\}_{j \geq 1} \subset (\mathcal{F} \cdot \bar{g}\times \{\bar{t}\}) \setminus \Pi(\mathcal{L}^{k,\alpha}_{\mathrm{top}}(g,t,\Omega;\Lambda,\kappa_0)) \subset (\mathcal{G}^{k,\alpha} \times \mathbb{R}) \setminus \Pi(\mathcal{L}^{k,\alpha}_{\mathrm{top}}(g,t,\Omega;\Lambda,\kappa_0))$ converging to $(\bar{g},\bar{t})$ in $\mathcal{G}^{k,\alpha} \times \mathbb{R}$ concluding the desired denseness.

    \bigskip

    Supposing this was not the case, we could find a sequence $\{u_i\}_{i \geq 1} \subset \mathcal{Z}$ converging to zero in $L^2(\Sigma)$ such that $u_i$ is not a critical point of $\widetilde{\Pi}$ for each $i \geq 1$. Let us denote $\widetilde{\Pi}(u_i) = f_i \in \mathcal{F}$, so that $f_i \rightarrow 0$ in $C^{k,\alpha}(M)$ (noting that the $L^\infty$ and $C^{k,\alpha}$ norms are equivalent since $\mathcal{F}$ is finite dimensional), for each $i \geq 1$ and by applying Lemma \ref{lemma: paramatrisation and projection} part 2 we have some $\bar{\Omega}_i \in \mathcal{I}((1+f_i)\bar{g},\bar{t})$ such that by Lemma \ref{lemma: compactness of isoperimetric regions} we have, up to a subsequence (not relabelled), $\bar{\Omega}_i \rightarrow \bar{\Omega}$ for some $\bar{\Omega} \in \mathcal{I}(\bar{g},\bar{t})$. By assumption we have that, for each fixed $i \geq 1$, whenever $\{h_j\}_{j \geq 1} \subset \mathcal{F}$ is such that $h_j \rightarrow f_i$ in $C^{k,\alpha}(M)$ we have that
    $$\frac{h_j - f_i}{d_j} \rightarrow \hat{f}_i \neq \{0\} \text{ and } (1 + f_i) \hat{f}_i \in \mathcal{F} \setminus \{0\}.$$
    Writing $\bar{g}_i = (1+f_i)\bar{g}$ for each $i \geq 1$, then by applying Lemma \ref{lemma: compactness for pairs} for each $i \geq 1$ we find non-zero solutions $\hat{u}_i \in C^2_\mathrm{loc}(\bar{\Sigma}_i)$ of slow growth to $\widetilde{L}_{\bar\Sigma_i,\bar{g}_i}\hat{u}_i = \nu_{\bar{\Sigma}_i,\bar{g}_i}(\hat{f}_i) - \frac{1}{|\bar\Sigma_i|_{\bar{g}_i}}\int_{\bar\Sigma_i}\nu_{\bar{\Sigma}_i,\bar{g}_i}(\hat{f}_i)$ 
    with $\int_{\bar\Sigma_i} \hat{u}_i \, dA_{\bar{g}_i} = \frac{n}{2}\int_{\bar\Omega_i} \hat{f}_i \, dV_{\bar{g}_i}$. Now, by considering  $\frac{\hat{f}_i}{\Vert \hat{f}_i\Vert_{C^{k,\alpha}(M)}} \rightarrow \hat{f} \in \mathcal{F} \setminus \{0\}$ (as $(1+f_i) \rightarrow 1$) and projecting $\hat{u}_i$ to $(\mathrm{Ker}\widetilde{L}_{\Sigma,g})^\perp$ we ensure by Theorem \ref{thm: spectral theorem for twisted Jacobi operator} part 3 that $\frac{\hat{u}_i}{\Vert \hat{f}_i\Vert_{C^{k,\alpha}(M)}} \rightarrow \hat{u} \in \mathscr{B}_T(\bar{\Sigma})$  (noting that $\int_{\bar{\Sigma}} \hat{u} \, dA_{\bar{g}} = \frac{n}{2}\int_{\bar{\Omega}} \hat{f} \, dV_{\bar{g}}= 0$ as $\hat{f} \in \mathcal{F}$) for some slow growth function which weakly, and hence by Remark \ref{rem: higher regularity for weak solutions} strongly, solves $\widetilde{L}_{\bar{\Sigma},\bar{g}}\hat{u} = \nu_{\bar{\Sigma},\bar{g}}(\hat{f}) - \frac{1}{|\bar{\Sigma}|_{\bar{g}}}\int_{\bar{\Sigma}}\nu_{\bar{\Sigma},\bar{g}}(\hat{f})$. Noting that then $\nu_{\bar{\Sigma},\bar{g}}(\hat{f}) - \frac{1}{|\bar{\Sigma}|_{\bar{g}}}\int_{\bar{\Sigma}}\nu_{\bar{\Sigma},\bar{g}}(\hat{f})$ is non-zero (since otherwise by Lemma \ref{lemma: paramatrisation and projection} part 1 we would have $\hat{f} = 0$), which also implies that $\hat{u} \neq 0$, we obtain a contradiction to Lemma \ref{lemma: paramatrisation and projection} part 1; hence some neighbourhood of zero in $\mathcal{Z}$ must consist of critical points of $\widetilde{\Pi}$ as desired.

    \bigskip

    As we have shown that for each $(\bar{g},\bar{t}) \in \mathcal{G}^{k,\alpha} \times \mathbb{R}$ there is an approximating sequence in $(\mathcal{F} \cdot \bar{g} \times \{\bar{t}\}) \setminus \Pi(\mathcal{L}^{k,\alpha}_{\mathrm{top}}(g,t,\Omega;\Lambda,\kappa_0))$ and, having previously showed that $\mathcal{G}^{k,\alpha}(g,t,\Omega;\Lambda,\kappa_0)$ was dense in $(\mathcal{G}^{k,\alpha} \times \mathbb{R}) \setminus \Pi(\mathcal{L}^{k,\alpha}_{\mathrm{top}}(g,t,\Omega;\Lambda,\kappa_0))$, we conclude that $\mathcal{G}^{k,\alpha}(g,t,\Omega;\Lambda,\kappa_0)$ is dense in $\mathcal{G}^{k,\alpha} \times \mathbb{R}$. 
    
    \bigskip

    For the final statement, if $\bar{g} \in \mathcal{G}^{k,\alpha}$ then, by intersecting each set with $[\bar{g}] \times \mathbb{R}$, the same covering argument using compactness proceeding and following (\ref{eqn: covering argument for denseness}) shows that $\mathcal{G}^{k,\alpha}(g,t,\Omega;\Lambda,\kappa_0) \cap ([\bar{g}] \times \mathbb{R})$ is in fact dense in $([\bar{g}] \times \mathbb{R}) \setminus \Pi(\mathcal{L}^{k,\alpha}_{\mathrm{top}}(g,t,\Omega;\Lambda,\delta))$. The arguments in the following paragraphs further show that there is an approximating sequence of metric volume pairs whose metric remains in a given conformal class; thus we see that $\mathcal{G}^{k,\alpha}(g,t,\Omega;\Lambda,\kappa_0) \cap ([\bar{g}] \times \mathbb{R})$ is in fact dense in $([\bar{g}] \times \mathbb{R})$. The openness follows since $\mathcal{G}^{k,\alpha}(g,t,\Omega;\Lambda,\kappa_0) \cap ([\bar{g}] \times \mathbb{R})$ is then relatively open as $\mathcal{G}^{k,\alpha}(g,t,\Omega;\Lambda,\kappa_0)$ was shown to be open.
\end{proof}

\subsection{Generic semi-nondegeneracy}\label{subsec: generic semi-nondegeneracy}

We now conclude the section by showing that semi-nondegeneracy is a generic property for isoperimetric regions. With the language and notation used in Appendix \ref{sec: appendix B} we have the following: 

\begin{proposition}\label{prop: cone decomposition for neighbourhoods}
    For any given $(g, t, \Omega) \in \mathcal{T}^{k, \alpha}$, $\beta \in (0, 1/100)$, and $\Lambda \in \mathbb{N}$, there exists a $\delta > 0$, depending on $g, t, \Omega, \beta$, and $\Lambda$, and $l > 0$, depending on $g$, $t$ and $\Omega$, such that for the following $\delta$-neighbourhood 
    \begin{align*}
        \mathcal{T}^{k, \alpha}&(g, t, \Omega; \Lambda, \delta) = \left\{(g^\prime, t^\prime, \Omega^\prime) \in \mathcal{T}^{k, \alpha} \left\vert \begin{array}{l}
			   \Vert g^\prime \Vert_{C^{k, \alpha}} \leq \Lambda, |\Sigma^\prime|_{g^\prime} \leq \Lambda;\\
            \Vert g - g^\prime \Vert_{C^{k-1, \alpha}} \leq \delta, \vert t - t^\prime \vert \leq \delta, |\Omega \Delta \Omega^\prime|_g \leq \delta, 
		\end{array}
         \right. \right\},
    \end{align*}
    we have:
    \begin{enumerate}
        \item There exists a finite collection of $(\theta, \sigma, \beta)$-models $\mathcal{S}$ and an integer $N$, so that any $(g^\prime, t^\prime, \Omega^\prime) \in \mathcal{T}^{k, \alpha}(g, t, \Omega; \Lambda, \delta)$ admits a large scale $(\theta_l, \beta, g, t, \Omega, \mathcal{S}, N)$-cone decomposition. 
        \item There exists a countable collection $\{(g_v, t_v, \Omega_v)\}_{v \in \mathbb{N}} \subset \mathcal{T}^{k, \alpha}(g, t, \Omega; \Lambda, \beta)$ with fixed large scale $(\theta_l, \beta, g, t, \Omega, \mathcal{S}, N)$-cone decomposition, such that every $(g^\prime, t^\prime, \Omega^\prime) \in \mathcal{T}^{k, \alpha}(g, t, \Omega; \Lambda, \beta)$ admits a large scale $(\theta_l, \beta, g, t, \Omega, \mathcal{S}, N)$-cone decomposition whose tree representation is $\beta$-close to that of some $(g_v, t_v, \Omega_v)$. 
    \end{enumerate}
\end{proposition}

\begin{proof}
For the first part, fix $(g, t, \Omega) \in \mathcal{T}^{k, \alpha}$, $\beta \in (0, 1/100)$, $\gamma \in [\beta, 1]$, and $\sigma \in (0, 1/200]$. Let $l \in \mathbb{N}$ be such that, by the discreteness of the densities in (\ref{eqn: discrete densities stable cones}), we have
\begin{equation*}
    \theta_l = \sup_{p \in \mathrm{Sing}(\Sigma)} \theta_{|\Sigma|_g}(p). 
\end{equation*}
In particular, for each $p \in \mathrm{Sing} (\Sigma)$ there exists a radius $r_{p} \in (0, 1)$, such that the balls $\{B^g_{2 r_{p}}(p)\}_{p \in \mathrm{Sing}(\Sigma)}$ are pairwise disjoint, the rescaled pullback of the metric satisfies
\begin{equation} \label{eq 1}
    \vert (2 r_{p})^{-2}(\eta_{p, 2 r_p}^{-1})^* g  - g_\mathrm{eucl} \vert \leq \delta_{l, \gamma, \beta, \sigma}/2, 
\end{equation}
where $\delta_{l, \gamma, \beta, \sigma} > 0$ is as in Theorem \ref{thm: cone decomposition}, and such that there exists a cone $\mathbf{C}_{p} \in \mathscr{C}$ for which, after expressing $\Sigma$ in conical coordinates over it, there holds 
\begin{equation} \label{eq 2}
    d_\mathcal{H}((\eta_{p, 2 r_{p}})( \Sigma) \cap \mathbb{B}_1, \mathbf{C}_{p} \cap \mathbb{B}_1) \leq \delta_{l, \gamma, \beta, \sigma}/2 \qquad \text{and} \qquad \theta_{|\mathbf{C}_{p}|}(0) \leq \theta_l, 
\end{equation}
and such that 
\begin{equation} \label{eq 3}
    \frac{3}{4} \theta_{|\mathbf{C}_{p}|}(0) \leq \theta_{(\eta_{p, 2 r_{p}})_\#\vert \Sigma \vert}(0, 1/2) \qquad \text{and} \qquad \theta_{(\eta_{p, 2 r_{p}})_\# \vert \Sigma \vert}(0, 1) \leq \frac{5}{4} \theta_{|\mathbf{C}_{p}|}(0). 
\end{equation} 
Finally, by potentially taking $r_{p}$ smaller, we ensure that $(\eta_{p,2r_p})(\Sigma)$ is the $C^2$ part of the boundary of a $(\delta_{l, \gamma, \beta, \sigma}/2, 1)$-almost minimiser in $\mathbb{B}_1$. In particular, by taking $\delta > 0$ smaller if necessary in the definition of $\mathcal{T}^{k, \alpha}(g, t, \Omega; \Lambda, \delta)$, we can ensure that each $(g^\prime, t^\prime, \Omega^\prime) \in \mathcal{T}^{k, \alpha}(g, t, \Omega; \Lambda, \delta)$ satisfies \eqref{eq 1}, \eqref{eq 2}, and \eqref{eq 3} with $\Sigma^\prime$ in place of $\Sigma$, and $\delta_{l, \gamma, \beta, \sigma}$ replacing $\delta_{l, \gamma, \beta, \sigma}/2$. Furthermore, since $\Omega' \in \mathcal{I}(g',t')$, we have that $(\eta_{p, r_{p}})(\Sigma^\prime)$ is the $C^2$ part of the boundary of a $(\delta_{l, \gamma, \beta, \sigma}, 1)$-almost minimiser in $\mathbb{B}_1$, while Allard's theorem implies the existence of $C^2$-functions $u : \Sigma \setminus \bigcup_{p \in \mathrm{Sing}(\Sigma)} B^g_{r_p/2}(p)\rightarrow \Sigma^\perp$ so that $\Sigma^\prime \setminus \bigcup_{p} B^g_{r_p}(p) = \mathrm{graph}_\Sigma(u) \setminus B^g_{r_p}(p)$. We can then apply Theorem \ref{thm: cone decomposition} for the parameters  to infer the existence of a $(\theta_l, \beta, \mathcal{S}_{p}, N_{p})$-cone decomposition for $\vert \Sigma \vert_g \mres B^g_{r_p}(p)$, as well as a large scale $(\theta_l, \beta, g, t, \Omega, \mathcal{S}, N)$-cone decomposition of $(g^\prime, t^\prime, \Omega^\prime)$. 

\bigskip

By part 1 above, for each $(g^\prime, t^\prime, \Omega^\prime) \in \mathcal{T}^{k, \alpha}(g, t, \Omega; \Lambda, \delta)$ there exists a large scale cone decomposition as in Definition \ref{definition: large scale cone dec}, which then has a corresponding tree representation as in Definition \ref{definition: tree representation of cone dec}. Finiteness of $N, \mathcal{S}$, and the discreteness of the set of densities of stable minimal hypercones in \eqref{eqn: discrete densities stable cones}, imply that there are only finitely many coarse tree representations as in Definition \ref{definition: tree rep of large scale cone dec}. For a fixed coarse tree representation, $T = (V, E)$, one then considers the set of all triples associated to this coarse tree representation, $\mathcal{L}^{\prime}(T)$, and finds a countable covering of this space; to do so we cover each node of the coarse tree. By definition of coarse tree representation, as in Definition \ref{definition: coarse tree representation of cone dec}), we have that all tree representations of triples in $\mathcal{L}^\prime(T)$ have the same root node. Arguing exactly as in the proof of \cite[Theorem 9.6 part 2]{LW25}, we see that all tree representations of triples in $\mathcal{L}^\prime(T)$ are contained in the countable union $\mathrm{PC}(T)$, where the notation is as in \cite[(74) in the proof of Theorem 9.6]{LW25}. Thus, we can then find a countable subset, $\mathcal{L}^{\prime \prime}(T)$, of $\mathcal{L}^\prime(T)$ such that each element of $\mathrm{PC}(T)$ contains at least one tree representation of a triple in $\mathcal{L}^\prime(T)$, and contains exactly one tree representation of a triple in $\mathcal{L}^{\prime \prime}(T)$. Taking a union over the finitely many coarse tree representations allows us to conclude, with the $\beta$-closeness property following by construction of the coverings. In particular, we note that no change is needed in order to account for enclosed volumes (since they play no role in the tree representation of a cone decomposition), and moreover one can take the multiplicities to be one throughout, when reasoning as in the proof of \cite[Theorem 9.6]{LW25}.
\end{proof}

\begin{definition}
    Given $(g,t,\Omega) \in \mathcal{T}^{k,\alpha}$, $\beta \in (0,1/100)$, $\Lambda \geq 1$, $\delta > 0$ and $v \in \mathbb{N}$ as in Proposition \ref{prop: cone decomposition for neighbourhoods}, we define the \textbf{intermediate neighbourhoods}, denoted $\mathcal{T}^{k,\alpha}_{v}(g,t,\Omega; \Lambda,\delta,\beta)$,  to be the set of triples $(g^\prime, t^\prime, \Omega^\prime) \in  \mathcal{T}^{k, \alpha}(g, t, \Omega; \Lambda, \delta)$ admitting a large scale $(\theta_l, \beta, g, t, \Omega, \mathcal{S}, N)$-cone decomposition whose tree representation is $\beta$-close to that of $(g_v, t_v, \Omega_v)$.
\end{definition}

In particular, with this notation Proposition \ref{prop: cone decomposition for neighbourhoods} implies that we have the following decomposition of the space of triples:
    \begin{equation}\label{eqn: covering of neighbourhoods by intermediates}
        \mathcal{T}^{k, \alpha} (g, t, \Omega; \Lambda, \delta) = \bigcup_{v \in \mathbb{N}}  \mathcal{T}^{k,\alpha}_{v}(g,t,\Omega; \Lambda,\delta,\beta).
    \end{equation}
    
These intermediate neighbourhoods satisfy the following:

\begin{lemma}[Properties of intermediate neighbourhoods]\label{lemma: properties of intermediate neighbourhoods}
    Given $(g,t,\Omega) \in \mathcal{T}^{k,\alpha}$, $\beta \in (0,1/100)$, $\Lambda \geq 1$, $\delta > 0$, and $v \in \mathbb{N}$, then, with the notation in Proposition \ref{prop: cone decomposition for neighbourhoods}, we have that:
    \begin{enumerate}
        \item The space $\mathcal{T}^{k,\alpha}_{v}(g,t,\Omega; \Lambda,\delta,\beta)$ is compact.

        \item Given any function $\kappa : \mathcal{T}^{k,\alpha}\times (0,\infty) \rightarrow (0,\infty)$, for each $v \in \mathbb{N}$ there exists some $N_v \in \mathbb{N}$ and $\{(g_{v, l}, t_{v, l}, \Omega_{v, l})\}_{l = 1, \ldots, N_v} \subset \mathcal{T}^{k,\alpha}_{v}(g,t,\Omega; \Lambda,\delta,\beta)$ such that   
   \begin{equation*}
        \mathcal{T}^{k,\alpha}_{v}(g,t,\Omega; \Lambda,\delta,\beta) \subset \bigcup_{l = 1}^{N_v} \mathcal{L}^{k, \alpha}(g_{v, l}, t_{v, l}, \Omega_{v, l}; \Lambda, \kappa_{v,l}), 
   \end{equation*}
   where $\kappa_{v,l} = \kappa(g_{v, l}, t_{v, l}, \Omega_{v, l},\Lambda) > 0$ for each $l = 1,\dots, N_v$.
    \end{enumerate}
\end{lemma}

\begin{proof} For part 1, we show sequential compactness in a similar manner to the proof of Lemma \ref{lemma: compactness of pseudo neighbourhoods}. Supposing that we had a sequence $\{(g_j,t_j,\Omega_j)\}_{j \geq 1} \subset \mathcal{T}^{k,\alpha}_{v}(g,t,\Omega; \Lambda,\delta,\beta)$, the uniform bounds, $\Vert g_j \Vert_{C^{k, \alpha}} \leq \Lambda$, $\Vert g - g_j \Vert_{C^{k - 1, \alpha}} \leq \delta$, $\vert \Omega \Delta \Omega_j \vert \leq \delta$, and $\vert t - t_j \vert \leq \delta$, along with the compactness of the embedding $C^{k, \alpha}(M) \hookrightarrow C^{k - 1, \alpha}(M)$ and Lemma \ref{lemma: compactness of isoperimetric regions}, imply the existence $(g_\infty, t_\infty, \Omega_\infty) \in \mathcal{T}^{k,\alpha}(g,t,\Omega; \Lambda,\delta)$, such that up to a subsequence (not relabelled) we have $(g_j,t_j,\Omega_j) \rightarrow (g_\infty,t_\infty,\Omega_\infty)$.

\bigskip

To conclude we need to show that $(g_\infty,t_\infty,\Omega_\infty) \in \mathcal{T}^{k,\alpha}_{v}(g,t,\Omega; \Lambda,\delta,\beta)$. The existence of a large scale $(\Lambda, \beta, g, t, \Omega, \mathcal{S}, N)$-cone decomposition follows from Proposition \ref{prop: cone decomposition for neighbourhoods} part 1, we now show that the tree representations of the large-scale cone decompositions of $(g_\infty, t_\infty, \Omega_\infty)$ and $(g_v, t_v, \Omega_v)$ are $\beta$-close. Consider a subtree rooted at an arbitrary $\alpha$ child of the root node, $(\Sigma, g, \{p_{\alpha}\}, \{r_{\alpha}\})$. Each $\vert \Sigma_j \vert_{g_j} \mres B^g_{r_{\alpha}}(p_{\alpha})$ admits a cone decomposition whose tree representation is $\beta$-close to the one of $\vert \Sigma_v \vert_{g_v} \mres B^g_{r_{\alpha}}(p_{\alpha})$, hence the number of strong-cone regions, and the number of smooth regions stay constant along the sequence. Furthermore, the densities of the corresponding cones are the same, as well as the smooth models $S_{s_b}$, and they also stay constant along the sequence. 

\bigskip

Consider then a strong cone region of the cone decomposition of $\vert \Sigma_v \vert \mres B^g_{r_{\alpha}}(p_{\alpha})$, as well as the corresponding sequence of strong cone regions $\beta$-close to it, arising from the sequence of the $\vert \Sigma_j \vert_{g_j}$; let us denote by $\mathbf{C}_v$ the cone for the former, while $\mathbf{C}_j$ the cones for the latter. In particular, the densities of the cones $\mathbf{C}_j$ are bounded from above and from below uniformly, in terms of $\theta_{\mathbf{C}_v}(0)$ and $\beta$, thus implying that we can extract, by the compactness of stable minimal hypercones discussed in Subsection \ref{subsec: notation}, a limiting cone, $\mathbf{C}_\infty$, having the same density as the sequence, and such that the corresponding links converge smoothly and with multiplicity one; this in turn implies $d_\mathcal{H}(\mathbf{C}_\infty \cap \partial \mathbb{B}_1, \mathbf{C}_v \cap \partial \mathbb{B}_1) \leq \beta$. The sequences of centres $\{x_j\}_{j \geq 1}$, and radii $\{\rho_j\}_{j \geq 1}, \{R_j\}_{j \geq 1}$ are also uniformly bounded, and we can therefore extract further subsequences (not relabelled) converging to limit points $x_\infty, \rho_\infty, R_\infty$ respectively. These limit points are $\beta$-close to $x_v, \rho_v,$ and $R_v$ as the sequences were $\beta$-close. A combination of Allard's theorem and unique continuation implies that the limiting varifold, $\vert \Sigma_\infty \vert_{g_\infty}$, restricted to the limiting annulus is a strong cone region with cone $\mathbf{C}_\infty$. This is then a node of the tree representation of the cone decomposition of $\vert \Sigma_\infty \vert_{g_\infty} \mres B^g_{ r_{\alpha}}(p_{\alpha}).$ One can argue similarly for smooth regions which, by definition of being $\beta$-close, have the same smooth models. Thus, every element of the cone decomposition of $\vert \Sigma_j \vert_{g_j}$ converges to the corresponding element of the cone decomposition of $\vert \Sigma_\infty \vert_{g_\infty}$. To conclude, we need to ensure that no further node (of either type I or type II) arises from the limiting procedure, however, as cone regions and smooth regions cover the whole domain of every $\vert \Sigma_j \vert_{g_j}$, we obtain a full covering of the domain of $\vert \Sigma_\infty \vert_{g_\infty}$, thus implying that no extra node is created and concluding the proof that $(g_\infty,t_\infty,\Omega_\infty) \in \mathcal{T}^{k,\alpha}_{v}(g,t,\Omega; \Lambda,\delta,\beta)$; hence $\mathcal{T}^{k,\alpha}_{v}(g,t,\Omega; \Lambda,\delta,\beta)$ is compact as desired.   

\bigskip

For part 2, we utilise the compactness established in part 1 to show that, for any $(g^\prime, t^\prime, \Omega^\prime) \in  \mathcal{T}_{v}^{k, \alpha}(g, t, \Omega; \Lambda, \delta, \beta)$ and $\delta^\prime > 0$, the pseudo-neighbourhood $\mathcal{L}^{k, \alpha}(g^\prime, t^\prime, \Omega^\prime; \Lambda^\prime, \delta^\prime)$ contains an open neighbourhood of $(g^\prime, t^\prime, \Omega^\prime)$ in $\mathcal{T}_{v}^{k, \alpha}(g, t, \Omega; \Lambda, \delta, \beta)$. Assuming for a contradiction this failed, then in particular there is $(g^\prime, t^\prime, \Omega^\prime) \in \mathcal{T}_{v}^{k, \alpha}(g, t, \Omega; \Lambda, \delta, \beta)$ and $\delta^\prime > 0$ such that the corresponding pseudo-neighbourhood $\mathcal{L}^{k, \alpha}(g^\prime, t^\prime, \Omega^\prime; \Lambda, \delta^\prime)$ does not contain any open neighbourhood. Thus, considering a countable neighbourhood basis around $(g^\prime, t^\prime, \Omega^\prime)$, there exists a sequence $$\{(g_j, t_j, \Omega_j)\}_{j \geq 1} \subset \mathcal{T}_{v}^{k, \alpha}(g, t, \Omega; \Lambda, \delta, \beta) \setminus \mathcal{L}^{k, \alpha}(g^\prime, t^\prime, \Omega^\prime; \Lambda, \delta^\prime)$$ converging to $(g^\prime, t^\prime, \Omega^\prime)$. In particular, for $j \geq 1$ sufficiently large, we have that: 
    \begin{itemize}
        \item $g_j$ is a $C^{k, \alpha}$ metric on $M$ satisfying $\Vert g_j \Vert_{C^{k, \alpha}} \leq \Lambda$, and $\Vert g_j - g^\prime \Vert_{C^{k - 1, \alpha}} \leq \delta^\prime$.
        \item $\Omega_j \in \mathcal{I}(g_j, t_j)$ with $\vert \Omega_j \Delta \Omega^\prime \vert_g \leq \delta^\prime$ and $\vert t^\prime - t_j \vert \leq \delta^\prime$. 
        \item Every $(g_j, t_j, \Omega_j)$ admits a large scale $(\Lambda, \beta, g, t, \Omega, \mathcal{S}, N)$-cone decomposition whose tree representation is $\beta$-close to that of $(g_v, t_v, \Omega_v)$. The compactness in part 1 of this lemma implies that $(g^\prime, t^\prime, \Omega^\prime)$ satisfies the same property. Thus, for every $x^\prime \in \mathrm{Sing}(\Sigma^\prime)$ there exists $x_j \in \mathrm{Sing}(\Sigma_j)$ with $\theta_{\vert \Sigma^\prime \vert_{g'}}(p^\prime) = \theta_{\vert \Sigma_j \vert_{g_j}}(p_j)$. This follows from the definition of being $\beta$-close for two $(\theta, \beta, \mathcal{S}, N)$-tree representation being close, see Definition \ref{definition: gamma-close tree representations of cone dec}, as the labels of the coarse tree representation have to coincide. 
    \end{itemize}
    In particular, these conditions above imply that by definition of the pseudo-neighbourhoods we have $(g_j, t_j, \Omega_j) \in \mathcal{L}^{k, \alpha}(g^\prime, t^\prime, \Omega^\prime; \Lambda, \delta^\prime)$ for sufficiently large $j \geq 1$, giving the desired contradiction.

    \bigskip

    To conclude part 2, given $\kappa : \mathcal{T}^{k,\alpha}\times (0,\infty) \rightarrow (0,\infty)$ we consider the cover of $\mathcal{T}_v^{k,\alpha}(g,t,\Omega;\Lambda,\delta,\beta)$ formed by open neighbourhoods of each $(g',t',\Omega')$ of radius at most $\kappa(g',t',\Omega',\Lambda) > 0$ (the existence of which is guaranteed by the previous paragraph). The compactness established in part 1 above then guarantees the desired conclusion.
\end{proof}

\begin{lemma}\label{lem: covering of triples by pseudo neighbourhoods}
    Given any function $\kappa : \mathcal{T}^{k,\alpha}\times (0,\infty) \rightarrow (0,\infty)$, there is a countable collection of tuples, $\{(g_j,t_j,\Omega_j,\Lambda_j)\}_{j \geq 1} \subset \mathcal{T}^{k,\alpha} \times (0,\infty)$, such that $$\mathcal{T}^{k,\alpha} = \bigcup_{j \geq 1} \mathcal{L}^{k,\alpha}(g_j,t_j, \Omega_j;\Lambda_j,\kappa_j),$$
    where $\kappa_j = \kappa(g_j,t_j,\Omega_j,\Lambda_j)$.
\end{lemma}

\begin{proof}
    For $\Lambda  \geq 1 $ we let 
    \begin{align*}
        \mathcal{T}^{k, \alpha}(\Lambda) = \{(g, t, \Omega) \in \mathcal{T}^{k, \alpha}; \Vert g \Vert_{C^{k, \alpha}} \leq \Lambda,\vert \Sigma \vert_g \leq \Lambda, t \in [(2\Lambda)^{-1}, \vert M \vert_g - (2\Lambda)^{-1}] \}, 
    \end{align*} 
    which is compact; indeed, for a sequence $\{(g_j, t_j, \Omega_j)\}_{j \geq 1} \subset \mathcal{T}^{k, \alpha}(\Lambda)$ the compactness of the embedding $C^{k, \alpha}(M) \hookrightarrow C^{k - 1, \alpha}(M)$ and the  set $[(2\Lambda)^{-1}, \vert M \vert_g - (2\Lambda)^{-1}]$ as well as Lemma \ref{lemma: compactness of isoperimetric regions}, imply the existence of a subsequential limit. In particular, there exist some finite collection of triples, $\{(g_i,t_i,\Omega_i)\}_{i = 1}^{K_\Lambda}\subset \mathcal{T}^{k,\alpha}(\Lambda)$, such that
    \begin{equation}\label{eqn: mass bound triple equality}
        \mathcal{T}^{k, \alpha}(\Lambda) = \bigcup_{i = 1}^{K_\Lambda} \mathcal{T}^{k, \alpha}(g_i, t_i, \Omega_i; \Lambda, \delta_i)
    \end{equation}
    for some $K_\Lambda \in \mathbb{N}$ and with $\delta_i > 0$ for each $i = 1,\dots, K_\Lambda$. Consequently, we infer the following decomposition of the space of triples
    \begin{align*}
        \mathcal{T}^{k, \alpha} & = \bigcup_{\Lambda \in \mathbb{N}} \mathcal{T}^{k, \alpha}(\Lambda) \\
        & = \bigcup_{\Lambda \in \mathbb{N}}\bigcup_{i = 1}^{K_\Lambda} \mathcal{T}^{k, \alpha}(g_i, t_i, \Omega_i; \Lambda, \delta_i) \\
        & = \bigcup_{\Lambda \in \mathbb{N}} \bigcup_{i = 1}^{K_\Lambda} \bigcup_{v = 1}^{\infty}  \mathcal{T}_{v}^{k, \alpha}(g_i, t_i, \Omega_i; \Lambda, \delta_i,\beta) \\
         & \subset \bigcup_{\Lambda \in \mathbb{N}} \bigcup_{i = 1}^{K_\Lambda} \bigcup_{v = 1}^{\infty} \bigcup_{l = 1}^{N_v} \mathcal{L}^{k, \alpha}(g_{v, i, l}, t_{v, i, l}, \Omega_{v, i, l}; \Lambda, \kappa_{v,i,l}); 
    \end{align*}
    here the second equality follows from (\ref{eqn: mass bound triple equality}), the third equality from (\ref{eqn: covering of neighbourhoods by intermediates}) after applying Proposition \ref{prop: cone decomposition for neighbourhoods} part 2, and the final inclusion from Lemma \ref{lemma: properties of intermediate neighbourhoods} part 2 where we set $\kappa_{v,i,l} = \kappa(g_{v,i,l},t_{v,i,l},\Omega_{v,i,l},\Lambda) > 0$. The desired result then follows by reindexing, after noting that the final inclusion must be an equality since the pseudo-neighbourhoods are subsets of $\mathcal{T}^{k,\alpha}$. 
\end{proof}

\begin{theorem}\label{thm: generic semi-nondegeneracy}
    Let $\mathcal{U}_0^{k,\alpha}$ be the set of $(g,t) \in \mathcal{G}^{k,\alpha} \times \mathbb{R}$ such that every isoperimetric region with respect to the metric $g$ of enclosed volume $t$ is semi-nondegenerate, then $\mathcal{U}_0^{k,\alpha}$ is a generic subset of $\mathcal{G}^{k,\alpha} \times \mathbb{R}$; namely, semi-nondegeneracy is a generic property for metric volume pairs. Moreover, if $\bar{g} \in \mathcal{G}^{k,\alpha}$ then $\mathcal{U}^{k,\alpha}_0 \cap ([\bar{g}] \times \mathbb{R})$ is generic in $[\bar{g}] \times \mathbb{R}$.
\end{theorem}

\begin{proof}
    Let $\kappa : \mathcal{T}^{k,\alpha} \times (0,\infty) \rightarrow (0,\infty)$ be the function taking $(g,t,\Omega,\Lambda) \in \mathcal{T}^{k,\alpha} \times (0,\infty)$ to $\kappa_0$ as determined by Lemma \ref{lemma: generic pairs for neighbourhoods}; applying Lemma \ref{lem: covering of triples by pseudo neighbourhoods} for this function guarantees the existence of a countable collection $\{(g_j,t_j,\Omega_j,\Lambda_j)\}_{j \geq 1} \subset \mathcal{T}^{k,\alpha} \times (0,\infty)$ such that $\mathcal{T}^{k,\alpha} = \bigcup_{j \geq 1} \mathcal{L}^{k,\alpha}(g_j,t_j, \Omega_j;\Lambda_j,\kappa_j),$
    where $\kappa_j = \kappa(g_j,t_j,\Omega_j,\Lambda_j)$. By Lemma \ref{lemma: generic pairs for neighbourhoods}, for each $(g_j,t_j,\Omega_j)$ as above there is some open and dense subset, $\mathcal{G}^{k,\alpha}(g
    _j,t_j,\Omega_j;\Lambda_j,\kappa_j)$, of $\mathcal{G}^{k,\alpha} \times \mathbb{R}$ and hence $\cap_{j \geq 1}\mathcal{G}^{k,\alpha}(g
    _j,t_j,\Omega_j;\Lambda_j,\kappa_j)$ is a countable intersection of open and dense sets in $\mathcal{G}^{k,\alpha} \times \mathbb{R}$. Moreover, if $(g,t) \in \cap_{j \geq 1}\mathcal{G}^{k,\alpha}(g
    _j,t_j,\Omega_j;\Lambda_j,\kappa_j)$ and $\Omega \in \mathcal{I}(g,t)$ then we have $(g,t,\Omega) \in \mathcal{L}^{k,\alpha}(g_j,t_j,\Omega_j;\Lambda_j,\kappa_j)$ for some $j \geq 1$, and thus $\Omega$ is semi-nondegenerate since $(g,t) \in \mathcal{G}^{k,\alpha}(g_j,t_j,\Omega_j;\Lambda_j,\kappa_j)$; hence we see that $\cap_{j \geq 1}\mathcal{G}^{k,\alpha}(g
    _j,t_j,\Omega_j;\Lambda_j,\kappa_j) \subset \mathcal{U}^{k,\alpha}_0$ as desired. 

    \bigskip

    For the final statement, if $\bar{g} \in \mathcal{G}^{k,\alpha}$ then by Lemma \ref{lemma: generic pairs for neighbourhoods} we have that $\mathcal{G}^{k,\alpha}(g
    _j,t_j,\Omega_j;\Lambda_j,\kappa_j) \cap ([\bar{g}] \times \mathbb{R})$ is open and dense in $[\bar{g}] \times \mathbb{R}$ for each $j \geq 1$, and thus $\cap_{j \geq 1} \mathcal{G}^{k,\alpha}(g
    _j,t_j,\Omega_j;\Lambda_j,\kappa_j) \cap ([\bar{g}] \times \mathbb{R)} \subset \mathcal{U}^{k,\alpha}_0 \cap ([\bar{g}] \times \mathbb{R})$ is a countable of intersection of open and dense sets in $[\bar{g}] \times \mathbb{R}$ as desired.
\end{proof}

\begin{remark}
    We note that, by Remark \ref{rem: slow growth remarks}, if $(g,t) \in \mathcal{U}_0^{k,\alpha}$ and $\Omega \in \mathcal{I}(g,t)$ is such that $\mathrm{Sing}(\Sigma) = \emptyset$, then in particular $\Sigma$ is nondegenerate.
\end{remark}

\section{Singular capacity for isoperimetric regions}\label{sec: singular capacity for isoperimetric regions}

We now introduce a notion to count the number of potential singularities that can arise along a sequence of converging isoperimetric regions.

\subsection{Definitions and properties}

The following definitions are made inductively over the densities of stable minimal hypercones which, as noted in (\ref{eqn: discrete densities stable cones}) in Subsection \ref{subsec: notation}, are discrete: 

\begin{definition}[Singular capacity for volume-constrained minimisers]\label{def: singular capacity}
Given $\delta > 0$, a Riemannian metric, $g$, on $\mathbb{B}_\delta$, and $\Omega\in \mathrm{VCM}(\delta,g)$, we define the \textbf{singular capacity} of $\Omega$ in an open set $U \subset \mathbb{B}_\delta$ by setting 
        $${\bf SCap} (\Omega,U,g) = \sum_{p \in \mathrm{Sing}(\Sigma) \cap U} {\bf SCap} ({\bf C}_p \Sigma) \in [0,\infty],$$
        where $\mathbf{C}_p\Sigma$ is the unique tangent cone to $\Sigma = \partial \Omega$ at $p \in \mathrm{Sing}(\Sigma)$, and the singular capacity, $\mathbf{SCap}(\mathbf{C})$, of a stable minimal hypercone, $\mathbf{C} \in \mathscr{C}$, is defined inductively by the following:
        
    \begin{enumerate}
        \item For a hyperplane, ${\bf P} \in \mathscr{C}_{\omega_7}$, we set $\mathbf{SCap}({\bf P}) = 0$.

        \item For all non-planar cones, ${\bf C }\in \mathscr{C} \setminus \mathscr{C}_{\omega_7}$, we set
\begin{equation*}
    \mathbf{SCap}({\bf C}) = 1 + \sup \left\{ \limsup_{j \to \infty} 
 {\bf SCap} (\Omega_j,\mathbb{B}_1,g_j)\right\},
\end{equation*}
    where the supremum is taken over all sequences of pairs, $\{(g_j, \Omega_j)\}_{j \geq 1}$, for $g_j$ a $C^{k,\alpha}$ metric on $\mathbb{B}_2$ and $\Omega_j \in\mathrm{VCM}(2,0,g_j)$ such that both:
\begin{enumerate}[(i)]
    \item  $g_j \to g_{\mathrm{eucl}}$ in $C^4(\mathbb{B}_1)$, 
          and $\Omega_j \to \mathbf{E}^+$ as $j \to \infty$, where $\mathbf{E}^+\in \mathcal{C}(\mathbb{B}_2)$ with $\partial  \mathbf{E}^+= {\bf C}$.
          
    \item $\theta_{|\Sigma_j|}(p_j) < \theta_C(0)$ for any $p_j \in \mathrm{Sing}(\Sigma_j) \cap \mathbb{B}_1$.
    \end{enumerate}
    If there is no such sequence of pairs satisfying the above for ${\bf C}$, we define $\mathbf{SCap}({\bf C}) = 1$.
\end{enumerate}
\end{definition}

As observed in Subsection \ref{subsec: preliminaries}, isoperimetric regions are locally volume constrained minimisers and so we make the following definition:

\begin{definition}[Singular Capacity for isoperimetric regions]\label{def: singular capacity for isoperimetric regions}
   Given $(g,t) \in \mathcal{G}^{k,\alpha} \times \mathbb{R}$ and an open set $U \subset M$, we define the \textbf{singular capacity} of $\Omega\in \mathcal{I}(g,t)$ in $U$ by setting 
        $${\bf SCap} (\Omega,U,g) = \sum_{p\in \mathrm{Sing}(\Sigma) \cap U} {\bf SCap} ({\bf C}_p \Sigma) \in [0,\infty],$$
        where $\mathbf{C}_p\Sigma$ denotes the unique tangent cone to $\Sigma$ at $p \in \mathrm{Sing}(\Sigma)$. We then define
        $$ {\bf SCap} (g,t) = \sup \{ {\bf SCap} (\Omega,M,g) \, | \, \Omega\in \mathcal{I}(g,t) \}.$$
\end{definition}

\begin{remark}\label{rem: singcap zero for smooth}
    We note that, according to the definition above, for any ${\bf C}\in \mathscr{C}_{\theta_1},$ we have ${\bf SCap}({\bf C}) = 1 $, and for any ${\bf C}\in \mathscr{C}_{\theta_2},$ we have $
    \mathbf{SCap}({\bf C}) = 1 + \sup \left\{ \limsup_{j \to \infty} 
\#(\mathrm{Sing}(\Omega_j)\cap \mathbb{B}_1)\right\},$ where the supremum is taken over all sequences as in Definition \ref{def: singular capacity} part 2 and $\#(\mathrm{Sing}(\Omega_j)\cap \mathbb{B}_1)$ denotes the number of elements in the set $(\mathrm{Sing}(\Omega_j)\cap \mathbb{B}_1)$. Also, if $|\Sigma|_g$ is sufficiently close to either zero or $|M|_g$, then by Remark \ref{rem: Morgan-Johnson} we have that $\mathbf{SCap}(\Omega,M,g) = 0$ since $\mathrm{Sing}(\Sigma) = \emptyset$.
\end{remark}

We now establish two technical lemmas to aid us in working with the singular capacity:

\begin{lemma}\label{Lemma: Density drop}
        Given a Riemannian metric, $g$, on $\mathbb{B}_5$, $\Omega \in\mathrm{VCM}(5,g)$, and ${\bf C}\in \mathscr{C}_{\Lambda}$ for some $\Lambda > 1$ with ${\bf E}^+\in \mathcal{C}(\mathbb{B}_5)$ such that $\partial {\bf E}^+ = {\bf C}$, then for each $\epsilon\in (0,1)$ there exists $\delta > 0$, depending on $\varepsilon$ and $\Lambda$, such that if both $\Vert g-g_{\mathrm{eucl}}\Vert_{C^4} \le \delta$ and $|\Omega\Delta{\bf E^+}|\le \delta$, then $\mathrm{Sing}(\Sigma) \cap \mathbb{B}_4 \subset \mathbb{B}_\epsilon$, and for each $x\in \mathbb{B}_1\cap \overline\Sigma$ at least one of the following holds:
        \begin{enumerate}
            \item $\theta_{|\Sigma|_g}(x)\le \theta_{\bf C}(0) - 2 \delta$.
            \item $\mathrm{Sing}(\Sigma)\cap \mathbb{B}_4 \subset \{x\}$.
        \end{enumerate}
    \end{lemma}
     \begin{proof}
        Assume for a contradiction that there exists $\varepsilon > 0$, 
        $\Omega_j \subset\mathrm{VCM}(5,g_j)$ for each $j \geq 1$, and $\{{\bf E}_j^+\}_{j \geq 1}\subset \mathcal{C}(\mathbb{B}_5)$ with $\partial \mathbf{E}_j^+ = \mathbf{C}_j$ for $\mathbf{C}_j \in \mathscr{C}_\Lambda$ such that $g_j \rightarrow g_\mathrm{eucl}$ in $C^4$ and $|\Omega_j\Delta{
        {\bf E}_j^+
        }|\to 0$ as $j \rightarrow \infty$, but $\mathrm{Sing}(\Sigma_j) \cap (\mathbb{B}_4 \setminus \mathbb{B}_\epsilon ) \neq \emptyset$ for sufficiently large $j \geq 1$. By the compactness of $\mathscr{C}_\Lambda$, as noted in Subsection \ref{subsec: notation}, there exists ${\bf E}^+ \in \mathcal{C}(\mathbb{B}_5)$ such that ${\bf E}_j^+\to {\bf E}^+$ with $\partial \mathbf{E}^+ = \mathbf{C} \in \mathscr{C}_\Lambda$. Thus, by applying Allard's theorem, $\mathrm{Sing}(\Sigma_j) \subset \mathbb{B}_\epsilon$ for sufficiently large $j \geq 1$, contradicting our assumption.

        \bigskip
        
        Let $x_j\in \overline{\Sigma}_j \cap \mathbb{B}_1$ and suppose that for a sequence as in the above paragraph we have for infinitely many $j \geq 1$ the existence of $y_j \in \mathrm{Sing}(\Sigma_j) \cap \mathbb{B}_4 \setminus \{x_j\}$ with
        $$ \limsup_{j \rightarrow \infty} \left[ \theta_{|\Sigma_j|_{g_j}}(x_j)- \theta_{\bf C}(0) \right]\ge 0.$$
        If ${\bf C} $ is a hyperplane, then by applying Allard's theorem, the $\Sigma_j$ are regular for sufficiently large $j \geq 1$, contradicting the assumption that the $y_j$ exist. Thus, $\bf C$ is not a hyperplane and so we have that by Allard's theorem and the upper semi-continuity of the density that
        $$ \limsup_{j \rightarrow \infty} \theta_{|\Sigma_j|_{g_j}}(x_j)\ge \theta_{\bf C}(0) > 1;$$
        hence $x_j\in \mathrm{Sing}(\Sigma_j)$ for sufficiently large $j \geq 1$.
        
        \bigskip
        
        By the first part of the lemma we know that both $x_j,y_j\to 0$ as $j \rightarrow \infty$. Now consider for each $j \geq 1$ the rescaled varifolds $ \widetilde{\Sigma}_j =\left(
        \eta_{x_j,r_j}\right)_\# |\Sigma_j|_{g_j}$ where $r_j= \frac{2}{5}\mathrm{dist}_{g_j}(x_j,y_j)$. Note that as $j\to \infty$, by writing $\tilde{g}_j = (r_j)^{-2}  (\eta_{x_j,r_j}^{-1})^*g_j$, we have that the following properties hold:
        \[
        \begin{cases}
            \Vert \tilde{g}_j- g_{eucl}\Vert_{C^4(\mathbb{B}_5)}\to 0\\
            \Vert \widetilde{\Sigma}_j\Vert (\mathbb{B}_5)\le 2\cdot 5^n \omega_n\Lambda\\
            \theta_{|\Sigma_j|_{\tilde{g}_j}}(0, 4) - \theta_{|\Sigma_j|_{\tilde{g}_j}}(0, 1)\to 0\\
        \end{cases}.
        \] 
        We now claim that the above properties imply that $y_j \notin \mathrm{Sing}(\Sigma_j)$ for sufficiently large $j \geq 1$, yielding a contradiction. Precisely, we now show that 
        there exists $\delta > 0$, depending on $\Lambda > 0$, such that if $\Omega\in\mathrm{VCM}(5,g)$ satisfies the following properties:
        \begin{equation}\label{eqn: claim in lemma}
        \begin{cases}
            \Vert g- g_{eucl}\Vert_{C^4(\mathbb{B}_5)}\le \delta\\
            \Vert \Sigma\Vert (\mathbb{B}_5)\le 2\cdot 5^n \omega_n\Lambda\\
            \theta_{|\Sigma|}(0, 4) - \theta_{|\Sigma|}(0, 1)\le \delta\\
            H_{\Sigma} \leq \delta
        \end{cases},
        \end{equation}
        then there is some ${\bf C}\in \mathscr{C}$ such that $\Sigma$ is regular near ${\bf C}\cap\mathbb{A}(2,3)$. If not, then there exist $\Lambda >0$ and a sequence $\{\Omega_j\}\subset\mathrm{VCM}(5,g_j)$ satisfying (\ref{eqn: claim in lemma}) for some $\delta_j \rightarrow 0$ but such that $\mathrm{Sing}(\Sigma_j) \cap \mathbb{A}(2,3) \neq \emptyset$. However, by the compactness of volume constrained minimisers, as discussed in Subsection \ref{subsec: preliminaries}, up to a subsequence (not relabelled) we have that $\Omega_j \rightarrow \Omega \in\mathrm{VCM}(4,g_\mathrm{eucl})$, which will then in fact be locally area minimising (since $H_{\Sigma_j} \rightarrow 0$),  with $\theta_{|\Sigma|_{g_{\mathrm{eucl}}}}(0, 4) - \theta_{|\Sigma|_{g_{\mathrm{eucl}}}}(0, 1) = 0$, but then the monotonicity formula ensures that $|\Sigma|_{g_{\mathrm{eucl}}}$ is a minimal hypercone in $\mathbb{B}_4$; hence by Allard's theorem $\Sigma_j$ is regular in $\mathbb{A}(2,3)$ for sufficiently large $j \geq 1$, a contradiction. 
        \end{proof}

    By using the claim in the proof of Lemma \ref{Lemma: Density drop}, we can prove the following upper bound on the number of singularities for singularities that approach a cone with a given density bound:
    
    \begin{lemma}\label{lemma: N (Lambda)}
    For each $\Lambda > 1$ there exists a constant $N(\Lambda) \geq 1$ such that, for every 
    ${\bf C} \in \mathscr{C}_{\Lambda}$ and all sequences $\Omega_j \in\mathrm{VCM}(5,g_j)$ such that   $g_j \to g_{\mathrm{eucl}}$ in $C^4(\mathbb{B}_5)$, 
    and $\Omega_j \to {\bf E}^+$, where ${\bf E}^+\in \mathcal{C}(\mathbb{B}_5)$ with $\partial  {\bf E}^+= {\bf C} \cap \mathbb{B}_5$, as $j \to \infty$ we have that
\[  
\limsup_{j \to \infty} \# \big( \mathrm{Sing}(\Sigma_j) \cap \mathbb{B}_4 \big) 
\leq N(\Lambda).
\]
\end{lemma}

\begin{proof}
The proof is similar \cite[Lemma C.6]{LW25}, except that we use the compactness theorem for volume constrained minimisers and apply our Lemma \ref{Lemma: Density drop}. Precisely, we can establish the corollary by induction on the densities of stable minimal hypercones, which by (\ref{eqn: discrete densities stable cones}) are discrete.

\bigskip

Let $\Lambda_k = \omega_7\theta_k$ for $k \in \mathbb{N}$, and note that $\Lambda_0 = \omega_7$ and hence any $\mathbf{C} \in \mathscr{C}_{\Lambda_0}$ is a hyperplane, thus by Allard's theorem we can set $N(\Lambda_0) = 0$. Now we suppose for a contradiction that $N(\Lambda_{k-1}) < \infty$ for some $k \in \mathbb{N}$ but there exist $\Omega_j \in \mathrm{VCM}(5,g_j)$ for each $j \geq 1$ as in the statement but such that
$\limsup_{j \to \infty} \# \big( \mathrm{Sing}(\Sigma_j) \cap \mathbb{B}_4 \big) \rightarrow \infty$. Now let $\delta > 0$ be chosen sufficiently small in Lemma \ref{Lemma: Density drop} for the choice of $\Lambda = \Lambda_k$ so that $\mathrm{Sing}(\Sigma_j) \subset \mathbb{B}_1$ for sufficiently large $j \geq 1$.

\bigskip

Fix $x_j \in \mathrm{Sing}(\Sigma_j) \cap \mathbb{B}_1$ and let 
\[
r_j= \inf \{ r>0 \, | \, \theta_{|\Sigma_j|_{g_j}}(x_j,4)-\theta_{|\Sigma_j|_{g_j}}(x_j,r)\le \delta\},
\]
we then have that, since $\Sigma_j \rightarrow \mathbf{C} \in \mathscr{C}_{\Lambda_k}$ in $\mathbb{B}_5$, both $x_j \rightarrow 0 \in \mathbb{R}^8$ and $r_j \rightarrow 0$ as $j \rightarrow \infty$. Moreover, as we assume that $\limsup_{j \to \infty} \# \big( \mathrm{Sing}(\Sigma_j) \cap \mathbb{B}_4 \big) \rightarrow \infty$, for sufficiently large $j \geq 1$ we have that $r_j > 0$ as eventually $\#\mathrm{Sing}(\Sigma_j) \cap \mathbb{B}_4 \geq 2$ since if $r_j = 0$ we would violate the dichotomy of Lemma \ref{Lemma: Density drop}. Moreover since we have that for each $s \in (r_j,1]$ that
$$\theta_{|\Sigma_j|_{g_j}}(x_j,4s)-\theta_{|\Sigma_j|_{g_j}}(x_j,s) \leq \theta_{|\Sigma_j|_{g_j}}(x_j,4)-\theta_{|\Sigma_j|_{g_j}}(x_j,r_j) = \delta,$$
we can apply the claim in the proof of Lemma \ref{Lemma: Density drop} to see that in particular $\mathrm{Sing}(\Sigma_j) \subset B^{g_j}_{r_j}(x_j)$.

\bigskip

Now consider for each $j \geq 1$ the rescaled sequence $ \widetilde{\Omega}_j =\left(\eta_{x_j,r_j}\right)_\# \Omega_j$ with $r_j > 0$ as above and by writing $\tilde{g}_j = (r_j)^{-2}  (\eta_{x_j,r_j}^{-1})^*g_j$, we have that as in the proof of Lemma \ref{Lemma: Density drop} that up to a subsequence (not relabelled) we have that $\widetilde{\Omega}_j \rightarrow \widetilde{\Omega} \in \mathrm{VCM}(5,g_{\mathrm{eucl}})$, which again is in fact locally area minimising, with $\mathrm{Sing}(\widetilde{\Sigma}) \subset \mathbb{B}_1$ (since $\mathrm{Sing}(\Sigma_j) \subset B^{g_j}_{r_j}(x_j)$ for large $j \geq 1$). Also since the choice of $r_j$ ensures that $\theta_{|\Sigma_j|_{g_j}}(x_j,r_j) = \theta_{|\Sigma_j|_{g_j}}(x_j,4) - \delta$ we see that in particular by the monotonicity formula that $\theta_{|\widetilde{\Sigma}|}(0,1) \leq \theta_{\bf C}(0) - \delta = \theta_{k} - \delta$.

\bigskip

By the induction assumption, we see $\# \mathrm{Sing}(\widetilde{\Sigma}_j \cap \mathbb{B}_4)\to \infty$, and $\widetilde{\Sigma}$ has only finitely many isolated singularities. There must then exist $p \in \mathrm{Sing}(\widetilde{\Sigma})$ and $s_j \rightarrow 0 $ such that $\# [\mathrm{Sing}(\widetilde{\Sigma}_j) \cap \mathbb{B}_{s_j}(p)] \to \infty$. Note that since $\theta_{|\widetilde{\Sigma}|}(0,1) \leq \theta_{k} - \delta$, the tangent cone to $\widetilde{\Sigma}$ at $p$ has density strictly less than $\theta_k$, and thus at most $\theta_{k-1}$. By rescaling around $p$ we thus produce a sequence contradicting the assumption that the statement held for $\Lambda = \Lambda_{k-1}$; thus we have that there exists some $N(\Lambda_k) < \infty$ for which the statement holds.
\end{proof}

Using Lemma \ref{lemma: N (Lambda)}, we are able to deduce the following:

\begin{proposition}[Properties of the singular capacity]\label{prop: upper semi-continuity of singular capacity} We have that:
\begin{enumerate}

    \item For any ${\bf C}\in \mathscr{C}$, ${\bf SCap} ({\bf C})<\infty$. 
    
    \item  Whenever $(g_j, t_j,\Omega_j) \to  (g_\infty,t_\infty,\Omega_\infty)$ in $\mathcal{T}^{k,\alpha}$ and $U\subset M$ is open with $ \partial U \cap \mathrm{Sing}(\Sigma_\infty) = \emptyset$, then
        \begin{align*}
            \limsup_{j \rightarrow \infty}  {\bf SCap} (\Omega_j,U,g_j) \le  {\bf SCap} (\Omega_\infty,U,g_\infty);
        \end{align*}
    namely, the singular capacity is upper semi-continuous.

    \item For $(g,t) \in \mathcal{G}^{k,\alpha} \times \mathbb{R}$ there is some $\Omega \in \mathcal{I}(g,t)$ with $$\mathbf{SCap}(g,t) =  \mathbf{SCap}(\Omega,M,t);$$
    and in particular $ {\bf SCap} (g,t) < \infty$.
    
\end{enumerate}
\end{proposition}

\begin{proof}
    For part 1, as noted in (\ref{eqn: discrete densities stable cones}) in Subsection \ref{subsec: notation} the densities,
    $\{1 = \theta_0 < \theta_1 < \theta_2 < \ldots \nearrow + \infty\}$, of stable minimal hypercones with isolated singularities are discrete. By Definition \ref{def: singular capacity}, for any ${\bf C} \in \mathcal{C}_{\theta_1}$ we have that ${\bf SCap} ({\bf C}) = 1$. We now show part 1 by induction on the densities. Suppose for some $k \geq 1$ we have ${\bf SCap} ({\bf C})<\infty$ for every $\mathbf{C} \in \mathscr{C}_{\theta_k}$, and note that for any $\{{\bf C}_j \}_{j \geq 1} \subset \mathscr{C}_{\theta_k} $, there is a ${\bf C} \in  \mathscr{C}_{\theta_k}  $ such that $|{\bf C}_j| \to |{\bf C}|$ as varifolds by the compactness result for minimal hypercones mentioned in Subsection \ref{subsec: notation}. If for each $j \geq 1$ we have a sequence $\{(g_i^{(j)},\Omega_i^{(j)})\}_{i \geq 1}$ as in Definition \ref{def: singular capacity} part 2, then by the upper semi-continuity of density and a diagonal subsequence argument, we observe that ${\bf SCap}$ is upper semi-continuous on $\mathscr{C}_{\theta_k}$. By the compactness of $\mathscr{C}_{\theta_k}$ again, there exists some $\bar{\mathbf{C}} \in \mathscr{C}_{\theta_k}$ such that
    \[
        {\bf SCap} (\bar {\bf C}) = \sup_{{\bf C}\in \mathcal{C}_{\theta_k}} {\bf SCap}({\bf C}).
    \]
    and so by the inductive assumption we have
     that $\sup_{{\bf C}\in \mathcal{C}_{\theta_k}} {\bf SCap}({\bf C})<\infty$. 
    Now if ${\bf C} \in \mathscr{C}_{\theta_{k+1}}$, and 
    $\{(g_j,\Omega_j)\}_{j \geq 1}$ is a sequence which attaining the supremum in Definition \ref{def: singular capacity} part 2, then by Lemma \ref{lemma: N (Lambda)} we have that
\begin{align*}
{\bf SCap}({\bf C})
&\le 1
  + \Bigl(\limsup_{j\to\infty}
     \#\mathrm{Sing}(\Sigma_j)\cap \mathbb{B}_1\Bigr)
     \cdot \sup_{{\bf C'}\in \mathcal{C}_{\theta_k}} {\bf SCap}({\bf C}')\\
&\le 1 + N\!\bigl(\omega_7\,\theta_{k+1}\bigr)\cdot
           \sup_{{\bf C'}\in \mathcal{C}_{\theta_k}}  {\bf SCap}({\bf C}')<+\infty;
\end{align*}
this proves part 1.

\bigskip
    Note that by Remarks \ref{rem: Morgan-Johnson} and \ref{rem: singcap zero for smooth} part 2 follows immediately if $t_\infty$ is close to zero or volume of manifold (since then for large $j \geq 1$ the $\Sigma_j$ are regular. For part 2, by Allard's regularity theorem and the assumption that $\Omega_j \rightarrow \Omega_\infty$, we have a sequence of $r_j  \rightarrow 0$ such that $\Sigma_j$ is regular away from from $B_{r_j} (\mathrm{Sing}(\Sigma_\infty))$; we then choose $r_0 \in (0,\mathrm{inj}(\Sigma_\infty,g_\infty))$, so that in particular $\Sigma_j ,\Sigma$ are regular away from $\bigcup_{p\in \mathrm{Sing}(
    \Sigma_\infty
    )} B_{r_0}^{g_\infty}(p)$ for $j \geq 1$ sufficiently large. By the definition of singular capacity, it suffices to show that for any $p\in \mathrm{Sing}(\Sigma_\infty)$, we have
    \[
        \limsup_{j \rightarrow \infty}  {\bf SCap} (\Omega_j,B_{r_0}^{g_\infty}(p),g_j) \le  {\bf SCap} (\Omega_\infty, B_{r_0}^{g_\infty}(p),g_\infty);
    \]
    noting that, by definition, for $j \geq 1$ sufficiently large we have that both
    \[
    \begin{cases}
        {\bf SCap} (\Omega_j,B_{2r_j}^{g_\infty}(p),g_j) = {\bf SCap} (\Omega_j,B_{r_0}^{g_\infty}(p),g_j)\\
        {\bf SCap} (\Omega_\infty,B_{2r_j}^{g_\infty}(p),g_\infty) = {\bf SCap} (\Omega_\infty,B_{r_0}^{g_\infty}(p),g_\infty)
    \end{cases}.
    \]

    Rescaling $\Omega_j$ by $1/r_j$ at $p$, we have, up to a subsequence (not relabelled), that $\Sigma_j/r_j $ and $\Sigma_\infty/r_j$ converge as varifolds to ${\bf C}_p(\Sigma_\infty)$. By the definition of the singular capacity, it is therefore sufficient by the above arguments to show that if $(\Omega_j, \mathbb{B}_5,g_{j})\to ({\bf E}^+,\mathbb{B}_5, g_\mathrm{eucl}) $ with $\Omega_j \in\mathrm{VCM}(5,g_j)$ for each $j \geq 1$, ${\bf E}^+  \in\mathrm{VCM}(5,g)$ and $\partial {\bf E}^+ = {\bf C}$, then 
    $$\limsup_{j \rightarrow \infty}  {\bf SCap} (\Omega_j,\mathbb{B}_5,g_j) \le  {\bf SCap} ({\bf E}^+,\mathbb{B}_5,g_\mathrm{eucl}).$$
     
     We only need to consider the case such that there exists some $p_j \in \mathrm{Sing}(\Sigma_j)\cap
\mathbb{B}_1$ such that $\theta_{|\Sigma_j|}(p_j) = \theta_{\bf C}(0)$; since otherwise the upper semi-continuity follows by the definition. We claim that in this scenario, we have $\mathrm{Sing}(\Sigma_j)\cap
\mathbb{B}_1 = \{p_j\}$. By the monotonicity formula, since the mean curvature is bounded along the sequence, and the upper semi-continuity of density we have
\[
\limsup_{j\to\infty} (\theta_{|\Sigma_j|}(p_j,1)
- \theta_{|\Sigma_j|}(p_j)) \leq \theta_{\bf C}(0,1) - \theta_{\bf C}(0) = 0;
\]
therefore by Lemma \ref{Lemma: Density drop} we see that $\mathrm{Sing}(\Sigma_j)\cap
\mathbb{B}_1 = \{p_j\}$. Then, as noted in the proof of part 1 above, the singular capacity is upper semi-continuous on cones so we have
\[
\limsup_{j\to\infty} {\bf SCap}(\Omega_j,\mathbb{B}_1,g_j)
= \limsup_{j\to\infty} {\bf SCap}({\bf C}_{p_j}(\Sigma_j))
\le {\bf SCap} ( {\bf C}) = \mathbf{SCap}(\mathbf{E}^+,\mathbb{B}_5,g_\mathrm{eucl}),
\]
as desired.

\bigskip

    For part 3, similarly to the proof of part 1 above, by Lemma \ref{lemma: compactness of isoperimetric regions} we have that $\mathbf{SCap}(g,t)$ is attained for some $\Omega \in \mathcal{I}(g,t)$ and hence part 3 follows by combining parts 1 and 2.
\end{proof}

\subsection{Singular capacity and semi-nondegeneracy}\label{subsec: Singular capacity and semi-nondegeneracy}

We now show that the metric volume pairs for which every isoperimetric region is semi-nondegenerate, namely the set $\mathcal{U}_0^{k,\alpha}$ from Theorem \ref{thm: generic semi-nondegeneracy}, are such that only finitely many isoperimetric regions maximise the singular capacity:

\begin{lemma}\label{lem: max singular capacity triples are finite}
Given $(g,t) \in \mathcal{U}_0^{k,\alpha} \subset \mathcal{G}^{k,\alpha} \times \mathbb{R}$ we denote the collection of  $(g,t,\Omega)\in \mathcal{T}^{k,\alpha}$ which achieves the most singularities by 
        $$\mathcal{S}_\mathrm{max}(g,t) =\{ (g,t,\Omega)\in \mathcal{T}^{k,\alpha} \, | \, {\bf SCap}(\Omega,M,g) = {\bf SCap}(g,t)   \},$$
then $\mathcal{S}_{\mathrm{max}}(g,t)$ is a finite set.
\end{lemma}

\begin{proof}
Suppose not, then by Lemma \ref{lemma: compactness of isoperimetric regions} we may suppose that there exists a sequence $\{(g,t,\Omega_j)\}_{j \geq 1} \subset \mathcal{S}_{\max}(g,t)$ such that $\Omega_j \rightarrow \Omega$ in $L^1_g(M)$ with $\Omega \in \mathcal{I}(g,t)$. Then by Proposition \ref{prop: perturbation of singularities with fast growth} part 1 we see that, since $(g,t) \in \mathcal{U}_0^{k,\alpha}$, $\Omega$ is semi-nondegenerate and thus $\mathrm{Sing}(\Sigma) \neq \emptyset$ (since if $\mathrm{Sing}(\Sigma) = \emptyset$ then we would be able to produce some non-zero twisted Jacobi field on $\Sigma$ by Theorem \ref{thm: induced twisted jacobi fields}, contradicting the nondegeneracy of $\Sigma$ implied by Remark \ref{rem: slow growth remarks}) so that in particular we have 
\[
   \mathbf{SCap}(\Omega_1, M, g) 
   = \limsup_{j \to \infty} \mathbf{SCap}(\Omega_j, M, g) \leq \mathbf{SCap}(\Omega, M, g) - 1.
\]
On the other hand, by definition of ${\bf SCap}(g,t)$ we have
\[
   \mathbf{SCap}(\Omega, M, g) \leq \mathbf{SCap}(g,t) = \mathbf{SCap}(\Omega_1, M, g),
\]
a contradiction; hence $\mathcal{S}_{\max}(g,t)$ is a finite set.
\end{proof}

By combining this finiteness result with the results of section \ref{sec: twisted jacobi fields}, we are able to reduce the singular capacity by metric perturbation. To state this we first introduce the following notation and given a subset $U \subset \mathcal{G}^{k,\alpha} \times \mathbb{R}$ we define the conformal closure and interior, analogously to \cite[(26)]{LW25}, as:
\begin{align}\label{eqn: closure/interior definition}
\begin{split}
\mathrm{Clos}^{\mathrm{conf}}(U) &= \left\{(g,t) \in \mathcal{G}^{k,\alpha} \times \mathbb{R} \,\, \bigg| \begin{array}{l}
\text{ for each } \varepsilon > 0, ((1+f)g,t')  \in U \text{ for some } f \in C^{k,\alpha}(M) \\ \text{ with } \Vert f\Vert_{C^{k,\alpha}(M)} < \varepsilon \text{ and } |t - t'| < \varepsilon
\end{array}\right\},\\
\mathrm{Int}^{\mathrm{conf}}(U) &= \left\{(g,t) \in U \,\, \bigg|\begin{array}{l}
\text{ for some } \delta > 0, ((1+f)g,t')  \in U \text{ for all } f \in C^{k,\alpha}(M)\\  \text{ with } \Vert f\Vert_{C^{k,\alpha}(M)} < \delta \text{ and } |t - t'| < \delta \end{array} \right\}.
\end{split}
\end{align}

\begin{proposition}[Reducing the singular capacity]\label{prop: reducing singular capacity} 
    Given $(g,t)\in\mathcal{U}^{k,\alpha}_0 \cap \mathrm{Int}^{\mathrm{conf}}(\mathrm{Clos}^{\mathrm{conf}}(\mathcal{U}^{k,\alpha}_0))$ with ${\bf SCap} (g,t) \geq 1$, there exists a sequence, $\{(g_j,t_j)\}_{j \geq 1} \subset ([g] \times \mathbb{R}) \cap \mathcal{U}^{k,\alpha}_0\cap \mathrm{Int}^{\mathrm{conf}}(\mathrm{Clos}^{\mathrm{conf}}(\mathcal{U}^{k,\alpha}_0))$, with $g_j\to g$ in $C^{k,\alpha}$, $t_j \rightarrow t$, and such that
        $$\limsup_{j\to \infty} {\bf SCap} (g_j,t_j) \le {\bf SCap} (g,t)-1. $$
        Moreover, one can in fact choose $t_j = t$ for each $j \geq 1$.
\end{proposition}

\begin{proof}    
    Note we have by Lemma \ref{lem: max singular capacity triples are finite} we have that $\mathcal{S}_{max}(g,t) = \{(g,t,\Omega_1),\dots (g,t,\Omega_N)\}$ for some $\{\Omega_i\}_{i = 1}^N \subset \mathcal{I}(g,t)$.
    For each $(g,\Omega_i) \in \mathcal{P}^{k,\alpha}(t)$ for $i = 1,\dots, N$ there is an open and dense subset $G_i\subset C^{k,\alpha}(M)$ (namely as defined in (\ref{eqn: perturbation sets})) from Proposition \ref{prop: perturbation functions open and dense}; hence $ G= \bigcap_{i=1}^N G_i$ is open and dense in $C^{k,\alpha} (M)$ also.
    
    \bigskip
    
    Fixing $f \in G$, we claim that without loss of generality we can assume that $\int_{\Omega_i} f \, dV_g = 0$ for each $i = 1,\dots,N$. To see this we note that if $N = 1$, then by considering $h \in C^\infty_c(\Omega_1)$ with $\int_{\Omega_1} h \, dV_g = -\int_{\Omega_1} f \, dV_g$ we have that $f + h = f$ on $\Sigma_1$ (hence $f + h \in G$ by definition of the $G_i$ in Proposition \ref{prop: perturbation functions open and dense}). For $N > 1$ and non-empty $J \subset \{1,\dots, N\}$ we denote $V_J = \left(\bigcap_{j \in J} \Omega_j \right) \setminus \bigcup_{k \notin J} \Omega_k$, and fix $\varphi_J \in C^\infty_c(V_J)$ with $\int_{V_J} \varphi_J\, dV_g = 1$. By the inclusion-exclusion principle, if we denote $\Omega_K = \bigcap_{k \in K} \Omega_k$ and for $J \subset \{1,\dots,N\}$ set
    \[
    c_J = \sum_{\{K \, | \, J \subset K \subset \{1,\dots,   N\}\}} (-1)^{|K| - |J| + 1} \int_{\Omega_K} f\, dV_g,
    \]
    we ensure that $\int_{\Omega_i} f \, dV_g = -\sum_{\{J \, | \, i \in J\}} c_J$ for each $i = 1,\dots,N$. Then, setting $h = \sum_{J} c_J \varphi_J$ so that both $h \in C_c^\infty(\bigcup_{i = 1}^N \Omega_i)$ and $h$ is zero on $\Sigma_i$ (which ensures $f + h \in G_i$) with $\int_{\Omega_i} (f+h)\, dV_g = 0$ for each $i=1, \dots, N$.

    \bigskip
    
    Thus, choosing such an $f\in G$ with $\int_{\Omega_i} f \, dV_g = 0$ for each $i = 1,\dots, N$, for any sequences $f_j \to f$ in $C^{k,\alpha}$ and $t_j \rightarrow t$ we have that for sufficiently large $j \geq 1$, for $g_j= (1+\frac{f_j}{j})g$ we have $(g_j,t_j)\in \mathcal{U}^{k,\alpha}_0 \cap \mathrm{Int}^{\mathrm{conf}}(\mathrm{Clos}^{\mathrm{conf}}(\mathcal{U}^{k,\alpha}_0))$; where here we use definition (\ref{eqn: closure/interior definition}) since $(g,t) \in \mathcal{U}^{k,\alpha}_0 \cap \mathrm{Int}^{\mathrm{conf}}(\mathrm{Clos}^{\mathrm{conf}}(\mathcal{U}^{k,\alpha}_0))$.

    \bigskip
    
    Assume for a contradiction that we have a sequence, $\{\Omega_j \}_{j \geq 1}$, such that $ \Omega_j \in \mathcal{I}(g_j,t
    _j)$ for each $j \geq 1$ with ${\bf SCap} (\Omega_j,M,g_j)={\bf SCap} (g_j,t_j)$
    but such that
    ${\bf SCap} (\Omega_j,M,g_j)\ge {\bf SCap} (g,t)$ for all large $j \geq 1$. By Lemma \ref{lemma: compactness of isoperimetric regions} there exists $\Omega_\infty \in \mathcal{I}(g,t)$ such that, up to a subsequence (not relabelled),  $\Omega_j\to\Omega_\infty$ in $L_g^1(M)$, $\mathrm{Per}_g(\Omega_j) \rightarrow \mathrm{Per}_g(\Omega_\infty)$, and $|\Sigma_j|\to |\Sigma_\infty|$ as varifolds. By Proposition \ref{prop: upper semi-continuity of singular capacity} part 2 and the contradiction assumption, we have
    \begin{equation}\label{eqn: contradiction for singcap reduction}
        {\bf SCap} (\Omega_\infty,M,g) \ge \limsup_{j \rightarrow \infty}  {\bf SCap} (\Omega_j,M,g_j) \ge  {\bf SCap} (g,t). 
    \end{equation}
    Therefore, ${\bf SCap} (\Omega_\infty,M,g) = {\bf SCap} (g,t)$ and so $(g,\Omega_\infty)\in \mathcal{S}_{max}(g,t)$; thus $\Omega_\infty \in\{\Omega_1,\dots, \Omega_N\}$. We renormalise $g_j$ to $\bar{g}_j$ so that $\Omega_j \in \mathcal{I}(\bar{g}_j,t)$ for each $j \geq 1$ and such that $(\bar{g}_j,t)\in \mathcal{U}^{k,\alpha}_0 \cap \mathrm{Int}^{\mathrm{conf}}(\mathrm{Clos}^{\mathrm{conf}}(\mathcal{U}^{k,\alpha}_0))$; namely, we have $\bar{g}_j=(1+ \bar{c}_j\bar{f}_j)g$, for sequences $\bar{c}_j \rightarrow 0$ and $\bar{f}_j \to f$ in $C^{k,\alpha}$. The fact that we can choose $t_j = t$ for each $j \geq 1$ for the sequence in the statement follows by this above renormalisation; namely by absorbing the volume change into the metric factor.

    \bigskip
    
    Since $\Omega_\infty$ is semi-nondegenerate we can apply Proposition \ref{prop: perturbation of singularities with fast growth} case (ii) (noting that $\int_{\Omega_\infty} f \, dV_g = 0$ and $\nu_{\Sigma,g}(f)$ is not constant by construction) to see that for some $p \in \mathrm{Sing}(\Sigma_\infty)$ we have
    \begin{align*}
        {\bf SCap} (\Omega_\infty,M,g)-1
        &\ge {\bf SCap} (\Omega_\infty,M,g) - {\bf SCap} (\mathbf{C}_p \Sigma_\infty)\\
        &\ge \limsup_j  {\bf SCap} (\Omega_j,M,\bar{g}_j) \\
        &= \limsup_j  {\bf SCap} (\Omega_j,M,g_j) \\
        &\ge  {\bf SCap} (g,t) \\
        &={\bf SCap} (\Omega_\infty,M,g).
    \end{align*}
    For the first inequality above we use the fact that $\mathbf{SCap}(\mathbf{C}_p\Sigma_\infty) \geq 1$, the second inequality is from the application of Proposition \ref{prop: perturbation of singularities with fast growth} and Proposition \ref{prop: upper semi-continuity of singular capacity} part 2, the first equality follows since rescaling does not affect the definition of the singular capacity, and both the third inequality and the second equality follow from (\ref{eqn: contradiction for singcap reduction}). This is a contradiction, and thus $\limsup_{j \rightarrow \infty}\mathbf{SCap}(g_j,t_j) \leq \mathbf{SCap}(g,t) - 1$ as desired. 
\end{proof}

\section{Proof of Theorems \ref{thm: generic regularity metrics and volumes} \& \ref{thm: generic regularity fixed volume}}\label{sec: proof of theorems 1 & 2}

We can now establish Theorem \ref{thm: generic regularity metrics and volumes} by iteratively reducing the singular capacity in combination with the genericity of semi-nondegeneracy by proving:

\begin{theorem}[Regular metric volume pairs are generic]\label{thm: implies theorem 1}
    Let $\mathcal{U}_{\mathrm{reg}}^{k,\alpha}$ be the set of $(g,t) \in \mathcal{G}^{k,\alpha} \times \mathbb{R}$ such that every isoperimetric region with respect to the metric $g$ of enclosed volume $t$ is regular, then $\mathcal{U}_{\mathrm{reg}}^{k,\alpha}$ is open and dense in $\mathcal{U}_0^{k,\alpha}$; in particular, $\mathcal{U}_{\mathrm{reg}}^{k,\alpha}$ is generic in $\mathcal{G}^{k,\alpha} \times \mathbb{R}$ and Theorem \ref{thm: generic regularity metrics and volumes} holds. Moreover, if $\bar{g} \in \mathcal{G}^{k,\alpha}$ then $\mathcal{U}^{k,\alpha}_\mathrm{reg} \cap ([\bar{g}] \times \mathbb{R})$ is open and dense in $\mathcal{U}^{k,\alpha}_0 \cap ([\bar{g}] \times \mathbb{R})$ and generic in $[\bar{g}] \times \mathbb{R}$. 
\end{theorem}

\begin{proof} 
    It is sufficient to prove the result for finite $k \geq 4$ since, as observed in \cite[Theorem 2.10]{W17}, the case for smooth metrics then follows immediately.

    \bigskip
    
    For openness, we show that the complement $\mathcal{U}^{k,\alpha}_0\setminus \mathcal{U}^{k,\alpha}_\mathrm{reg}$ is closed. Thus, we consider a sequence, $\{(g_j,t_j)\}_{j \geq 1} \subset \mathcal{U}^{k,\alpha}_0\setminus \mathcal{U}^{k,\alpha}_\mathrm{reg}$, with 
    $g_j \to g_\infty$ in $C^{k,\alpha}(M)$, $t_j\to t_\infty$, and $(g_\infty,t_\infty) \in \mathcal{U}^{k,\alpha}_0$. By assumption, for each $j \geq 1$ there exists 
    $\Omega_j \in \mathcal{I}(g_j,t_j)$ such that $\mathrm{Sing}(\Sigma_j) \neq \emptyset$. By Lemma \ref{lemma: compactness of isoperimetric regions} 
    there exists $\Omega_\infty \in \mathcal{I}(g_\infty, t_\infty)$ so that in particular,
    up to a subsequence (not relabelled), we have $|\Sigma_j| \rightarrow |\Sigma_\infty|$ as varifolds (using Lemma \ref{lemma: compactness of isoperimetric regions} part 2). If $(g_\infty,t_\infty) \in \mathcal{U}^{k,\alpha}_\mathrm{reg}$ then $\Omega_\infty$ is regular, and so by Allard's theorem, for sufficiently large $j \geq 1$ we have that $\Omega_j$ is also regular, contradicting the assumption that $\mathrm{Sing}(\Sigma_j)\neq \emptyset$; hence $(g_\infty ,t_\infty) \in \mathcal{U}^{k,\alpha}_0 \setminus \mathcal{U}^{k,\alpha}_\mathrm{reg}$ and so $\mathcal{U}^{k,\alpha}_\mathrm{reg}$ is open in $\mathcal{U}^{k,\alpha}_0$.  

    \bigskip
    
    For denseness, we fix $(g,t) \in \mathcal{U}^{k,\alpha}_0$, $\varepsilon>0$, and note that by Proposition \ref{prop: upper semi-continuity of singular capacity} part 3 we have that $\mathbf{SCap}(g,t) < \infty$. By Theorem \ref{thm: generic semi-nondegeneracy}, in particular by the genericity in the space of metric volume pairs for a fixed conformal class, and (\ref{eqn: closure/interior definition}) we have that $([g] \times \mathbb{R}) \cap \mathrm{Clos}^{\mathrm{conf}}(\mathcal{U}^{k,\alpha}_0) = [g] \times \mathbb{R}$ and so we have
    $$([g] \times \mathbb{R}) \cap \mathcal{U}^{k,\alpha}_0 = ([g] \times \mathbb{R}) \cap \mathcal{U}^{k,\alpha}_0 \cap \mathrm{Int}^\mathrm{conf}(\mathrm{Clos}^{\mathrm{conf}}(\mathcal{U}^{k,\alpha}_0)).$$

    Hence, by repeatedly applying Proposition \ref{prop: reducing singular capacity}, we obtain $(g_\varepsilon, t_\varepsilon) \in ([g] \times \mathbb{R}) \cap  \, \mathcal{U}^{k,\alpha}_0 \cap \mathrm{Int}^\mathrm{conf}(\mathrm{Clos}^{\mathrm{conf}}(\mathcal{U}^{k,\alpha}_0))$ with $\mathbf{SCap}(g_\varepsilon,t_\varepsilon) = 0$; thus we have that $(g_\epsilon,t_\epsilon) \in \mathcal{U}^{k,\alpha}_\mathrm{reg}$ with $\|g_\varepsilon-g\|_{C^{k,\alpha}}<\varepsilon$ and $|t-t_\epsilon| <\epsilon$. This proves that $\mathcal{U}^{k,\alpha}_\mathrm{reg}$ is dense in $\mathcal{U}^{k,\alpha}_0$ as desired. Theorem \ref{thm: generic regularity metrics and volumes} then follows by Theorem \ref{thm: generic semi-nondegeneracy}; in other words it follows since semi-nondegeneracy is a generic property for metric volume pairs.

    \bigskip

    For the final statement, if $\bar{g} \in \mathcal{G}^{k,\alpha}$ then by the genericity of $\mathcal{U}^{k,\alpha}_0 \cap ([\bar{g}] \times \mathbb{R})$ in $[\bar{g}] \times \mathbb{R}$ as shown in Theorem \ref{thm: generic semi-nondegeneracy}, we see that $\mathcal{U}^{k,\alpha}_\mathrm{reg} \cap ([\bar{g}] \times \mathbb{R})$ is open and dense in $\mathcal{U}^{k,\alpha}_0 \cap ([\bar{g}] \times \mathbb{R})$ (since in particular the metric perturbations of Proposition \ref{prop: reducing singular capacity} remain in the given conformal class) and generic in $[\bar{g}] \times \mathbb{R}$.
\end{proof}

We similarly establish Theorem \ref{thm: generic regularity fixed volume} by proving:

\begin{theorem}[Regular metrics are generic for fixed volume]\label{thm: implies theorem 2} Given $t \in \mathbb{R}$, let $\mathcal{G}^{k,\alpha}_\mathrm{reg}(t)$ be the set of $g \in \mathcal{G}^{k,\alpha}$ such that every isoperimetric region with respect to $g$ of enclosed volume $t$ is regular, then $\mathcal{G}^{k,\alpha}_\mathrm{reg}(t)$ is open and dense in $\mathcal{U}_{0}^{k,\alpha} \cap (\mathcal{G}^{k,\alpha} \times \{t\})$; in particular, $\mathcal{G}^{k,\alpha}_\mathrm{reg}(t)$ is generic in $\mathcal{G}^{k,\alpha}$ and Theorem \ref{thm: generic regularity fixed volume} holds. Moreover, if $\bar{g} \in \mathcal{G}^{k,\alpha}$ then $\mathcal{G}^{k,\alpha}_\mathrm{reg}(t) \cap ([\bar{g}] \times t)$ is open and dense in $\mathcal{U}^{k,\alpha}_0 \cap ([\bar{g}] \times \{t\})$ and generic in $[\bar{g}] \times t$. 
\end{theorem}

\begin{proof}
    We observe first that if one re-defines pseudo-neighbourhoods in Definition \ref{def: pseudo-neighbourhoods} by fixing $t \in \mathbb{R}$, the proofs of Lemmas \ref{lemma: compactness of pseudo neighbourhoods}, \ref{lemma: compactness of twisted JF}, and \ref{lemma: compactness for pairs} in Subsection \ref{subsec: pseudo-neighbourhoods and three compactness lemmas} are all unchanged (since one is then just considering a constant sequence of enclosed volumes). With this in hand, the results of Subsections \ref{subsec: sard-smale} and \ref{subsec: generic semi-nondegeneracy} go through identically for a fixed $t \in \mathbb{R}$; in particular, we obtain the genericity of $\mathcal{U}_0^{k,\alpha} \cap (\mathcal{G}^{k,\alpha} \times \{t\})$ in $\mathcal{G}^{k,\alpha} \times \{t\}$ (and similarly if we restrict to a given conformal class in the metric factor). Moreover, the final statement in Proposition \ref{prop: reducing singular capacity} shows that we can fix $t$ in order to find a sequence of metric volume pairs that reduce the singular capacity. Thus, by combining all of the above, we can proceed exactly as in the proof of Theorem \ref{thm: implies theorem 1} above, now fixing $t$ in place of $\mathbb{R}$ in the volume factor, to conclude the desired results.
\end{proof}

\appendix

\counterwithin{remark}{section}
\counterwithin{definition}{section}
\counterwithin{proposition}{section}
\counterwithin{lemma}{section}
\counterwithin{theorem}{section}
\counterwithin{corollary}{section}

\section{Results for minimal cones and hypersurfaces}\label{sec: appendix A}

We collect here some results from \cite{LW25} on the asymptotic and growth rates of cones, which should be compared with Definition \ref{definition: asymptotic and growth rates}, introduce notation for the mean curvature operator of hypersurfaces, and record a Caccioppoli type inequality on smooth subsets of the boundary of an isoperimetric region. 

\subsection{Asymptotic rates for cones}

\begin{lemma}\label{lemma: J monotone for divergence form} 
    Given $\sigma >0, \Lambda >1$ there exists $K >2$, $\delta_0\in (0,1/2)$, and $H_0>0$, all depending on $\sigma$ and $\Lambda$, such that for $H$ a bounded continuous function with $|H|\le H_0,\mathbf{C}\in \mathscr{C}_{\Lambda}$ and $\gamma\in (-\Lambda,\Lambda)$ with
     $$\mathrm{dist}_{\mathbb{R}} (\gamma, \Gamma(\mathbf{C}) \cup \{-(n - 2)/2\}) \ge \sigma ,$$
    if $u\in W^{1,2} (\mathbb{A}(K^{-3}, 1)) \cap L^2(\mathbb{A}(K^{-3}, 1))$ is a non-zero weak solution of 
        \begin{align}\label{align: divergence equation}
        \mathrm{div}_\mathbf{C} (\nabla_\mathbf{C} u + \Vec{B}_0(x)) +|\mathrm{I\!I}_\mathbf{C}|^2 u + |x|^{-1} B_1(x) =H,
    \end{align}
    where $\Vec{B}_0 ,B_1$ satisfy the following estimate:
    \begin{align}\label{align: B_0,B_1}
        |\vec{B}_0 |(x) +|B_1|(x) \le \delta_0 \left(|x|^{-1} |u|(x) +|\nabla u|(x) +\Vert u\Vert_{L^2(\mathbb{A}(K^{-3}, 1))}\right), 
    \end{align}
    on $\mathbb{A}(K^{-3}, 1)$.
    Then we have 
     \[
    J_{K,\mathbf{C}}^\gamma (u; K^{-2}) -2(1+\delta_0)J_{K,\mathbf{C}}^\gamma (u; K^{-1})+J_{K,\mathbf{C}}^\gamma (u; 1)> 0. 
    \]
\end{lemma}
    \begin{proof}
    Assume the result is not true, then there exists $ \sigma > 0$, $\Lambda > 1$, a sequence of stable hypercones $\{\mathbf{C}_j\} \in \mathcal{C}_{\Lambda} $, a sequence $\{\gamma_j \}_{j \geq 1} \subset (-\Lambda, \Lambda)$ with $\mathrm{dist}_{\mathbb{R}} (\gamma_j, \Gamma(\mathbf{C}_j) \cup \{-(n - 2)/2\}) \ge \sigma $, and finally a sequence $\{u_j\}_{j \geq 1}\subset  W^{1,2}_{\mathrm{loc}} (\mathbb{A}(K^{-3}, 1)) \cap L^2(\mathbb{A}(K^{-3}, 1))$ of non-zero weak solutions of \eqref{align: divergence equation} over $\mathbf{C}_j$ satisfying \eqref{align: B_0,B_1} with $H_j=1/j$, and $\delta_j=1/j$, but such that for all $j$, we have
    \begin{align}\label{align: J_K for gamma_j}
    J_{K,\mathbf{C}_j}^{\gamma_j} (u_j; K^{-2}) -2(1+1/j)J_{K,\mathbf{C}_j}^{\gamma_j} (u_j; K^{-1})+J_{K,\mathbf{C}_j}^{\gamma_j}(u_j; 1)\le 0. 
    \end{align}
    Suppose $\gamma_j$ converges, up to subsequence, to $\gamma_\infty$, and $\mathbf{C}_j$ to some $\mathbf{C}_\infty \in \mathscr{C}_{\Lambda}$. Then, by the continuity of the spectrum of cones under varifold convergence, we have
    $ \mathrm{dist}_{\mathbb{R}} (\gamma_\infty, \Gamma(\mathbf{C}_\infty) \cup \{-(n - 2)/2\}) \ge \sigma.$ Denote $c_j =J_{K,\mathbf{C}_j}^{\gamma_j} (u_j; K^{-1})$, we then have two cases to consider, whether  the sequence $\{c_j\}_{j \geq 1}$ is bounded, or not. We start with the former. By the definition of $J_{K,\mathbf{C}}^{\gamma} (u; r)$, there exists a constant $C > 0$, depending on $K$ and $\Lambda$, such that
    \[
    C^{-1} \le \frac{J_{K,\mathbf{C}_j}^{\gamma} (u_j; r) r^{\gamma+n/2}}{\Vert u_j\Vert_{L^2(\mathbb{A}(K^{-1}r, r))}}  \le C, 
    \]
    for all $r \in (K^{-2}, 1)$. In particular, \eqref{align: J_K for gamma_j} implies a uniform $L^2$ bound for the sequence $\{u_j\}_{j \geq 1}$. On the other hand, because the $\{u_j\}_{j \geq 1}$ are weak solutions of \eqref{align: divergence equation} with $H_j\to 0$, for any open set $\Omega\subset \subset \mathbb{A}(K^{-3}, 1)$, there exists $C > 0$, depending on $K$ and $\Omega$, such that
    \begin{equation} \label{equation: pre Caccioppoli inequality}
        \int_\Omega \Vert \nabla_{\mathbf{C}_j} u_j \Vert^2 \, d\Vert \mathbf{C}_j\Vert  \le C \int_{\mathbb{A}(K^{-3}, 1)}  (1+|\mathrm{I\!I}_{\mathbf{C}_j}|^2)u_j^2 \, d\Vert \mathbf{C}_j\Vert . 
    \end{equation}
    This inequality follows by multiplying \eqref{align: divergence equation} by $u_j \eta^2$ for some $\eta \in C_{c}^{\infty}(\mathbb{A}(K^{-3}, 1))$ satisfying $\eta = 1$ on $\Omega$, integrating by parts, and appealing to \eqref{align: B_0,B_1}. Thus, by the Sobolev embedding theorem, there is $u_\infty \in W^{1,2}_{\mathrm{loc}} (\mathbb{A}(K^{-3}, 1)) \cap L^2(\mathbb{A}(K^{-3}, 1))$ such that  $u_j \rightharpoonup u_\infty $ weakly in $W^{1,2}_{\mathrm{loc}} (\mathbb{A}(K^{-3}, 1)) $ and strongly in $L^2(\mathbb{A}(K^{-3}, 1))$. Therefore, $ u_\infty$ is a weak solution of 
    \[
    \Delta_{\mathbf{C}_\infty} u_\infty +|\mathrm{I\!I}_{\mathbf{C}_\infty}|^2 u_\infty =0,
    \]
    on $\mathbb{A}(K^{-3}, 1)$. However, from \eqref{align: J_K for gamma_j} we infer
    \[
        J_{K,\mathbf{C}_\infty}^{\gamma_\infty} (u_\infty; K^{-2}) -2(1+1/j)J_{K,\mathbf{C}_\infty}^{\gamma_\infty} (u_\infty; K^{-1})+J_{K,\mathbf{C}_\infty}^{\gamma_\infty}(u_\infty; 1)\le 0,
    \]
    contradicting \cite[Lemma 6.1]{LW25}. For the case $c_j \rightarrow \infty$, denote $\hat{u}_j = u_j/c_j$. As the $\hat{u}_j$'s are weak solutions of \eqref{align: divergence equation} with $\hat{H}_j= H_j/c_j$, we can infer an $L^2$ bound for $\hat{u}_j$ from \eqref{align: J_K for gamma_j}. Arguing as in the previous case, we again get a contradiction to \cite[Lemma 6.1]{LW25}.
\end{proof}

We will frequently exploit the following dichotomy result:

\begin{lemma}\label{lem: dichotomy}
    For $\sigma,\kappa > 0$, $\Lambda \geq 1$ and $K > 2$ as in Lemma \ref{lemma: J monotone for divergence form}, then there exists some $\delta > 0$, depending on $\sigma,\kappa,\Lambda$ and $K$, for which the following holds. Consider $\mathbf{C} \in \mathcal{C}_\Lambda$ and  $\gamma \in (-\Lambda,1)$ such that
    \[
    \mathrm{dist}_{\mathbb{R}}\left(\gamma, \Gamma(\mathbf{C}) \cup \{-(n - 2)/2\}\right) \geq \sigma. 
    \]
    Then, if $\{(g_j, t_j, \Omega_j)\}$ and $\{(\bar{g}_j, \bar{t}_j, \bar{\Omega}_j)\}_j$, with $V_j = \vert \Sigma_j \vert_{g_j}$ and $\bar{V}_j = \vert \bar{\Sigma}_j \vert_{\bar{g}_j}$, are sequences in $\mathcal{T}^{k,\alpha}$ such that $|\Omega_j \Delta \bar{\Omega}_j|_{g_j} \rightarrow 0$ as $j \rightarrow \infty$, with normal coordinates on $\mathbb{B}_2$ such that
    \begin{itemize}
        \item[(i)] The metrics are conformally equivalent, with $\Vert g_j - \bar{g}_j\Vert_{C^4} \rightarrow 0 \text{ and } \sup_{j \geq 1}\Vert g_j - g_{\mathrm{eucl}}\Vert_{C^4} \leq \delta$.
        \item[(ii)] For each $j \geq 1$ we have $\mathrm{Sing}(V_j) \cap \mathbb{B}_2 = \{0\}$, and there exist $w_j \in C^4(\mathbf{C})$ with 
        \[
        |V_j|_{g_j} \mres \mathbb{B}_2 = |\mathrm{graph}_{\mathbf{C}}^{g_{\mathrm{eucl}}}(w_j)|_{g_\mathrm{eucl}} \mres \mathbb{B}_2 \quad \text{ and } \quad \Vert w_j\Vert_{C^2_*(\mathbb{B}_2)} \leq \delta.
        \]
    \end{itemize}
    Then, if we define the \textbf{graphing radius} to be
    \begin{equation}\label{eqn: graphing radius}
        \tau_j = \inf\{s \in (0,1) \, | \, \bar{V}_j \mres \mathbb{A}( s,1) = \mathrm{graph}_{V_j}^{g_j}(v_j) \text{ and } \Vert v_j\Vert_{C^2_*(\bar{V}_j \cap 
 \mathbb{A}(s,1))} \leq \delta \}
    \end{equation}
    we have, up to a subsequence, that either:
    \begin{itemize}
        \item[(1)] For each $j \geq 1$, the graphing radius, $\tau_j$, is strictly positive and $\tau_j^{-1}\bar{V}_j$ converges to a stationary integral varifold, $V_\infty$, in $\mathbb{R}^{8}$ which is not $\mathbf{C}$ but is asymptotic to it at infinity, with $\mathcal{AR}_\infty(V_\infty) < \gamma$.

        \item[(2)] For $v_j$ in the definition of the graphing radius and $t \in (K^3\tau_j,1)$ we have
        $$J^\gamma_{K;V_j,g_j}(v_j;K^{-1}t) \leq \max\{J^\gamma_{K;V_j,g_j}(v_j,t), \kappa\Vert g_j - \bar{g}_j\Vert_{C^4}\}.$$
    \end{itemize}
\end{lemma}

\begin{proof}
This follows exactly as in the proof of \cite[Lemma F.5]{LW25} by appealing to Lemma \ref{lem: dichotomy} in place of \cite[Corollary 6.2]{LW25}.
\end{proof}

\begin{lemma}\label{lemma: lower growth rate implies smooth}
  Suppose $\Sigma \subset \mathbb{R}^8$ be a stable minimal hypersurface which is asymptotic to $\mathbf{C}$ at infinity and does not equal $\mathbf{C}$. Then the asymptotic rates satisfy one of the following holds:
\begin{enumerate}[(1)]
    \item $\mathcal{AR}_{\infty}(\Sigma) \ge \gamma_2^{+}(\mathbf{C});$
    \item $\mathcal{AR}_{\infty}(\Sigma) \in \{\gamma_1^{\pm}(\mathbf{C})\}$ and $\mathrm{Sing}(\Sigma) = \emptyset.$
\end{enumerate}
\end{lemma}

\begin{proof}
This is precisely the statement of \cite[Lemma F.3]{LW25}.
\end{proof}

\subsection{The mean curvature operator on graphs}\label{subsec: mean curvature operator}

We will adopt the notation from \cite[Appendix B]{LW25}, and suppose that we have Riemannian metrics $g,g^\pm \in \mathcal{G}^{k,\alpha}$ on $M$, and $f^\pm \in C^2(M)$ with
\begin{equation} \label{equation: condition fpm}
   [ f^\pm ]_{x,g,C^2_*} \le \delta \qquad \text{and} \qquad 
   g^\pm = (1+f^\pm) g,
\end{equation}
for some dimensional constant $\delta > 0$. Denote by $\mathcal{M}^f : C^2_*(\Sigma) \to C_{\mathrm{loc}}(\Sigma) $ the mean curvature operator on $\Sigma$ under the metric $(1+f)g$, i.e.~the operator defined by 
\begin{equation*}
    \int_\Sigma \mathcal{M}^f(u) \cdot \varphi \, dA_g = \left. \frac{d}{dt} \right\vert_{t = 0} \int_\Sigma F^f(x, u + t\varphi, d(u + t \varphi)) \, dA_g, 
\end{equation*}
for $\varphi \in C_{c}^{1}(\Sigma)$, and where $F^f$ is the $C^1$ area density function $F^f = F^f(x, z, \xi)$ giving, for any $\varphi \in C_c(M \setminus \mathrm{Sing}(\Sigma))$, the identity
\begin{equation*}
    \int_M \varphi(x) \, d\Vert \mathrm{graph}_\Sigma(u) \Vert_{(1 + f)g}(x) = \int_M \varphi \circ \Phi^u(x) \cdot F^f(x, u(x), du(x)) \, d\Vert \Sigma \Vert_g, 
\end{equation*}
where $\Phi^u(x) = \exp_x^g(u(x) \cdot \nu(x))$, 
provided $u \in C^2(\Sigma)$ satisfies $\Vert u \Vert_{C_{*}^{2}(\Sigma)} \leq \delta$, and $f \in C^2(M)$ is such that $[f]_{x,g, C_*^2} \leq \delta$ for every $x \in \Sigma$; see \cite[Theorem B.1 (ii)]{LW25} for a proof of existence of such $F$. In particular, for every pair, $f^\pm$, as in \eqref{equation: condition fpm}, and every pair, $u^\pm$, with $\Vert u^\pm \Vert_{C_{*}^{2}(\Sigma)} \leq \delta$, the difference of the mean curvature operators takes the form 
\begin{equation} \label{equation: M(u+) - M(u-)}
    \mathcal{M}^{f^+}(u^+) - \mathcal{M}^{f^-}(u^-) = - L_{\Sigma, g}(u^+ - u^-) + \frac{n}{2}\nu(f^+ - f^-) + \mathrm{div}_{\Sigma, g}(\mathcal{E}_1) + r_S^{-1} \mathcal{E}_2, 
\end{equation}
where $\mathcal{E}_1$ and $\mathcal{E}_2$ are functions on $\Sigma$ satisfying, for each $x \in \Sigma$, the following pointwise bound:
\begin{align}\label{eqn: bounds on errors}
\begin{split}
    |\mathcal{E}_1(x)| + |\mathcal{E}_2(x)| \leq C \left( \sum_{i = \pm} [f^i]_{x,g, C_*^2}  + \frac{1}{r_S(x)}\vert (u + v)(x) \vert + |d(u + v)(x)| \right) \\
    \cdot  \left([f^+-f^-]_{x, g,C_{*}^2(M)}+|(u - v)(x)|+|d(u - v)(x)|\right). 
\end{split}
\end{align}
where $C > 0$ is a constant depending on $g$.

\begin{proposition}[Caccioppoli inequality]\label{prop: caccioppoli inequality}
      Suppose that $u,v\in C_{\mathrm{loc}}^2(\Sigma)$ solve $\mathcal{M}^{f+} u -\mathcal{M}^{f^-}v= h$, where $h$ is a constant. Then, for any open sets $\Omega' \subset\subset \Omega \subset\subset \Sigma$ there exists a constant $C >0$, depending on $\Omega',\Omega, g, f^\pm,$ and $\Vert u-v\Vert_{C^1(\Omega)}$, such that 
 \begin{equation*}
        \int_{\Omega'} |\nabla (u-v)|^2 \,dA_g \leq C \left( \int_{\Omega} (u-v)^2 \,dA_g+ \Vert f^+-f^-\Vert_{C^2(M)}\cdot \Vert u-v\Vert_{L^1(\Omega)}\right).
    \end{equation*}   
\end{proposition}

\begin{proof}
As $h$ is a constant, for any volume preserving test function $\phi \in C_{c}^1(\Omega) \cap L_{T}^{1}(\Omega)$, we have 
\begin{equation*}
    \int_\Omega (\mathcal{M}^{f^+}u - \mathcal{M}^{f^-}v) \phi \, dA_g= 0. 
\end{equation*}
In particular, appealing to \eqref{equation: M(u+) - M(u-)} and integrating by parts, the above becomes
\begin{equation*}
    \int_\Omega \nabla (u - v) \cdot \nabla \phi - \left( \left( \vert \sff_\Sigma \vert^2 + \mathrm{Ric}(\nu) \right)(u - v) - \frac{n}{2}\nu(f^+ - f^-) + \mathrm{div}_{\Sigma, g}(\mathcal{E}_1) + r_S^{-1}\mathcal{E}_2 \right) \phi \, dA_g = 0, 
\end{equation*}
which can then be rewritten as 
\begin{align} \label{equation: Caccioppoli identity}
    \int_\Omega \nabla (u - v) \cdot \nabla \phi \, dA_g= \int_\Omega B(x) \cdot \phi + \mathcal{E}_1 \cdot \nabla \phi \, dA_g
\end{align}
with $B(x)=(|\sff_\Sigma|^2+\mathrm{Ric}(\nu))(u - v) -\frac{n}{2}\nu(f^+ - f^-)- r_S^{-1}\mathcal{E}_2$. 
Consider then a function $\phi\in C^1_c (\Omega)$ with $\phi = 1$ on $\Omega^\prime$, and fix another non-negative function $\eta \in C^1_c (\Omega')$ with $\int_{\Omega^\prime} \eta \, dA_g= 1$. Denote $C_0= \int_\Omega \phi^2 (u - v) \, dA_g$, so that $\int_{\Omega} (\phi^2(u - v) - C_0\eta) \, dA_g = 0$; applying \eqref{equation: Caccioppoli identity} with this function and using the bounds in (\ref{eqn: bounds on errors}) one can derive the desired inequality (noting that $r_S$ is uniformly bounded away from zero on $\Omega$ since $\Omega \subset \subset \Sigma$). 
\end{proof}

\section{The space of almost minimisers in dimension eight}\label{sec: appendix B}

In \cite{E24} the definition and existence of a cone decomposition for seven dimensional minimal hypersurfaces with bounded mass and index was established, and prove that such hypersurfaces belong to a finite collection diffeomorphism types. In this appendix we adapt this notion, and the corresponding existence result, to the setting of almost minimisers in dimension eight. As a consequence, one can deduce that almost minimisers (with prescribed $\Lambda \geq 0$ and bounded mass) also belong to a finite collection of diffeomorphism types; this result was recorded in \cite[Theorem 5.5]{ESV24}. We also record some definitions from \cite[Section 9.1]{LW25}, adapted to the setting of isoperimetric regions, that will be of use in Subsection \ref{subsec: generic semi-nondegeneracy} when decomposing the space of triples.

\subsection{Cone decomposition}

We start by recalling several notions from \cite{E24}, adapted to the setting of almost minimisers of perimeter, that will be used in defining the cone decomposition. We note that some of the definitions below could be simplified (for instance, by avoiding the framework of varifolds), but we have chosen to present them in this form for consistency with the exposition in \cite{E24}.

\begin{definition}[Strong cone region] \label{def: strong cone region}

    Let $g$ be a $C^2$ metric on $\mathbb{B}_R(a) \subset \mathbb{R}^8$, and $E$ a set of finite perimeter on $(\mathbb{B}_R(a), g)$. Given $\mathbf{C} \in \mathscr{C}$, $\beta, \tau, \sigma \in [0, 1/4]$, $\rho \in [0, R]$, we say that $\vert \partial E \vert_g \mres {(\mathbb{A}(a; \rho/8, R), g)}$ is a $(\mathbf{C}, 1, \beta)$-\textbf{strong cone region} provided there is some $C^1$ function $u : (a + \mathbf{C}) \cap \mathbb{A}(a; \rho/8, R) \rightarrow \mathbf{C}^\perp$ such that, for any $r \in [\rho, R] \cap (0, \infty)$ we have 
    \begin{enumerate}
        \item Small $C^1_*$-norms: $r^{-1}\vert u \vert + \vert \nabla u \vert \leq \beta$. 
        \item Almost constant density ratios: $ \theta_{|\mathbf{C}|}(0) - \beta \leq \theta_{\vert \partial E \vert}(a, r) \leq \theta_{|\mathbf{C}|}(0) + \beta$. 
        \item Graphicality: $\vert \partial E \vert \mres { \mathbb{A}(a; \rho/8, R)} = \vert \mathrm{graph}_{a + \mathbf{C}}(u) \cap \mathbb{A}(a; \rho/8, R) \vert$. 
    \end{enumerate}
\end{definition}

\begin{remark}
    Since for almost minimisers of perimeter one can in general only expect $C^{1,\alpha}$ (for $\alpha \in (0,\frac{1}{2})$) regular part for their boundary, see \cite[Theorem 21.8]{M12}, we weaken the required norm control in the definition of strong cone regions compared to \cite[Definition 6.0.3]{E24}. See also Remark \ref{rem: C^2 for isoperimetric regions} for isoperimetric regions, for which one can upgrade this regularity to $C^2$. 
\end{remark}

\begin{remark} \label{remark: weak, standard, strong cone regions}
    In \cite{E24}, further refinements of the above definition are introduced. More precisely, one starts with the notion of weak cone region in which the cones $\mathbf{C}$ and the centres are allowed to change from scale to scale. One then proceeds with cone regions, in which the centre is fixed for every scale, and finally concludes with the strong cone regions defined above. It is then a consequence of \cite[Lemma 6.2 \& Theorem 6.3]{E24} that all these notions are effectively equivalent.  
\end{remark}

\begin{definition}[Smooth model]
    Given $\Lambda, \gamma \geq 0$, and $\sigma \in (0, 1/3)$, a tuple $(S, \mathbf{C}, \{(\mathbf{C}_\alpha, \mathbb{B}(y_\alpha, r_\alpha))\}_\alpha)$ is called a $(\Lambda, \sigma, \gamma)$-\textbf{smooth model} if $S$ is a $7$-dimensional local perimeter minimiser in $(\mathbb{R}^8, g_{\mathrm{eucl}})$ with $\theta_S(0, \infty) \leq \Lambda$, and $\mathbf{C}, \{\mathbf{C}_\alpha\}_\alpha \subset \mathscr{C}_\Lambda$, and  $\{\mathbb{B}_{r_\alpha}(y_\alpha) \}_\alpha$ is a finite collection of disjoint balls in $\mathbb{B}_{1 - 3\sigma}$, provided that the following is satisfied
    \begin{enumerate}
        \item $S$ can be represented by a union of disjoint closed, smooth, smooth, embedded, minimal hypersurfaces in $\mathbb{R}^8 \setminus \{y_\alpha\}_\alpha$, i.e.~$\vert S \vert = \sum_{j = 1}^k \vert S_j \vert$.
        \item $\vert S \vert \mres {(\mathbb{A}(1, \infty), g)}$ is a $(\mathbf{C}, 1, \gamma)$-strong cone region.
        \item For each $\alpha,$ there is a $j = 1,\dots, k$ so that $\mathrm{spt}\Vert S \Vert \cap \mathbb{A}(y_\alpha; 0, 2 r_\alpha) = S_j \cap \mathbb{A}(y_\alpha; 0, 2 r_\alpha)$ and it is a $(\mathbf{C}, 1, \gamma)$-strong cone region. 
    \end{enumerate}
\end{definition}

\begin{remark}
    Note that we do not need to change the variational hypothesis (i.e.~the assumption of zero mean curvature) in the definition of smooth models, for instance by requiring them to to have constant mean curvature or almost minimisers, as these arise from blow-up arguments. In particular, we only adapted the definition to our setting by requiring the stronger condition of being a perimeter minimiser, instead of a stationary integral varifold. 
\end{remark}

\begin{definition}[Smooth model scale constant]
    Given a $(\Lambda, \sigma, \gamma)$-smooth model $S$, we let $\epsilon_S$ be the largest number smaller than $\min(1, \min_\alpha \{r_\alpha\})$ for which the graph map 
    \begin{equation*}
        \mathrm{graph}_S : T^{\perp}\Big(\bigcup_j S_j \Big) \rightarrow \mathbb{R}^8, \quad \mathrm{graph}_S(x, v) = x + v, 
    \end{equation*}
    is a diffeomorphism from 
    \begin{align*}
        \left\{ (x, v) \in T^{\perp}\Big(\bigcup_j S_j \Big) \, \bigg| \, x \in \mathbb{B}_2 \setminus \bigcup_\alpha \mathbb{B}_{ \frac{r_\alpha}{8}}(y_\alpha), \vert v \vert < 2 \epsilon_S \right\} 
    \end{align*}
    onto its image, and satisfies $\vert \Vert D \mathrm{graph}_S \vert_{(x, v)}  - \mathrm{Id}  \Vert  \leq \vert \epsilon_S \vert^{-1} \vert v \vert$. 
\end{definition}

\begin{definition}[Smooth region]
    Given a smooth model $S$, a $C^2$ metric $g$ on $\mathbb{B}_R(a) \subset \mathbb{R}^8$, and $\beta \in (0, 1)$, we say that (the varifold associated with) the boundary of a set of finite perimeter $\vert \partial E \vert_g \mres {\mathbb{B}_R(a)}$ is a $(S, \beta)$-\textbf{smooth region} if for each $i = 1, 2, \ldots, k$, there is a $C^2$ function $u_{i} : S_i \rightarrow S_{i}^{\perp}$ so that 
    \begin{align*}
        ((\eta_{a, R})_{\#} \vert \partial E \vert_g) \mres {\mathbb{B}_1 \setminus \bigcup_\alpha \mathbb{B}_{\frac{r_\alpha}{4}}(y_\alpha)} = \sum_{i = 1}^{k} \left|\mathrm{graph}_{S_i}(u_{i}) \cap \mathbb{B}_1 \setminus \bigcup_\alpha \mathbb{B}_{\frac{r_\alpha}{4}}(y_\alpha)\right|_{R^{-2}(\eta_{a, R}^{-1})^\ast g}
    \end{align*}
    and $\Vert u_{i} \Vert_{C^2(S_i)} \leq \beta \epsilon_S$ for all $i$, where $\epsilon_S$ is the scale constant in the previous definition.  
\end{definition}

\begin{definition}[Cone decomposition] \label{def: cone decomposition}
Given $\theta, \gamma, \beta \in \mathbb{R}$, $\sigma \in (0, 1/3)$, $R > 0$, and $N \in \mathbb{N}$, we let $g$ be a $C^{k, \alpha}$ metric on $\mathbb{B}_R(x) \subset \mathbb{R}^8$, $E \in \mathcal{C}(\mathbb{B}_R(x))$, and $\mathcal{S} = \{S_s\}_s$ be a finite collection of $(\theta, \sigma, \gamma)$-smooth models.  
A $(\theta, \beta, \mathcal{S}, N)$-\textbf{cone decomposition} of $\vert \partial E \vert_g \mres \mathbb{B}_R(x)$ consists of the following parameters: 
\begin{itemize}
    \item Integers $N_C, N_S$ satisfying $N_C + N_S \leq N$, where $N_C$ is the number of strong-cone regions, while $N_S$ is the number of smooth regions. 
    \item Points $\{x_a\}_a, \{x_b\}_b \subset \mathbb{B}_R(x)$, where the $\{x_a\}_a$ are the centres of the strong-cone regions, and $\{x_b\}_b$ are the centres of the smooth regions.
    \item Radii $\{R_a, \rho_a \, | \, R_a \geq 2 \rho_a\}_a$, respectively $\{R_b\}_b$ corresponding to radii of annuli in the definition of strong-cone region, respectively of balls in the definition of smooth regions.
    \item Cones $\{\mathbf{C}_a\}_a \subset \mathscr{C}$.
    \item Indices $\{s_b\}_b$ corresponding to the smooth models $S_{s_b}$.
\end{itemize}
Where in the above $a = 1, \ldots, N_C$ and $b = 1, \ldots, N_S$. Furthermore, these parameters determine a covering of balls and annuli satisfying 
\begin{enumerate}
    \item Every $\vert \partial E \vert_g \mres \mathbb{A}(x_a;\rho_a, R_a)$ is a $(\mathbf{C}_a, 1, \beta)$-strong cone region and every $\vert \partial E \vert_g \mres \mathbb{B}_{R_b}(x_b)$ is a $(S_{s_b}, \beta)$-smooth region. 
    \item In the previous point, there is either a strong-cone region $\mathbb{A}(x_a; \rho_a, R_a)$ for $\vert \partial E \vert_g$ with $R_a = R$ and $x_a = x$, or a smooth region $B_{R_b}(x_b)$ for $\vert \partial E \vert_g$ with $R_b = R$ and $x_b = x$.  
    \item If $\vert \partial E \vert_g \mres \mathbb{A}(x_a; \rho_a, R_a)$ is a $(\mathbf{C}_a, 1, \beta)$-strong cone region and $\rho_a > 0$, then there exists either a smooth region $\mathbb{B}_{R_b}(x_b)$ for $\vert \partial E \vert_g$ with $R_b = \rho_a$, or another cone region $\mathbb{A}(x_{a'}; \rho_{a'}, R_{a'})$ for $\vert \partial E \vert_g$ with $R_{a'} = \rho_a, x_{a'} = x_a$. If $\rho_a = 0$, then $\theta_{\mathbf{C}_a}(0) > 1$.  
    \item If $\vert \partial E \vert \mres (\mathbb{B}(x_b, R_b), g)$ is a smooth region with $(S, \mathbf{C}, \{\mathbf{C}_{\alpha}, B(y_{\alpha}, r_{\alpha})\}_{\alpha}) \in \mathcal{S}$, then for any index $\alpha$, there exists a point $x_{b, \alpha}$, and a radius $R_{b, \alpha}$ satisfying 
    \begin{equation*}
        \vert x_{b, \alpha} - y_{\alpha} \vert \leq \beta R_b r_{\alpha} \qquad \text{and} \qquad \frac{1}{2} \leq \frac{R_{b, \alpha}}{R_b r_{\alpha}} \leq 1 + \beta, 
    \end{equation*}
    and either a strong-cone region $\mathbb{A}(x_{a'}, \rho_{a'}, R_{a'})$ for $\vert \partial E \vert$ with $R_a = R_{b, \alpha}$ and $x_a = x_{b, \alpha}$, or another smooth region $\mathbb{B}(x_{b'}, R_{b'})$ with $R_{b'} = R_{b, \alpha}$, and $x_{b'} = x_{b, \alpha}$. 
\end{enumerate}
\end{definition}

In light of \eqref{eqn: discrete densities stable cones}, we have an enumeration of $\{\theta_{|\mathbf{C}_i|}(0) \, | \, \mathbf{C} \in \mathcal{C}\} = \{\theta_l\}_{l \in \mathbb{N}}$; compared to \cite{E24}, we do not have to enumerate the set of densities-with-multiplicity as our varifolds are of multiplicity one. The strategy introduced in \cite{E24} to prove the existence of a cone decomposition is general and relies on good compactness and partial regularity theorems (sheeting), a monotonicity formula, and a \L ojasiewicz-Simon (or epiperimetric type) inequality for singular models, all of which are available for almost minimisers, see \cite{M12} and \cite{ESV24}.

\begin{theorem}[Existence of cone decomposition] \label{thm: cone decomposition}
    Given $l \in \mathbb{N}$, and $0 < \beta \leq \gamma \leq 1$, and $\sigma \in (0, 1/200]$, there are constants $\delta_{l, \gamma, \beta, \sigma} > 0$ and $N \in \mathbb{N}$, as well as a finite collection of $(\theta_l, \sigma, \beta)$-smooth models $\mathcal
    S = \{S_s\}_s$, all depending on $(l, \gamma, \beta$, and $\sigma$, so that the following holds.
    Let $g$ be a $C^3$ metric on $\mathbb{B}_1$ satisfying $\Vert g - g_{\mathrm{eucl}} \Vert_{C^3} \leq \delta_{l, \gamma, \beta, \sigma}$. Consider $\Omega$ a $(\delta_{l, \gamma, \beta, \sigma}, 1)$-almost minimiser of perimeter in $\mathbb{B}_1(0)$, and $\mathbf{C} \in \mathscr{C}$ such that $\theta_{|\mathbf{C}_i|}(0) \leq \theta_l$. Denote $\Sigma \subset \partial \Omega$ the $C^{1, \alpha}$ part of the boundary with $\overline{\Sigma} = \partial \Omega$, and suppose that 
    \begin{equation*}
        d_\mathcal{H}(\mathrm{spt} \vert \Sigma \vert \cap \mathbb{B}_1, \mathbf{C} \cap \mathbb{B}_1) \leq \delta_{l, \gamma, \beta, \sigma}, 
    \end{equation*}
    and
    \begin{equation*}
        \frac{1}{2} \theta_{|\mathbf{C}_i|}(0) \leq \theta_{\vert \Sigma \vert}(0, 1/2), \qquad \text{and} \qquad \theta_{\vert \Sigma \vert}(0, 1) \leq \frac{3}{2}\theta_{|\mathbf{C}_i|}(0).
    \end{equation*}
    Then there exists $r \in (1 - 40\sigma, 1)$, so that $\vert \Sigma \vert \mres (\mathbb{B}_r, g)$ admits a $(\theta_l, \beta, \mathcal{S}, N)$-cone decomposition.
\end{theorem}

\begin{remark}\label{rem: C^2 for isoperimetric regions}
    In Subsection \ref{subsec: generic semi-nondegeneracy} we will apply Theorem \ref{thm: cone decomposition} to isoperimetric regions which, since our Riemannian metrics are at least $C^{4, \alpha}$ regular, have at least $C^2$ regular part of their boundary. As a consequence, the cone decomposition for the boundary of an isoperimetric region has improved regularity, analogous to the case of minimal hypersurfaces with bounded mass and index as considered in \cite{E24}. In particular one can assume the various notions of cone regions have $C^2_*$ control, as opposed to just $C^1_*$ control, exactly as in \cite[Section 6]{E24}. Moreover, if one is only considering isoperimetric regions in the proof below, one can bypass the general regularity and compactness theory for almost minimisers and instead invoke Allard’s theorem directly (since its hypotheses are satisfied in the situations under consideration).
\end{remark}
\begin{proof}
    We will only sketch the proof here, pointing out the relevant alterations required to adapt to the setting of almost minimisers of perimeter, and refer to \cite[Theorem 7.1]{E24} for precise details. The proof will proceed by induction on $l \in \mathbb{N}$ and contradiction. Constants will be chosen throughout, but the following hierarchy should be kept in mind:
    \begin{equation*}
        \beta^{\prime \prime} \ll \tau \ll \beta^\prime \ll \beta \leq \gamma \ll \sigma < 1, 
    \end{equation*}
    where the parameters $\beta''$ and $\beta'$ will correspond to notions in \cite{E24} we will give reference to. 

    \bigskip

    Suppose now for a contradiction that the theorem fails. Then, there are sequences $\delta_i \rightarrow 0$, $C^3$ metrics $g_i$, $(\delta_i, 1)$-almost minimisers of perimeter $\Omega_i$, and cones, $\mathbf{C}_i$, such that $g_i, \Omega_i$ satisfy the hypothesis of the theorem with $\delta_i, g_i$, and $\mathbf{C}_i$ in place of $\delta, g$, and $\mathbf{C}$, but with the property that for any finite collection of $\mathcal{S}^\prime$ of $(\theta_l, \sigma, \beta)$-smooth models, and any $N^\prime \in \mathbb{N}$, there is some $i_0 \geq 1$ such that $\vert \Sigma_i \vert_{g_i} \mres (\mathbb{B}_r, g_i)$ does not admit a $(\theta_l, \beta, \mathcal{S}^\prime, N^\prime)$-strong cone decomposition for all $i > i_0$, and $r \in (1 - 40 \sigma, 1)$. Passing to a subsequence (not relabelled), by the compactness of stable minimal regular cones (see the discussion preceding \eqref{eqn: discrete densities stable cones}), we ensure that $\mathbf{C}_i \rightarrow \mathbf{C} \in \mathscr{C}$ smoothly, with multiplicity one away from the origin, and $\theta_{|\mathbf{C}_i|}(0) = \theta_{|\mathbf{C}|}(0)$ for all $i \geq 1$ sufficiently large. Appealing to the compactness of $(\Lambda, r_0)$-almost minimisers, \cite[Section 21.5]{M12} (or alternatively to \cite[Lemma 5.6]{ESV24}), by passing to a further subsequence (not relabelled), we ensure that $\vert \Sigma_i \vert_{g_i} \rightarrow \vert \mathbf{C} \vert_{g_{\mathrm{eucl}}}$ as varifolds in $\mathbb{B}_1$ and in $C^{1, \alpha}$ on compact subsets of the complement of $\mathrm{Sing}(\mathbf{C}) \subset \{0\}$, with $\theta_{|\mathbf{C}_i|}(0) \leq \theta_l$. 

    \bigskip
    
    Note that if $l = 0$, corresponding to $\vert \mathbf{C} \vert$ being a multiplicity one hyperplane then, provided $i \geq 1$ is large enough, we can appeal to the regularity theory of $(\Lambda, r_0)$-almost minimisers \cite[Theorem 21.8]{M12} to infer that each $\vert \Sigma_i \vert_{g_i} \mres (\mathbb{B}_{1 - 5\sigma}, g_i)$ is a $(S, \beta)$-smooth region, with $S$ the smooth model $(\vert \mathbf{C} \vert, \mathbf{C}, \emptyset)$. 

    \bigskip
    
    By the inductive hypothesis, we can assume the theorem holds for $l^\prime < l$. For each $i \geq 1$, we define 
    \[
    \rho_i = \inf \{ \rho \, | \, \vert \Sigma_i \vert_{g_i} \mres (\mathbb{A}(\rho, 1), g_i)  \text{ is a $(\mathbf{C}, 1, \beta^{\prime \prime}, \tau, \sigma)$-weak cone region}\};
    \]
    see \cite[Definition 6.0.1]{E24} for the definition of weak cone region, where $\beta^{\prime \prime}$ and $\tau$ will be be chosen later, depending on $l, \beta$, and $\gamma).$ In addition, let $a_i = a_{\rho_i}(V_i)$ be the annulus centre at radius $\rho_i$ as appearing in \cite[Definition 6.0.1]{E24}. Because of convergence $\vert \Sigma_i \vert_{g_i}  \rightarrow \vert \mathbf{C} \vert$ in both the varifold sense and in $C^{1, \alpha}$ on compact subsets of $B_1 \setminus \{0\}$, we necessarily have $a_i \rightarrow 0$, as well as $\rho_i \rightarrow 0$. Assuming $\beta^{\prime \prime}$ and $\tau$ small enough, we have that $\vert \Sigma_i \vert_{g_i} \mres (\mathbb{A}(a_i; \rho_i, 1 - 3 \sigma), g_i)$ is a $(\mathbf{C}, 1, \beta^\prime)$-cone region as defined in \cite[Definition 6.0.2]{E24}. Using the propagation of graphicality for almost minimisers appearing in \cite[Theorem 5.2]{ESV24}, which can be applied due to the hypothesis $\delta_i \rightarrow 0$, we see that $\vert \Sigma_i \vert_{g_i} \mres (\mathbb{A}(a_i; \rho_i, 1 - 3 \sigma), g_i)$ is a $(\mathbf{C}, 1, \beta^\prime)$-strong cone region as in Definition \ref{def: strong cone region} for $i \geq 1$ sufficiently large.  

    \bigskip
    
    Suppose now $\rho_i = 0$ for infinitely many $i \geq 1$ then if $\theta_{|\mathbf{C}|}(0) = 1$, again by using the regularity theory of $(\Lambda, r_0)$-almost minimisers \cite[Theorem 21.8]{M12}, we conclude that $\vert \Sigma_i \vert \mres (\mathbb{B}_{1 - 5 \sigma}, g_i)$ is a $(S, \beta)$-smooth region for $S$ the smooth model given by $(\vert \mathbf{C} \vert, \mathbf{C}, \emptyset)$. But this means that infinitely many $\vert \Sigma_i \vert_{g_i} \mres (\mathbb{B}_{1 - 5 \sigma}, g_i)$ admit a $(\theta_l, \beta, \{S\}, 1)$-cone decomposition, giving a contradiction. Analogously, if $\theta_{|\mathbf{C}_i|}(0) > 1$, we have that $\vert \Sigma_i \vert_{g_i} \mres (\mathbb{B}_r, g_i)$ admit a strong cone decomposition for sufficiently large $i \geq 1$, giving another contradiction. Thus, $\rho_i > 0$ for sufficiently large $i \geq 1$, and we can consider the rescaled varifolds $V_i^\prime = (\eta_{a_i, \rho_i})_\# \vert \Sigma_i \vert_{g_i}$ under the corresponding rescaled metrics $g_i^\prime = \rho_i^{-2} (\eta_{a_i, \rho_i}^{-1})^* g_i$. In particular we see that
    \[
        \inf\{\rho \, | \,  V_i^\prime \mres (\mathbb{A}(\rho, R), g_i^\prime) \text{ is a $(\mathbf{C}, 1, \beta^{\prime \prime}, \tau, \sigma)$-weak cone region}\} = 1, 
    \]
    and $a_1(V_i^\prime) = 0$, where again the notation is as in \cite[Definition 6.0.1]{E24}. Appealing once more to the compactness theory of $(\Lambda, r_0)$-almost minimisers, \cite[Section 21.5]{M12} or \cite[Lemma 5.6]{ESV24}, we can now extract a subsequence (not relabelled) converging to a limiting minimiser $V^\prime$ in both the varifold topology as well as in $C^{1, \alpha}$ on compact subsets of $\mathbb{R}^8 \setminus \mathrm{Sing}(V^\prime)$; for the monotonicity argument appearing in the paragraph below \cite[Theorem 7.1, (64)]{E24}, we instead apply the monotonicity formula for almost minimisers \cite[(5.2)]{ESV24} in place of of \cite[(18)]{E24}. Provided $\beta$ is sufficiently small, using Arzelà--Ascoli, and \cite[Theorem 5.1]{E24}, we know that $V^\prime \mres (\mathbb{A}(1, \infty), g_{\mathrm{eucl}})$ is a $(\mathbf{C}, 1, \beta)$-strong cone region,  with $\vert \Sigma_i^\prime \vert \rightarrow V^\prime$ in $C^{1, \alpha}$ on compact subsets of $\mathbb{R}^8 \setminus \overline{\mathbb{B}}_{1/8}$. In particular, we have that $\mathrm{Sing}(V^\prime) \subset \overline{\mathbb{B}}_{1/8}$ is a finite set, and one can check that any tangent cone to $V^\prime$ at infinity is of the form $\vert \mathbf{C}^\prime \vert$ for some $\mathbf{C}^\prime \in \mathscr{C}$. We now have two cases to analyse:
    \begin{enumerate}[(a)]
        \item $\theta_{V^\prime}(a) \geq \theta_l$ for some $a \in \mathrm{spt}(V^\prime)$. By the monotonicity formula, we infer that $V^\prime = \vert a + \mathbf{C}^\prime \vert$ for some $\mathbf{C}^\prime \in \mathscr{C}$. The rest of the argument goes unchanged with respect to \cite[Theorem 7.1, Subcase 1A]{E24}, with the simplifications arising from the fact that we have no points of index concentration, and that we are in a multiplicity one setting. 
        \item $\theta_{V^\prime}(a) \leq \theta_l$ for all $a \in \mathrm{spt}(V^\prime)$. This case follows verbatim as in \cite[Theorem 7.1, Subcase 1B]{E24}, replacing the application of \cite[Theorem 6.3]{E24} there with \cite[Theorem 5.2]{ESV24}. 
    \end{enumerate}
    In either case, we reach an contradiction under the assumption that the theorem fails.
\end{proof}

\subsection{Tree representations and large-scale cone decompositions}\label{subsec: tree representations}

Here we record several definitions from \cite[Section 9.1]{LW25} adapted to the setting of isoperimetric regions for use in Subsection \ref{subsec: generic semi-nondegeneracy}:

\begin{definition}[Tree representation of a cone decomposition] \label{definition: tree representation of cone dec}
    Given a $(\theta, \beta, \mathcal{S}, N)$-cone decomposition of $\vert \partial E \vert_g \mres (\mathbb{B}_R(x), g)$ with parameters as labelled as in Definition \ref{def: cone decomposition}, the corresponding \textbf{tree representation} is a rooted tree uniquely defined by: 
    \begin{itemize}
        \item There are two types of nodes: every of \textbf{type I} is labelled $(\mathbf{C}_a, 1, x_a, R_a, \rho_a)$, while every node of \textbf{type II} is labelled $(S_{s_b}, x_b, R_b)$. 
        \item The root is labelled with either $(\mathbf{C}_a, 1, x_a = x, R_a = R, \rho_a)$, or $(S_{s_b}, x_b = x, R_b = R)$. 
        \item For any type I node $(\mathbf{C}_a, 1, x_a, R_a, \rho_a)$, either $\rho_a = 0$, $\theta_{\mathbf{C}_a}(0) > 1$ and it is a leaf; or $\rho_a > 0$ and it has a unique child of either:
        \begin{enumerate}
            \item type I $(\mathbf{C}_{a'}, 1, x_{a'} = x_a, R_{a'} = \rho_a, \rho_{a'})$.
            
            \item type II $(S_{s_{b'}}, x_{b'} = x, R_{b'} = R)$. 
        \end{enumerate}
        \item For any type II node $(S_{s_b}, x_b, R_b)$ where $S_{s_b} = (S, \mathbf{C}, \{\mathbf{C}_{\alpha}, 1, \mathbb{B}_{r_{\alpha}}(y_{\alpha})\}_{\alpha \in I_b})$ it has $\# I_b$ child nodes such that for each index $\alpha$, there exists $R_{b, \alpha}$ and $x_{b, \alpha}$ such that 
         \begin{equation*}
        \vert x_{b, \alpha} -  y_{\alpha} \vert \leq \beta R_b r_{\alpha} \qquad \text{and} \qquad \frac{1}{2} \leq \frac{R_{b, \alpha}}{R_b r_{\alpha}} \leq 1 + \beta, 
    \end{equation*}
    so that the corresponding child node is either:
    \begin{enumerate}
        \item type I $(\mathbf{C}_{a'} = \mathbf{C}_{\alpha}, 1, x_{a'} = x_{b, \alpha}, R_{a'} = R_{b, \alpha}, \rho_{a'})$.
        \item type II $(S_{s_{b'}}, x_{b'} = x_{b, \alpha}, R_{b'} = R_{b, \alpha})$.
    \end{enumerate}
    \end{itemize}
    
\end{definition}
\begin{definition}[Coarse tree representation] \label{definition: coarse tree representation of cone dec}
    The \textbf{coarse tree representation} of a cone decomposition is obtained by relabelling the rooted tree appearing in Definition \ref{definition: tree representation of cone dec}: type I nodes $(\mathbf{C}_a, 1, x_a, R_a, \rho_a)$ are simply denoted by $(\theta_{\mathbf{C}_a}(0))$, while type II nodes $(S_{s_b}, x_b, R_b)$ by $S_{s_b}$. 
\end{definition}

\begin{definition}[Closeness of tree representations] \label{definition: gamma-close tree representations of cone dec}
For $\gamma \in (0,1/100)$, two $(\theta, \beta, \mathcal{S}, N)$-tree representations with parameters
\begin{itemize}
    \item $(N_S, N_C, \{x_a\}, \{x_b\}, \{R_a\}, \{\rho_a\}, \{R_b\}, 1, \{\mathbf{C}_a\}, \{S_b\})$,
    \item $(N_S', N_C', \{x_a'\}, \{x_b'\}, \{R_a'\}, \{\rho_a'\}, \{R_b'\}, 1, \{\mathbf{C}_a'\}, \{S_b'\})$,
\end{itemize}
are said to be \textbf{$\gamma$-close} if $N_S' = N_S$, $N_C' = N_C$ and they have the same coarse tree representations, such that:

\begin{enumerate}
    \item If the corresponding two nodes are both of type I, then we have:
    \begin{itemize}
        \item $d_\mathcal{H}(\mathbf{C}_a \cap \partial \mathbb{B}_1, \mathbf{C}_a' \cap \partial \mathbb{B}_1) \leq \gamma$.
        \item If $\rho_a > 0$, then
        \begin{equation*}
              |\rho_a - \rho_a'| \leq \gamma \min(\rho_a, \rho_a'), \qquad |x_a - x_a'| \leq \gamma \min(\rho_a, \rho_a'), \qquad |R_a - R_a'| \leq \gamma \min(\rho_a, \rho_a'). 
        \end{equation*}
        Otherwise if $\rho_a = 0$, then
        \begin{equation*}
            \rho_a' = 0, \qquad |x_a - x_a'| \leq \gamma \min(R_a, R_a'), \qquad \text{and} \qquad |R_a - R_a'| \leq \gamma \min(R_a, R_a'). 
        \end{equation*}
    \end{itemize}

    \item If the corresponding two nodes are both of type II, then we have:
    \begin{equation*}
        |x_b - x_b'| \leq \gamma \min(R_b, R_b') \min_{\alpha \in I_b}(r_{\alpha}), \qquad \text{and} \qquad |R_b - R_b'| \leq \gamma \min(R_b, R_b') \min_{\alpha \in I_b}(r_{\alpha}). 
    \end{equation*}
\end{enumerate}
\end{definition} 

Since the notion of a cone decomposition is local, in describing the space of triples it will be use to provide a global decomposition of the boundary of a given isoperimetric region:

\begin{definition}[Large-scale cone decomposition] \label{definition: large scale cone dec} Given $\theta, \gamma, \beta \in \mathbb{R}_+$, $\sigma \in (0, 1/3)$, and $N \in \mathbb{N}$, let $g_0, g$ be two $C^{k,\alpha}$ metrics on $M$, $t_0, t \in \mathbb{R}$, $\Omega_0 \in \mathcal{I}(g_0, t_0)$ and $\Omega \in \mathcal{I}(g, t)$, and $\mathcal{S} = \{S_s\}_s$ be a finite collection of $(\theta, \sigma, \gamma)$-smooth models. A \textbf{large-scale $(\theta, \beta, g_0, t_0, \Omega_0, \mathcal{S}, N)$-cone decomposition} of $\Omega$ (or of $\partial \Omega$ by abuse of language) consists of:
\begin{itemize}
    \item a collection of radii $\{r_{\alpha}\}_{\alpha}$ corresponding to the singular sets $\mathrm{Sing}(\Sigma_0) = \{p_{\alpha}\}_{\alpha}$, such that $\{B^g(p_{\alpha}, r_{\alpha})\}_{\alpha}$ are pairwise disjoint.
    
    \item a $(\theta, \beta, \mathcal{S}, N)$-cone decomposition for each $|\partial \Omega| \mres B^g_{r_{\alpha}}(p_{\alpha})$.
    
    \item a $C^1$ function $u: \Sigma_0 \setminus \bigcup_{p_{\alpha} \in \mathrm{Sing}(\Sigma_0)} B^g_{\frac{r_{\alpha}}{2}}(p_{\alpha}) \to \Sigma^\perp_0$ so that for $r_0 = \min_{\alpha} \{r_{\alpha}\} > 0$,
    \[
    r_0^{-1}|u| + |\nabla u| \leq \beta,
    \]
    and $\partial \Omega \setminus B^g_{r_{\alpha}}(p_{\alpha})$ coincides with $\mathrm{graph}_{\partial \Omega_0}(u) \setminus B^g_{r_{\alpha}}(p_{\alpha})$.
\end{itemize}
\end{definition}

\begin{remark}
    While Definition \ref{def: cone decomposition} is phrased for Euclidean balls equipped with a Riemannian metric, by choosing the radius smaller than $\mathrm{inj}(M,g)$, it suffices in Definition \ref{definition: large scale cone dec} to work in a Riemannian geodesic ball of the same radius; hence we will use both notations interchangeably.
\end{remark}

Similarly, we can define the corresponding notions of tree representation and $\gamma$-closeness for these global decompositions:

\begin{definition}[Tree representation of large-scale cone decomposition] \label{definition: tree rep of large scale cone dec}
Given a large-scale \\
$(\theta, \beta, g_0, t_0, \Omega_0, \mathcal{S}, N)$-cone decomposition of $\Omega \in \mathcal{I}(g,t)$ with parameters: $\mathrm{Sing}(\Sigma_0) = \{p_{\alpha}\}_{\alpha}$, radii $\{r_{\alpha}\}_{\alpha}$, and $(\theta, \beta, \mathcal{S}, N)$-cone decompositions for each $\vert \partial \Omega \vert \mres B^g_{r_{\alpha}}(p_{\alpha})$. The corresponding \textbf{tree representation} of the large-scale cone decomposition is a rooted tree uniquely defined by:
\begin{enumerate}
    \item The root node is labelled by a tuple $(\Sigma_0, g_0, \{p_{\alpha}\}, \{r_{\alpha}\})$.
    \item The root node has $\# \mathrm{Sing}(\Sigma_0)$ children, indexed by $\alpha$. The corresponding subtree rooted at the $\alpha$-child is the tree representation of the $(\theta, \beta, \mathcal{S}, N)$-cone decomposition for each $|\partial \Omega|_g \mres B^g_{r_{\alpha}}(p_{\alpha})$. 
\end{enumerate}
Finally, the \textbf{coarse tree representation} will be the directed rooted tree with the subtrees above replaced by their corresponding coarse trees.
\end{definition} 

\begin{definition}[Closeness of tree representations of large-scale cone decompositions] \label{definition: gamma-close tree representations large scale cone dec} For $\gamma \in (0,1/100)$, two $(\theta, \beta, g_0, t_0, \Sigma_0, \mathcal{S}, N)$-tree representations of large scale cone decompositions (with $\beta \leq \gamma$) are said to be \textbf{$\gamma$-close} if their root nodes have the same label, and their subtrees corresponding to the $\alpha$-child are $\gamma$-close for each $\alpha$. 
\end{definition}

\section*{Acknowledgements}

The authors thank Nick Edelen, Martin Man-chun Li, and Luca Spolaor for encouragement and useful discussions, as well as Yangyang Li and Zhihan Wang for sharing with us an earlier version of \cite{LW25}. 

\bigskip

KMS was supported in part by the EPSRC [EP/N509577/1], [EP/T517793/1]. GN was substantially supported by research grants from the Research Grants Council of the Hong Kong Special Administrative Region, China [Project No.: CUHK 14304121 and CUHK 14305122]. Part of this work was carried out while DP was supported by an AMS-Simons travel grant. 

\begingroup

\bibliographystyle{alpha}
\bibliography{main}
\endgroup

\hrule 

\Addresses

\end{document}